\documentclass[12pt,a4paper,reqno]{amsart}
\usepackage[all]{xy} % for complicated commutative diagrams
\usepackage{amsmath}
\usepackage{amsfonts}
\usepackage{amssymb}
\usepackage{mathrsfs}

\DeclareFontFamily{U}{wncy}{} % to enable usage of cyrillic fonts
\DeclareFontShape{U}{wncy}{m}{n}{<->wncyr10}{}
\DeclareSymbolFont{mcy}{U}{wncy}{m}{n}
\DeclareMathSymbol{\Sha}{\mathord}{mcy}{"58}

\theoremstyle{plain}
\newtheorem{theorem}{Theorem}[subsection]
\newtheorem{lemma}[theorem]{Lemma}
\newtheorem{proposition}[theorem]{Proposition}
\newtheorem{corollary}[theorem]{Corollary}
\newtheorem{definition}[theorem]{Definition}

\theoremstyle{definition}
\newtheorem{remark}[theorem]{{\textit{Remark}}}

\newtheorem{Aproposition}{Proposition}[section] % For the appendix
\newtheorem{Aremark}[Aproposition]{Remark}

\newcommand{\Z}{\mathbb{Z}}
\newcommand{\Q}{\mathbb{Q}}
\newcommand{\F}{\mathbb{F}}

\newcommand{\C}{\mathbb{C}}

\newcommand{\N}{\mathcal{N}}
\renewcommand{\O}{\mathcal{O}}
\newcommand{\cG}{\mathcal{G}}
\newcommand{\cS}{\mathcal{S}}

\DeclareMathOperator{\Br}{Br}
\DeclareMathOperator{\coh}{H}

\DeclareMathOperator{\Gal}{Gal}

\DeclareMathOperator{\Pic}{Pic}

\DeclareMathOperator{\Sel}{{Sel}}

\DeclareMathOperator{\spec}{Spec}
\DeclareMathOperator{\Hom}{Hom}

\DeclareMathOperator{\Tr}{Tr}

\DeclareMathOperator{\ord}{ord}

\DeclareMathOperator{\Ker}{Ker}

\begin{document}

\title[the $\mu$-invariants]{On the $\mu$-invariants of abelian varieties over function fields of positive characteristic}

\author{King-Fai Lai}
\address{School of Mathematics and Statistics,
Henan University,
Jin Ming Avenue, Kaifeng, Henan, 475004, China}
\email{kinglaihonkon@163.com}
%\thanks{}

\author{Ignazio Longhi}
\address{Department of Mathematics, Indian Institute of Science, Bangalore - 560012, India}
\email{ignaziol@iisc.ac.in}
%\thanks{}

\author{Takashi Suzuki}
\address{Department of Mathematics, Chuo University,
1-13-27 Kasuga, Bunkyo-ku, Tokyo 112-8551, Japan}
\email{tsuzuki@gug.math.chuo-u.ac.jp}
\thanks{The third author (T. S.) is a Research Fellow of Japan Society for the Promotion of Science
and supported by JSPS KAKENHI Grant Number JP18J00415.}

\author[Tan]{Ki-Seng Tan}
\address{Department of Mathematics\\
National Taiwan University\\
Taipei 10764, Taiwan}
\email{tan@math.ntu.edu.tw}
\thanks{The fourth author (K.-S. T.) was supported in part by the Ministry of Science and Technology of Taiwan,
MOST 107-2115-M-002-010-MY2.
It is our pleasure to thank NCTS/TPE
for supporting a number of meetings of the authors in National Taiwan University.}

\author{Fabien Trihan}
\address{Sophia University,
Department of Information and Communication Sciences
7-1 Kioicho, Chiyoda-ku, Tokyo 102-8554, Japan}
\email{f-trihan-52m@sophia.ac.jp}
%\thanks{}


\begin{abstract} Let $A$ be an abelian variety over a global function field $K$ of characteristic $p$. We study the $\mu$-invariant appearing in the Iwasawa theory of $A$ over the unramified $\Z_p$-extension of $K$. Ulmer suggests that this invariant is equal to what he calls the dimension of the Tate-Shafarevich group of $A$ and that it is indeed the dimension of some canonically defined group scheme. Our first result is to verify his suggestions. He also gives a formula for the dimension of the Tate-Shafarevich group (which is now the $\mu$-invariant) in terms of other quantities including the Faltings height of $A$ and Frobenius slopes of the numerator of the Hasse-Weil $L$-function of $A / K$ assuming the conjectural Birch-Swinnerton-Dyer formula. Our next result is to prove this $\mu$-invariant formula unconditionally for Jacobians and for semistable abelian varieties. Finally, we show that the ``$\mu=0$'' locus of the moduli of isomorphism classes of minimal elliptic surfaces endowed with a section and with fixed large enough Euler characteristic is a dense open subset.
\end{abstract}

\date{October 13, 2020}
\maketitle
\tableofcontents

\section{Introduction}\label{s:int}

Let $\Gamma$ be a profinite group isomorphic to $\Z_p$ and let $\Lambda$ denote the completed group algebra $\Z_p[[\Gamma]]$. By a general structure theorem, if $M$ is a finitely generated torsion $\Lambda$-module then one has a pseudo-isomorphism (that is, a homomorphism with finite kernel and cokernel)
\begin{equation} \label{e:ps0} M\,\sim\,\bigoplus_i\Lambda/(p^{m_i})\oplus\bigoplus_j\Lambda/(f_j) \end{equation}
(where the $f_j$'s can be identified with certain irreducible polynomials after choosing a topological generator of $\Gamma$ and hence fixing an isomorphism between $\Lambda$ and the power series ring $\Z_p[[t]]$). In particular \eqref{e:ps0} makes it possible to define the number
$$\mu(M):=\sum_im_i\;,$$
which is called the $\mu$-invariant of $M$. The algebraic side of Iwasawa theory deals with modules over rings exemplified by $\Lambda$ and understanding their $\mu$-invariant has always been an interesting question. For example, in the setting above $\mu(M)=0$ means that $M$ is finitely generated as a $\Z_p$-module. More generally, let $M_{\Gamma^{p^n}}$ denote the set of coinvariants for the subgroup of index $p^n$. Then \eqref{e:ps0} implies the formula
$$\frac{\log|M_{\Gamma^{p^n}}|}{p^n\log p}=\mu(M)+o(1)\,\;\;\text{ as }n\rightarrow\infty\,,$$
provided that $|M_{\Gamma^{p^n}}|<\infty$ holds for every $n$.

The $\mu$-invariant was first studied in the following situation. Let $F$ be a number field and let $F^{cyc}$ denote its $\Z_p$-cyclotomic extension: then one takes $\Gal(F^{cyc}/F)$ as $\Gamma$ and $\Z_p[[\Gal(F^{cyc}/F)]]$ as $\Lambda$. The $\Lambda$-module to be considered is $M(F^{cyc}/F)$, the inverse limit of the $p$-parts of class groups in the intermediate extensions of $F^{cyc}/F$. Iwasawa showed that $M(F^{cyc}/F)$ is finitely generated and torsion and he conjectured the equality $\mu\big(M(F^{cyc}/F)\big)=0$. His conjecture was proved by Ferrero-Washington in the case when $F/\Q$ is an abelian extension, but it is still open in the general setting. On the other hand, Iwasawa \cite{Iwa73} also gave examples of $\Z_p$-extensions $L/F$ of number fields such that $\mu\big(M(L/F)\big)>0$.

The $\Lambda$-modules of interest in this paper are duals of Selmer groups of abelian varieties over function fields. The Iwasawa theory of abelian varieties was started by Mazur in \cite{Maz72}, where he already gave some examples of elliptic curves with non-vanishing $\mu$ over $\Q^{cyc}$ (\cite[\S10]{Maz72}). More recently Coates and Sujatha have conjectured that the $\mu$-invariant of an elliptic curve always vanishes in the cyclotomic extension of a number field if one replaces the Selmer group with the fine Selmer group (\cite[Conjecture A]{CS05}; see also \cite{Suj10} and \cite{Suj11} for an introduction to these ideas).

\subsection{Our setting: $\Z_p$-extensions of function fields} In the following, $K$ is a function field in one variable over a finite field of characteristic $p$. The $\Z_p$-extension of $K$ we are going to consider is the arithmetic one, that is, the unique $\Z_p$-extension obtained by enlarging the field of constants $\F_{q}$, which is an obvious analogue of $F^{cyc}/F$.

\subsubsection{$\mu$-invariants over function fields: class groups} In  this paper, we study the $\mu$-invariant attached to abelian varieties. Before discussing it, we review the situation for class groups, to place our results better in context. Let $h_n$ denote the cardinality of the $p$-part of the class group at the $n$-th layer of the extension. In the case of the arithmetic $\Z_p$-extension, the growth of $h_n$ can be computed from the zeta function and Weil's work on the Riemann hypothesis gives some estimation. Interpreting class groups as sets of rational points in a Jacobian, geometric considerations and some simple Galois theory yield a more precise result: see \cite{Ros73}. The outcome is that $\log(h_n)$ is proportional to $n$, for $n$ large enough: that is, the relevant $\mu$-invariant vanishes (\cite[Proposition 5.4]{LZ97}). (It might be worth to recall that similar observations and the attempt to develop an analogue theory for number fields were Iwasawa's starting point: see \cite{Tha94} for more on this function field-number field dialogue.)

Things change rather dramatically if instead one takes a geometric $\Z_p$-extension of $K$ - that is, an extension where the constant field is the same at any layer. Such extensions can ramify at infinitely many different places and in such a situation the limit of class groups is not a noetherian $\Lambda$-module (\cite[Theorem 2]{GK88}). Moreover, it turns out that in a geometric $\Z_p$-extension $\log(h_n)$ grows at least as $p^{2n}$ (\cite[Theorem 1]{GK88}).

\begin{remark} In the function field setting there is also a different kind of ``cyclotomic'' extension, obtained by adding to $K$ the $\mathfrak p^\infty$-torsion of a rank 1 Drinfeld module, for $\mathfrak p$ a place of $K$. It turns out that when $K$ is the rational function field (that is, $K=\F_q(\theta)$ ) and the cyclotomic extension is generated through the Carlitz module, an analogue of the Ferrero-Washington theorem holds (\cite[Theorem 1.3]{ABBL20}). An important difference with the situation we discussed before is that the Galois group of the extension in \cite{ABBL20} is isomorphic to an infinite product of copies of $\Z_p$: hence the corresponding Iwasawa algebra is not even noetherian. A general theory for modules over such algebras is still lacking, so it is not completely clear what is the meaning (or the general definition) of a $\mu$-invariant in this setting. The result in \cite{ABBL20} consists in proving that a certain $p$-adic $L$-functions generates the ($\chi$-part of the) Fitting ideal of a limit of class groups, and computing that the reduction of this $L$-function modulo $p$ is not zero. \end{remark}

\subsubsection{$\mu$-invariants over function fields: abelian varieties} For an abelian variety $A/K$, Iwasawa theory over the arithmetic extension of $K$ was first developed in \cite{OT09}. In particular, \cite[Theorem 1.7]{OT09} proves that the dual of the Selmer group is a finitely generated torsion $\Lambda$-module: thus one can apply \eqref{e:ps0} and define a $\mu$-invariant $\mu_{A/K}$. Moreover, \cite[Theorem 1.8]{OT09} gives some conditions for its vanishing when $A$ is isotrivial: more precisely, \begin{itemize}
\item $\mu_{A/K}=0$ always holds if, after base change by a finite extension, $A$ becomes isomorphic to a constant ordinary abelian variety;
\item in the supersingular case, $\mu_{A/K}=0$ is true if $A$ becomes constant over a finite extension of $K$ with invertible Hasse-Witt matrix. \footnote{In \cite[Theorem 1.8]{OT09}, this is given as an ``if and only if''. However, that is wrong: we discuss it in \S\ref{su:sic}.}  \end{itemize}
Finally, \cite[Theorem 1.9]{OT09} shows an application of $\mu_{A/K}=0$ to non-commutative Iwasawa theory.

We are not aware of other works discussing the $\mu$-invariants of abelian varieties over global function fields,
except the paper of Ulmer that we will discuss below.

\subsubsection{Ulmer's notion of the dimension of $\Sha$}
\label{ss:Ulmer}
Outside the context of Iwasawa theory,
Ulmer \cite[Section 4]{Ulm19} defines a certain number
called the \emph{dimension of $\Sha(A)$} and denoted as $\dim \Sha(A)$,
by looking at the asymptotic behavior of the order of the $p$-primary part
of the Tate-Shafarevich group $\Sha(A)$ of an abelian variety $A / K$
under extensions of the constant field $\F_{q}$,
assuming the finiteness of Tate-Shafarevich groups.
His aim is to study the so-called Brauer-Siegel problem for abelian varieties over function fields
using $\dim \Sha(A)$.
The relevance of his work to our context of Iwasawa theory
lies in his suggestion \cite[Remark 4.3 (5)]{Ulm19} of the potential relation
between $\dim \Sha(A)$ and the $\mu$-invariant $\mu_{A / K}$.

By the result of Kato-Trihan \cite{KT03},
the assumption of the finiteness of Tate-Shafarevich groups implies
the Birch and Swinnerton-Dyer formula for the leading coefficient
of the Hasse-Weil $L$-function of $A / K$ at $s = 1$.
Using this formula, Ulmer \cite[Proposition 4.2]{Ulm19} gives a formula for $\dim \Sha(A)$
in terms of various other quantities including the degree of the Hodge bundle of the N\'eron model of $A$
(the function field analogue of the Faltings height)
and slopes ($p$-adic valuations of the roots) of the numerator of the $L$-function of $A / K$.
From this formula, he deduces that the limit defining $\dim \Sha(A)$ exists
(as a real number a priori!)\ and is an integer.
However, there is a gap in the proof of this final integrality result,
since the corresponding integrality property of the quantity involving slopes is not properly justified there.
It only shows that $e \cdot \dim \Sha(A)$ is an integer,
where $q = p^{e}$.
We will see (and solve) this problem in the appendix of this paper.

In \cite[Remark 4.3 (2)]{Ulm19}, Ulmer justifies the terminology ``dimension of $\Sha$''
in the special case where $A / K$ is the Jacobian of a proper smooth curve over $K$.
In this case, its proper flat regular model gives a surface $\mathcal{S}$ over $\F_{q}$.
He relates the definition of $\dim \Sha(A)$
to the genuine dimension of the actual group scheme $\underline{\coh}^{2}(p^{\infty})$ over $\F_{q}$
defined by Artin \cite{Art74b} and Milne \cite{Mil76},
which is the canonical group scheme structure on the second flat cohomology of $\mathcal{S}$.
Ulmer \cite[Remark 4.3 (3), (4)]{Ulm19} then suggests a relation
between his formula for $\dim \Sha(A)$ for the Jacobian $A / K$
and Milne's formula \cite[the last equation of Section 6]{Mil75}
on Euler characteristics for the surface $\mathcal{S}$,
still assuming the finiteness of Tate-Shafarevich groups.

This relation between Ulmer's work and Artin-Milne's work
does not easily generalize to arbitrary abelian varieties over $K$,
since an arbitrary abelian variety is not necessarily a product factor
but only an isogeny factor of a Jacobian,
and the degree of the needed isogeny might be a power of $p$ that kills
the (unipotent) identity component of $\underline{\coh}^{2}(p^{\infty})$.

By \cite{Gro68}, the finiteness of $\Sha(A)$ for the Jacobian $A$ is equivalent
to the finiteness of the Brauer group of $\mathcal{S}$,
the latter of which is the Tate conjecture for divisors on $\mathcal{S}$.
This conjecture is verified for several surfaces including the ones treated in \cite{Ulm19}.
In general, it is a hard conjecture of motivic origin.

Notice that the group scheme $\underline{\coh}^{2}(p^{\infty})$ for a surface $\mathcal{S}$ and
the $\mu$-invariant $\mu_{A / K}$ for an arbitrary abelian variety $A$ are unconditionally defined.
The general philosophy here is that
in Iwasawa theory and other purely $p$-adic theories, we can do much more
``without hard motivic conjectures'' (at least in the function field case).

\subsection{Our results: contents of this paper}
We assume that the field of constants of our $K$ contains $q$ elements. The same $q$ appears in our definition of the $\mu$-invariant $\mu_{A/K}$, which is slightly different from the usual one (see Definition \ref{d:mu}). In particular, it is not a priori obvious that $\mu_{A/K}$ is always an integer: this will be a consequence of Theorem \ref{muissA}.

First, we study the behavior with respect to base change (\S\ref{sub:bc}). We check that, if $L/K$ is a finite Galois extension, we have $\mu_{A/K}\leqslant\mu_{A/L}$ (Lemma \ref{l:kl}) and we give conditions ensuring that $\mu_{A/K}=0$ implies $\mu_{A/L}=0$: in particular, we prove that this implication holds in a $p$-extension unramified outside of ordinary places (Theorem \ref{t:ocht}).

In \S\ref{su:ts}, we observe that $\mu_{A/K}$ is determined by the Tate-Shafarevich group of $A$ over the unramified $\Z_{p}$-extension. As a consequence, one can compute it from the cardinality of the $p^m$-torsion of $\Sha$ in intermediate extensions.
This is a non-trivial step: the crucial fact here is the control theorem for Selmer groups \cite[Theorem 4]{Tan10},
which shows how to obtain information about the Selmer groups for $A$ over the intermediate extensions
from the Selmer group for $A$ over the $\Z_{p}$-extension.
Assuming that the $p$-part of $\Sha$ is always finite in our $\Z_p$-tower, we deduce the limit formula  \eqref{e:shafni} for $\mu_{A/K}$ by a somewhat technical argument (coming from the difference between $\Sha$ and $\Sel$).
This shows that Ulmer's $\dim \Sha(A)$ is equal to $\mu_{A/K}$, as suggested in \cite[Remark 4.3 (5)]{Ulm19}.

In \S\ref{su:su}, following \cite{Suz19}, we define an $\F_q$-group scheme $\cG_A^{(1)}$ which represents the Tate-Shafarevich group $\Sha(A / K \bar{\F}_{q})$.%
\footnote{This representation is a bit subtle.
See the first paragraph of \S\ref{su:su} for the precise meaning.
The subtleties exist but are ``controllable'',
which is the content of the control theorem for $\Sha$,
namely Proposition \ref{p:ShaGone}.}
We prove the formula
$$\mu_{A/K}=\dim\cG_A^{(1)}$$ 
(Theorem \ref{muissA}).
With $\dim \Sha(A) = \mu_{A / K}$ (assuming the finiteness of Tate-Shafarevich groups),
this formula justifies the terminology ``dimension of $\Sha$'' for an arbitrary abelian variety $A / K$.
It also shows that $\mu_{A / K}$ (or $\dim \Sha(A)$) is an integer.
Using the group scheme $\mathcal{G}_{A}^{(1)}$
(and the related group schemes $\mathcal{G}_{A}^{(i)}$),
we prove a control theorem for $\Sha$ (Proposition \ref{p:ShaGone}),
which is not a consequence of the control theorem for $\Sel$ mentioned earlier.
This gives another more direct proof of $\mu_{A / K} = \dim \Sha(A)$.

In \eqref{e:muform}, we reinterpret the formula in \cite[Proposition 4.2]{Ulm19} on $\dim \Sha(A)$
(which assumes the finiteness of Tate-Shafarevich groups and uses the BSD formula) as a formula on $\mu_{A / K}$
and hence give an upper bound \eqref{e:upperbound} for $\mu_{A/K}$.
Since both of the sides of the resulting $\mu$-invariant formula \eqref{e:muform} are unconditionally defined,
it is then natural to ask if the formula \eqref{e:muform} can be proved unconditionally.
We will give partial results in subsequent chapters.

In Chapter \ref{s:JacSur}, we take $A$ to be the Jacobian of a curve $\cS_{K}$ over $K$. In this case, we define an $\F_q$-group scheme $\underline{\Br}$ representing the Brauer group of the associated surface $\cS$ over $\F_{q}$ using Artin-Milne's group scheme $\underline{\coh}^{2}(p^{\infty})$. We show that the group schemes $\mathcal{G}_{A}^{(1)}$ and $\underline{\Br}$ are isomorphic up to finite \'etale groups (Proposition \ref{p:ShaBr}). This implies that $\mu_{A/K} = \dim \underline{\Br}$. As a consequence, we can show that in this case the formula \eqref{e:muform} is equivalent to Milne's formula \cite[the last equation of Section 6]{Mil75} (Corollary \ref{c:MilUlm}), as suggested by Ulmer \cite[Remark 4.3 (3), (4)]{Ulm19}. This proves \eqref{e:muform} independently of the finiteness of $\Sha$. As an application, in Propositions \ref{p:trivmu} and \ref{p:aborK3}, we give necessary and sufficient conditions for $\mu_{A/K}=0$ and (under some additional assumptions) $\mu_{A/K}=1$. Proposition \ref{p:Shio} provides an explicit example of an elliptic curve with $\mu$-invariant 1.

Chapter \ref{s:ssav} deals with the case of semistable abelian varieties. Theorem \ref{t:muss} shows that formula \eqref{e:muform} holds with no condition on $\Sha$ also in this setting. The proof is based on the fact that in this case the Iwasawa Main Conjecture holds (\cite{LLTT16}) and so $\mu_{A/K}$ can be computed explicitly from the ($p$-adic) $L$-function.

Summarizing the previous three chapters,
the $\mu$-invariant formula \eqref{e:muform} is true for the following three cases:
\begin{itemize}
	\item if the Tate-Shafarevich groups are finite (by Ulmer),
	\item if $A$ is a Jacobian, or
	\item if $A$ is semistable everywhere.
\end{itemize}

In Chapter \ref{s:muelliptic}, we specialize to $A$ being an elliptic curve. As explained in Theorem \ref{t:tan}, this gives a simplified version of \eqref{e:muform} and \eqref{e:upperbound} and hence more instances of $\mu_{A/K}>0$ (Proposition \ref{p:KodOne}). In \S\ref{su:isogen}, we investigate the variation of the $\mu$-invariant in isogeny classes. If $p>2$, the Legendre form of the Weierstrass equation provides sufficient conditions for $\mu_{A/K}=0$ (Theorem \ref{t:Legendretype}).

In Chapter \ref{su:generic}, we prove that, at least for $p>3$, the vanishing of $\mu_{A/K}$ is the normal behavior in the following sense: for $n>\frac{1}{2}(g-1)$ (where $g$ is the genus of the curve corresponding to $K$), we build an irreducible variety $X(n,\bar{\mathcal C})$ parameterizing elliptic curves over $K$ with discriminant of degree $12n$ and we show that the locus determined by $\mu=0$ is open and dense (Theorem \ref{t:GenVan}).

Finally, Chapter \ref{s:verify} contains some explicit examples where we can compute the $\mu$-invariant.
With Theorem \ref{t:tan} and Magma \cite{BCP97},
we can routinely and quickly calculate the $\mu$-invariant of
(a quadratic base change of) an elliptic curve over a rational function field
of small characteristic with mild complexity.
We demonstrate this in \S \ref{su:high} and \S \ref{su:nonss} among other examples.

The paper ends with an appendix explaining the gap in the literature mentioned in Section \ref{ss:Ulmer}
and proving a certain integrality property
of the exact $p$-power factor of the $L$-function of $A$.
This shows that the slope term in the formula \eqref{e:muform} is an integer.
This is not a consequence of the result $\mu_{A / K} \in \Z$ in Theorem \ref{muissA}
since \eqref{e:muform} has not been verified unconditionally for all abelian varieties.
These two types of integrality should be considered as separate results.

\subsection*{Acknowledgments}
The authors are grateful to
Luc Illusie, Bruno Kahn, Masato Kurihara, Kentaro Mitsui, Yukiyoshi Nakkajima,
Tadashi Ochiai, Atsushi Shiho, Douglas Ulmer and Christian Wuthrich
for helpful discussions.
Many thanks also to the referee for the careful reading and the comments, which helped to significantly improve the paper.

\subsection{Notation}\label{su:notation}
Let $p>0$ be a prime number. Let $K$ be a function field in one variable with field of constants $\F_q$, $q=p^e$. Write $\F_{q,n}$, $\F_{q,\infty}$, $K_n^{(p)}$ and $K_\infty^{(p)}$, for
$\F_{q^{p^n}}$, $\bigcup_n \F_{q,n}$, $K\F_{q,n}$ and $K\F_{q,\infty}$ so that $K_n^{(p)}$ is the $n$th layer of $K_\infty^{(p)}/K$.
Denote $\Gamma:=\Gal(K_\infty^{(p)}/K)$ and $\Lambda:=\Z_p[[\Gamma]]$.

The cohomology groups will be usually Galois or \'etale cohomology,
except for cohomology of finite group schemes where we use the flat cohomology, denoted $\coh^*_{fl}$ (see \cite[II, \S 1 and III, Definition 1.5]{Mil80} for its definition and \cite[III]{Mil06} for the facts needed in this paper).

Let $\mathcal C$ be the smooth projective curve over $\F_q$ having $K$ as its function field. Write $g_{\mathcal C}$ for
the genus of $\mathcal C$.  Let $A/K$ be an abelian variety.
For an abelian group $H$, we denote its $p^n$-torsion subgroup by $H[p^n]$,
including $n = \infty$ (so that $H[p^{\infty}]$ means the $p$-primary torsion part).
However we let $A_{p^n}$ denote the kernel of the multiplication by $p^n$ on $A$
and put $A_{p^\infty}:=\bigcup_n A_{p^n}$.
For a field extension $L/K$, denote the $p^{n}$-Selmer group and the $p^{n}$-Tate-Shafarevich groups ($n=\infty$ included) by
$$\Sel_{p^n}(A/L):=\ker (\coh^1_{fl}(L,A_{p^n})\longrightarrow \bigoplus_w \coh^1(L_w, A)),$$
$$\Sha_{p^n}(A/L):=\ker (\coh^1(L,A)[p^{n}]\longrightarrow \bigoplus_w \coh^1(L_w, A)),$$
where $w$ runs through all places of $L$.
By dropping ``$[p^{n}]$'' in the latter equation,
we obtain the usual definition of the (full) Tate-Shafarevich group $\Sha(A / L)$.

Let ${}^\vee$ denote the Pontryagin dual.

We will denote by $\mathcal{A}$ the N$\acute{\text{e}}$ron model of $A$.
The $K / \F_{q}$-trace of $A$ (\cite{Con06}) is denoted by $\mathrm{Tr}_{K / \F_{q}}(A)$.

\section{The $\mu$-invariant for abelian varieties}\label{s:gen}
Let $A/K$ be an abelian variety and let $X_{A/K}$ be the Pontryagin dual of the Selmer group $\Sel_{p^\infty}(A/K_\infty^{(p)})$.
Then $X_{A/K}$ is finitely generated and torsion over $\Lambda$ (\cite[Theorem 1.7]{OT09}).
From a general structure theorem for such modules, there is an exact sequence
\begin{equation}\label{e:x}
0\longrightarrow \bigoplus_{i=1}^k \Lambda/(p^{\mu_i}) \oplus \bigoplus_{j=1}^l \Lambda/(g_j)\longrightarrow X_{A/K}
\longrightarrow F\rightarrow 0,
\end{equation}
where $F$ is finite and each $g_j$ corresponds to a power of an irreducible distinguished polynomial in the isomorphism $\Lambda\simeq\Z_p[[t]]$.
The following definition is different from the
convention, but it suits us well (see Theorem \ref{muissA}).
\begin{definition} \label{d:mu}
We define the $\mu$-invariant of $X_{A/K}$ as the non-negative rational number $\mu_{A/K}$ such that 
\begin{equation} \label{e:dfmu} q^{\mu_{A/K}}=\prod_{i=1}^k p^{\mu_i}\,. \end{equation}
\end{definition}

Readers should note that on the left-hand side of \eqref{e:dfmu} there is a power of $q$ but on the right-hand side there is a product of powers of $p$.
We will see in Theorem \ref{muissA} that $\mu_{A / K}$ is actually an integer.

The above is not specific to $X_{A / K}$.
For any finitely generated torsion $\Lambda$-module $X$ appearing in this paper,
we define its $\mu$-invariant by \eqref{e:x} and \eqref{e:dfmu}.
It is an integer divided by $e$ in general, where $q = p^{e}$.

\subsection{Base changes}\label{sub:bc}
The snake lemma applied to the multiplication by $p$ on \eqref{e:x} yields the exact sequence
$$F[p]\longrightarrow \bigoplus_{i=1}^k\Lambda/(p) \oplus \bigoplus_{j=1}^l\Lambda/(g_j,p)   \longrightarrow X_{A/K}/pX_{A/K}\longrightarrow F/pF.$$

\begin{lemma}\label{l:mu=0}
$\mu_{A/K}=0$ if and only if $\Sel_p(A/K_\infty^{(p)})$ is finite.
\end{lemma}
\begin{proof}
Because each $\Lambda/(g_j,p)$ is finite, the lemma follows via duality from the above exact sequence.
\end{proof}

\begin{lemma}\label{l:kl}
If $L/K$ is a finite Galois extension with $\mu_{A/L}=0$, then $\mu_{A/K}=0$.
\end{lemma}
\begin{proof}
The kernel of the restriction map
$\Sel_p(A/K_\infty^{(p)})\longrightarrow \Sel_p(A/L_\infty^{(p)})$ is finite
since it is a subgroup of the finite group
$\coh^1(L_\infty^{(p)}/K_\infty^{(p)}, A_p(L_\infty^{(p)}))$.
By Lemma \ref{l:mu=0}, we obtain the result.
\end{proof}

For a Galois extension $L/K$ with $G:=\Gal(L_\infty^{(p)}/K_\infty^{(p)})$, denote
$$\N_{L/K}:=\ker\,[\coh^1_{fl}(K_\infty^{(p)},A_p)\longrightarrow \bigoplus_w \coh^1(K_{\infty,w}^{(p)},A)/\coh^1(G_w,A(L_{\infty,w}^{(p)}))],$$
where $w$ runs through all places of $K_\infty^{(p)}$.

\begin{lemma}\label{l:pext} Let $L/K$ be a finite Galois $p$-extension.
Then $\mu_{A/L}=0$, if and only $\N_{L/K}$ is finite.
\end{lemma}
\begin{proof} For simplicity write $M$ for $\Sel_p(A/L)$ %regarded as a discrete group
and let $M^\vee$ be the Pontryagin dual.
Let $I$ be the augmentation ideal of $\F_p[G]$. If $M$ is infinite, then\footnote{This should be well known for more general context. Here is a simple proof for our need. As in the commutative case, we  need to show that $M/IM=0$ implies $M=0$. But this is obvious, since $G$ is a $p$-group, $I^m=0$ for some $m$.} so is $M^\vee/IM^\vee$. Hence
by duality $M^G$ is also infinite. The exact sequence
$$\xymatrix{0\ar[r] & \N_{L/K}\cap\coh^1(G,A_p(L)) \ar[r] & \N_{L/K} \ar[r] & M^G \ar[r] & \coh^2(G,A_p(L))}$$
implies $\N_{L/K}$ is infinite. Conversely, if $\N_{L/K}$ is infinite, then so is $M^G$.

\end{proof}

\begin{lemma}\label{l:local}
Let $L/K$ be a finite Galois $p$-extension with $G:=\Gal(L_\infty^{(p)}/K_\infty^{(p)})$. Then $\N_{L/K}$ is finite if and only if $\mu_K=0$ and
$\bigoplus_w\coh^1(G_w,A(L_{\infty,w}^{(p)}))$ is finite.
\end{lemma}
\begin{proof} Write $W$ for $\bigoplus_w\coh^1(G_w,A(L_{\infty,w}^{(p)}))$, $W$ is finite if and only if $W[p]$
is finite. Consider the commutative diagram
$$\xymatrix{ & \coh^1(K_\infty^{(p)},A) \ar[r]^-{\mathscr{L}} & \bigoplus_w \coh^1(K_{\infty,w}^{(p)},A)\\
\Sel_p(A/K_\infty^{(p)}) \ar@{^{(}->}[r] & \N_{L/K} \ar[r]^-{\mathscr{L}_p} \ar[u] & W[p]\ar[u],}$$
where $\mathscr{L}$ and $\mathscr{L}_p$ are localization maps.
By using the (generalized) Cassels-Tate exact sequence \cite{GAT07}, one can deduce (see \cite[Proposition 4.2]{Tan14})
that $\mathscr L$ has co-kernel of finite co-rank. Hence the co-kernel of
$\mathscr{L}_p$ is finite.

\end{proof}

\begin{theorem}\label{t:ocht}
Let $L/K$ be a finite Galois $p$-extension unramified outside ordinary places of $A/K$. If $\mu_{A/K}=0$, then $\mu_{A/L}=0$.
\end{theorem}
\begin{proof}One can follow the proof of \cite[Theorem 1.9]{OT09}, or apply Lemmas \ref{l:kl}, \ref{l:pext}, \ref{l:local} and show
that  $W:=\bigoplus_w\coh^1(G_w,A(L_{\infty,w}^{(p)}))$ has finite $p$-rank. If $L/K$ is unramified at $w$, then $\coh^1(G_w,A(L_{\infty,w}^{(p)}))$ is finite, and is trivial for good places (see \cite[Proposition I.3.8]{Mil06}).
We may assume that $L/K$ is cyclic of degree $p$,
so that by local class field theory, $L_v/K_v$ is an intermediate extension of some $\Z_p$-extension of $K_v$.
By \cite[Theorem 2(c)]{Tan10} and the proof of Theorem 4, {\em{op.\ cit.}}, if $A$ has good ordinary or split multiplicative reduction
at $w$, then $\coh^1(G_w,A(L_{\infty,w}^{(p)}))$ is finite. If $w$ is a non-split multiplicative place, then there is a constant field extension $K'/K$
such that $A/K'_w$ is split multiplicative. Let $L'=K'L$. Then the kernel of the restriction map
$$\coh^1(G_w, A(L_{\infty,w}^{(p)})) \longrightarrow \coh^1((L')_{\infty,w}^{(p)}/(K')_{\infty,w}^{(p)}, A((L')_{\infty,w}^{(p)})),$$
where the target is finite, is contained in $\coh^1((K')_{\infty,w}^{(p)}/K_{\infty,w}^{(p)}, A((K')_{\infty,w}^{(p)}))$ which is finite,
since $K'/K$ is unramified.

\end{proof}

\subsection{The $L$-function} \label{s:Lfun}
We collect here mostly well-known facts about the Hasse-Weil $L$-function of $A$
and define several invariants related to it
as preparations for subsequent sections.

Let $L_{A}(s)$ be the (completed) Hasse-Weil $L$-function of $A$.
See for example \cite[\S 1]{Sch82} and \cite[V, \S 6]{Kah18a}
for lists of properties of $L_{A}(s)$.
We can write
	\[
			L_{A}(s)
		=
			\frac{P_{1}(t)}{P_{0}(t) P_{2}(t)},
	\]
where $P_{i}(t) \in \Z[t]$ is a polynomial in $t = q^{-s}$ with constant term $1$
whose reciprocal roots (with multiplicities) are Weil $q$-numbers $\{\alpha_{ij}\}$ of weight $i + 1$
(\cite[V, 6.8.2]{Kah18a}).

\begin{proposition}
	The function $P_{0}(t)^{-1}$ (with $t = q^{-s}$) is
	the zeta function of $\mathrm{Tr}_{K / \F_{q}}(A)$,
	and $P_{2}(t) = P_{0}(q t)$
\end{proposition}

\begin{proof}
	By \cite[Satz 1]{Sch82},
	$P_{0}(t / q)$ is the characteristic polynomial of $q$-th power Frobenius on
	$\mathrm{V}_{l}(\mathcal{A}(\mathcal{C} \times_{\F_{q}} \bar{\F}_{q}))$,
	where $l$ is a prime different from $p$
	and $\mathrm{V}_{l}$ denotes the $l$-adic Tate module tensored with $\Q_{l}$.
	We have $\mathcal{A}(\mathcal{C} \times_{\F_{q}} \bar{\F}_{q}) = A(K \bar{\F}_{q})$,
	whose quotient by the subgroup $\mathrm{Tr}_{K / \F_{q}}(A)(\bar{\F}_{q})$
	is a finitely generated abelian group by the Lang-N\'eron theorem
	(\cite[Theorem 7.1]{Con06}).
	Hence $\mathrm{V}_{l} A(K \bar{\F}_{q}) \cong \mathrm{V}_{l} \mathrm{Tr}_{K / \F_{q}}(A)(\bar{\F}_{q})$.
	This implies the statement for $P_{0}(t)$.
	The functional equation (\cite[V, 6.8.1, 6.8.2]{Kah18a}) shows that
	$L_{A}(2 - s)$ is some exponential function times $L_{A}(s)$.
	Hence the weight argument shows that $P_{2}(t) = P_0(q t)$.
\end{proof}

Let $\lambda_{j} \in \Q$ be the $q$-valuation of $\alpha_{1 j}$
(where the $q$-valuation means the $p$-adic valuation
such that the valuation of $q$ is $1$).

\begin{definition} \label{d:theta}
	We define $\theta = \theta_{A}$ to be the non-negative rational number
	such that $q^{\theta} P_{1}(t / q)$ is $p$-primitive
	(i.e.\ all the coefficients are $p$-adically integral
	and some coefficients are $p$-adic units).
	
	We also define $a = a_{A} := \deg(P_{1})$.
\end{definition}

\begin{proposition} \label{p:slope}
	We have
		\[
				\theta
			=
				\sum_{\lambda_{j} < 1} (1 - \lambda_{j})
			\le
				\frac{a}{2}.
		\]
\end{proposition}

\begin{proof}
	The equality can be obtained by a direct calculation or by noting that both sides are the absolute value of the height of the lowest point
	of the Newton polygon of the polynomial $P_{1}(t / q)$ with respect to the $q$-valuation.
	
	The functional equation $L_{A}(s) \leftrightarrow L_{A}(2 - s)$ (\cite[V, 6.8.1, 6.8.2]{Kah18a})
	implies a functional equation $P_{1}(t) \leftrightarrow P_{1}(q^{2} / t)$
	since $\{\alpha_{i j}\}$ are Weil numbers.
	Hence the set $\{\alpha_{1 j}\}$ (with multiplicities) is equal to
	$\{q^{2} / \alpha_{1 j}\}$,
	so $\{\lambda_{j}\} = \{2 - \lambda_{j}\}$.
	We have $\lambda_{j} \ge 0$ since $\alpha_{i j}$ are algebraic integers.
	Therefore
		\[
				\sum_{\lambda_{j} < 1} (1 - \lambda_{j})
			\le
				\sum_{\lambda_{j} < 1} 1
			\le
				\frac{1}{2} \sum_{\text{all } \lambda_{j}} 1
			=
				\frac{\deg(P_{1})}{2}.
		\]
\end{proof}

Let $N = N_{A / K} = \sum_{v} \varepsilon_{v}^{l} \cdot v$ be the global conductor of $A / K$,
which is an effective divisor on $\mathcal{C}$
(\cite[II, \S 3]{Ray95}).
By the Ogg-Shafarevich formula \cite[II, Th\'eor\`eme 3 (ii)]{Ray95}, we have
	\begin{equation} \label{e:OS}
			a
		=
			\deg(N) + 4 \dim(A) (g_{\mathcal{C}} - 1) + 4 \dim \mathrm{Tr}_{K / \F_{q}}(A).
	\end{equation}

\begin{remark}
	The number $\theta$ is actually an integer
	as we will see in the appendix.
	This fact will not be used below.
\end{remark}

\subsection{The Tate-Shafarevich groups}\label{su:ts}
For a finitely generated torsion Iwasawa module,
if its Pontryagin dual is $p$-divisible, then its $\mu$-invariant is zero.
Therefore, the proposition below follows from the exact sequence
$$0\longrightarrow A(K_\infty^{(p)})\otimes \Q_p/\Z_p\longrightarrow \Sel_{p^\infty}(A/K_\infty^{(p)})\longrightarrow \Sha_{p^\infty}(A/K_\infty^{(p)})\longrightarrow 0.$$

\begin{proposition}\label{p:ts}
The $\mu$-invariant of $\Sha_{p^\infty}(A/K_\infty^{(p)})^\vee$ equals $\mu_{A/K}$.
\end{proposition}

Next, using the control theorem \cite[Theorem 4]{Tan10} for Selmer groups,
we will give a formula for $\mu_{A / K}$ in terms of the asymptotic behavior
of Tate-Shafarevich groups for finite levels $A / K_{n}^{(p)}$,
as suggested by Ulmer \cite[Remark 4.3 (5)]{Ulm19},
with a variant that does not assume the finiteness of Tate-Shafarevich groups.

\begin{proposition}\label{p:limint} We have
\begin{equation}\label{e:sha}
\mu_{A/K}=\lim_{n\rightarrow \infty}\frac{\log |\Sha_{p^m}(A/K_n^{(p)})|}{{p^n}\log q}
\end{equation}
for all large $m$.
\end{proposition}
\begin{proof}
Let $d_j$ denote the degree of the polynomial $g_j$ in \eqref{e:x}.
 Let $\sigma\in \Gal(K_\infty^{(p)}/K)$ be a topological generator.
Since $|\Lambda/(p^m, g_j)|=p^{md_j}$ and $|\Lambda/(p^m, \sigma^{p^n}-1)|=p^{mp^n}$, for every fixed $m$ one has
$$\frac{\log |\Sel_{p^m}(A/K_\infty^{(p)})^{\Gamma^{p^n}}|}{{p^n}\log p}=\sum_i \lambda_i+\mathrm{o}(1),\; \text{as}\; n\rightarrow\infty,$$
with $\lambda_i:=\mathrm{min}\{m,\mu_i\}$. By \cite[Theorem 4]{Tan10}, we have the exact sequence 
$$\xymatrix{0\ar[r] & \mathrm{Ker}_n\ar[r] & \Sel_{p^\infty}(A/K_n^{(p)}) \ar[r] & \Sel_{p^\infty}(A/K_\infty^{(p)})^{\Gamma^{p^n}}
\ar[r] & \mathrm{Coker}_n\ar[r] & 0}$$
with $|\mathrm{Ker}_n|$ and $|\mathrm{Coker}_n|$ bounded as $n$ varies. For the time being, we call this the control sequence
of Selmer groups. It implies
$$\log |\Sel_{p^m}(A/K_\infty^{(p)})^{\Gamma^{p^n}}|=\log |\Sel_{p^m}(A/K_n^{(p)})| +\mathrm{O}(1),\;\text{as}\;n,m\;\text{vary}.$$
The rank of of $A(K_n^{(p)})$ is bounded. Since $\frac{\log p^m}{{p^n}\log p}=\mathrm{o}(1)$, the exact sequence
$$0\longrightarrow A(K_n^{(p)})\otimes \Q_p/\Z_p\longrightarrow \Sel_{p^\infty}(A/K_n^{(p)})\longrightarrow \Sha_{p^\infty}(A/K_n^{(p)})\longrightarrow 0$$
leads to
$$\frac{\log |\Sha_{p^m}(A/K_n^{(p)})|}{{p^n}\log p}=\sum_i \lambda_i+\mathrm{o}(1),\; \text{as}\; n\rightarrow\infty.$$
If $m\geq \mu_i$ for all $i$, then the above tends to $\sum_i \mu_i$, as $n\rightarrow\infty$. 

\end{proof}
Denote
$\sigma_n:=\sigma^{p^n}$, $\omega_n=\sigma_n-1$ and write $\omega_n=\nu_0\cdot \nu_1\cdot\cdots\nu_n$, with
$\nu_0=\omega_0$ and $\nu_i=1+\sigma_{i-1}+\sigma_{i-1}^2+\cdots+\sigma_{i-1}^{p-1}$, for $i>0$.
Then $\nu_0,\nu_1,\dots,\nu_n$ are relatively prime irreducible elements in $\Lambda$.
\begin{lemma}\label{l:serre}
Let $g_j$ be as in \eqref{e:x} and set $V=\Lambda/(g_j)$. If $g_j$ is relatively prime to all $\omega_n$, then $V/\omega_nV$ is finite of order
$p^{n d_j+\mathrm{O}(1)}$. If $g_j$ is a power of some $\nu_i$ and $\omega'_n=\omega_n/\nu_i$, for $n\geq i$,
then $V/\omega'_nV$ is finite of order
$p^{n d_j+\mathrm{O}(1)}$.
\end{lemma}
\begin{proof} Since $V$ is a finite free $\Z_p$-module, if $\varphi$ is a $\Z_p$-endomorphism on $V$, then the order
of $V/\varphi(V)$ equals $|\det \varphi|_p^{-1}$, where $|\;\;|_p$ denote the normalized $p$-adic absolute value.
The proof of \cite[Theorem 8]{Ser95}, Case iii, shows $|\det \nu_l|_p^{-1}=p^{d_j}$, for $l$ large enough
(where Serre denoted our $\nu_{l}$ by $\nu_{l - 1}'$).
\end{proof}

\begin{proposition}\label{p:liminf}
If $\Sha_{p^\infty}(A/K_n^{(p)})$ is finite for all $n\in\mathbb N$, then 
\begin{equation}\label{e:shafni}
\mu_{A/K}=\lim_{n\rightarrow \infty}\frac{\log |\Sha_{p^\infty}(A/K_n^{(p)})|}{{p^n}\log q}.
\end{equation}
The right-hand side is equal to the ``\emph{dimension of $\Sha(A)$}'' denoted as $\dim \Sha(A)$
in the sense of Ulmer \cite[Proposition/Definition 4.1]{Ulm19}.
\end{proposition}
\begin{proof}
Since $\Sha_{p^\infty}(A/K_n^{(p)})$ is the $p$ co-torsion part of $\Sel_{p^\infty}(A/K_n^{(p)})$, via the aforementioned control sequence,
we can estimate its order by that of the $p$ co-torsion part of $\Sel_{p^\infty}(A/K_\infty^{(p)})^{\Gamma^{p^n}}$, or by duality,
the order of the $p$-primary torsion part of $X_{A/K}/\omega_n X_{A/K}$. In view of \eqref{e:x}, we need to deal with the elementary $\Lambda$-module appearing in \eqref{e:x}. If $V=\Lambda/(p^{\mu_i})$, then $V/\omega_n V$ has oder $p^{\mu_i p^n}$.
In the case where $V=\Lambda/(g_j)$, $g_j=\nu_i^{l}$, the exact sequence
	\[
			0
		\to
			\nu_i\Lambda/(\nu_i^{l},\omega_n)
		\to
			V/\omega_n V
		\to
			\Lambda/\nu_i\Lambda
		\to
			0
	\]
shows $\Lambda/\nu_i\Lambda$ is the $\Z_p$-free quotient of $V/\omega_n V$, while since $\nu_i^{l-1}$ and $\omega'_n$ 
are relatively prime, $\nu_i\Lambda/(\nu_i^{l},\omega_n)\simeq \Lambda/ (\nu_i^{l-1},\omega'_n)$ is the $p$-primary torsion part.
Thus, Lemma \ref{l:serre} says, for every $g_j$ in \eqref{e:x}, with $V=\Lambda/(g_j)$, the order of the $p$-primary torsion part of
$V/\omega_n V$ is bounded by $p^{n d_j+\mathrm{O}(1)}$. These lead to the desired
$$\frac{\log |\Sha_{p^\infty}(A/K_n^{(p)})|}{{p^n}\log p}=\sum_i\mu_i+\mathrm{o}(1), \; \text{as}\; n\rightarrow\infty.$$
\end{proof}

\begin{remark} \mbox{}
	\begin{enumerate}
		\item
			Without the assumption of the finiteness of Tate-Shafarevich group,
			the above proof shows that, if we replace  $\Sha_{p^\infty}(A/K_n^{(p)})$
			by its $p$ co-torsion part, then \eqref{e:shafni} also holds.
		\item
			The proof of Proposition \ref{p:liminf} is complicated
			due to the lack of a control theorem for $\Sha$ at this point.
			It is not clear how to deduce such a theorem from the (known) control theorem for $\Sel$.
			The difficulties mainly come from the fact that
			$H^{1}(\Gamma^{p^{n}}, A(K_{\infty}^{(p)}) \otimes \Q_{p} / \Z_{p})$ is not finite in general.
			In the next subsection, we will prove a control theorem for $\Sha$
			without using a control theorem for $\Sel$,
			from which Proposition \ref{p:liminf} will follow.
	\end{enumerate}
\end{remark}

In the next subsection, we will use $\Sha_{p^{\infty}}(A / K \bar{\F}_{q})$
in addition to $\Sha_{p^\infty}(A / K_{\infty}^{(p)})$.
Their relation is given as follows.

\begin{proposition} \label{p:inv}
	The $\Gal(\bar{\F}_{q} / \F_{q, \infty})$-invariant part of
	$\Sha_{p^{\infty}}(A / K \bar{\F}_{q})$ is
	$\Sha_{p^\infty}(A / K_{\infty}^{(p)})$.
\end{proposition}

\begin{proof}
	 The group $\Gal(\bar{\F}_{q} / \F_{q, \infty})$ is pro-prime-to-$p$.
	 Hence the $\Gal(\bar{\F}_{q} / \F_{q, \infty})$-invariant part functor
	 on $p$-primary torsion discrete $\Gal(\bar{\F}_{q} / \F_{q, \infty})$-modules is an exact functor.
	 This implies that the $\Gal(\bar{\F}_{q} / \F_{q, \infty})$-invariant part of
	 $\coh^{1}(K \bar{\F}_{q}, A)[p^{\infty}]$ is
	 $\coh^{1}(K_{\infty}^{(p)}, A)[p^{\infty}]$.
	 On the other hand, we have
	 	\[
	 			\bigoplus_{\bar{v}}
	 				\coh^{1}((K \bar{\F}_{q})_{\bar{v}}, A)
	 		\cong
	 			\bigoplus_{w}
	 				\coh^{1}(K_{\infty, w}^{(p)} \otimes_{\F_{q, \infty}} \bar{\F}_{q}, A),
	 	\]
	 where $\bar{v}$ and $w$ run through all places of
	 $K \bar{\F}_{q}$ and $K_{\infty}^{(p)}$, respectively.
	 (Here note that $K_{\infty, w}^{(p)} \otimes_{\F_{q, \infty}} \bar{\F}_{q}$ is
	 not a field but a finite product of fields).
	 The group $\Gal(\bar{\F}_{q} / \F_{q, \infty})$ acts on the right-hand side term by term.
	 The $\Gal(\bar{\F}_{q} / \F_{q, \infty})$-invariant part of
	 $\coh^{1}(K_{\infty, w}^{(p)} \otimes_{\F_{q, \infty}} \bar{\F}_{q}, A)[p^{\infty}]$ is
	 $\coh^{1}(K_{\infty, w}^{(p)}, A)[p^{\infty}]$ by the same reasoning as before.
	 Now the proposition follows from the definition of Tate-Shafarevich groups.
\end{proof}

\subsection{The Tate-Shafarevich group scheme}\label{su:su}

In \cite{Suz19}, a commutative perfect group scheme
$$\mathcal{G}^{(i)} = \mathcal{G}_{A}^{(i)}:={\bf H}^i(\mathcal{C},\mathcal{A})$$
over $\F_{q}$ for each integer $i$ is defined
(where perfect means that the Frobenius is an isomorphism).
It is characterized by the property \cite[Proposition 2.7.8]{Suz19}
that for any perfect field extension $k/\F_q$,
we have a canonical isomorphism
	\[
			\mathcal{G}_{A}^{(i)}(k)
		\cong
			\coh^i(\mathcal C_{\bar{k}}, \mathcal{A})^{\Gal({\bar k}/k)}
	\]
functorial in the variable $k$,
where $\mathcal{C}_{\bar{k}} = \mathcal{C} \times_{\F_{q}} \bar{k}$ denotes the base change
and the cohomology is taken in the \'etale topology
(or, equivalently, in the flat topology, as the N\'eron model $\mathcal{A}$ is smooth).
We call $\mathcal{G}_{A}^{(1)}$ the Tate-Shafarevich scheme, because
	\[
			\mathcal{G}_{A}^{(1)}(\bar{\F}_{q})
		=
			\coh^{1}(\mathcal{C}_{\bar{\F}_{q}}, \mathcal{A})
		\cong
			\Sha(A / K \bar{\F}_{q}),
	\]
where the last canonical isomorphism is \cite[III, Lemma 11.5]{Mil06}.
Taking the $\Gal(\bar{\F}_{q} / \F_{q, \infty})$-invariant parts of the $p$-primary torsion parts of both of the sides
and using Proposition \ref{p:inv}, we have
\begin{equation} \label{e:Gsha}
			\mathcal{G}_{A}^{(1)}(\F_{q,\infty})[p^\infty]
		\cong
			\Sha_{p^\infty}(A / K_{\infty}^{(p)}).
	\end{equation}
Therefore,
\begin{equation}\label{e:gsh}
(\mathcal{G}_{A}^{(1)}(\F_{q,\infty})[p^\infty])^\vee
		\cong
			\Sha_{p^\infty}(A / K_{\infty}^{(p)})^\vee.
\end{equation}

The group $\mathcal{G}_{A}^{(1)}$ is the perfection
(inverse limit along Frobenius morphisms)
of a commutative smooth group scheme over $\F_{q}$
with unipotent identity component and torsion cofinite component group $\pi_{0}(\mathcal{G}_{A}^{(1)})$
(\cite[Theorem 3.4.1 (2)]{Suz19}).
Here cofinite means that
$\pi_{0}(\mathcal{G}_{A}^{(1)})[n]$ is finite \'etale for any $n \ge 1$.

\begin{definition}
Define $s_{A / K}: = \dim  \mathcal{G}_{A}^{(1)}$.
\end{definition}

\begin{proposition} \label{p:Gchar}
	Let $\mathcal{G}$ be a commutative smooth group scheme over $\F_{q}$
	such that its identity component $\mathcal{G}^{0}$ is unipotent
	and its component group $\pi_{0}(\mathcal{G})$ is $p$-primary torsion and cofinite.
	Then the Iwasawa module $\mathcal{G}(\F_{q, \infty})^{\vee}$ is
	finitely generated torsion over $\Lambda$,
	with characteristic ideal generated by
	$q^{\dim(\mathcal{G})}$ times the characteristic polynomial of the Frobenius action
	on the finite-dimensional $\Q_{p}$-vector space
	$\pi_{0}(\mathcal{G})(\F_{q, \infty})^{\vee} \otimes_{\Z_{p}} \Q_{p}$.
	In particular, the $\mu$-invariant of $\mathcal{G}(\F_{q,\infty})^{\vee}$ is $\dim(\mathcal{G})$.
\end{proposition}

\begin{proof}
	Consider the connected-\'etale sequence
	$0 \to \mathcal{G}^{0} \to \mathcal{G} \to \pi_{0}(\mathcal{G}) \to 0$.
	This induces an exact sequence
		\[
				0
			\to
				\mathcal{G}^{0}(\F_{q,\infty})
			\to
				\mathcal{G}(\F_{q,\infty})
			\to
				\pi_{0}(\mathcal{G})(\F_{q, \infty})
			\to
				0
		\]
	since $\coh^{1}(\F_{q, \infty}, \mathcal{G}^{0})$
	(which is the direct limit of $\coh^{1}(\F_{q, n}, \mathcal{G}^{0})$ in $n$)
	vanishes by Lang's theorem.%
	\footnote{Recall the statement of Lang's theorem
	\cite[III, \S 5, 7.5 and 7.11]{DG70}:
	for any commutative connected (not necessarily affine) algebraic group $G$ over a finite field $\F_{q}$,
	the cohomology $\coh^{n}(\F_{q}, G)$ vanishes for all $n \ge 1$.}
	Hence we have an exact sequence
		\[
				0
			\to
				\pi_{0}(\mathcal{G})(\F_{q,\infty})^{\vee}
			\to
				\mathcal{G}(\F_{q,\infty})^{\vee}
			\to
				\mathcal{G}^{0}(\F_{q,\infty})^{\vee}
			\to
				0
		\]
	of $\Lambda$-modules.
	Since characteristic ideals are multiplicative in short exact sequences
	(\cite[Appendix A, Proposition 1]{CS06}),
	we may assume that $\mathcal{G}$ is either connected or \'etale.
	A commutative smooth connected unipotent group over $\F_{q}$ is
	a finite successive extension of copies of $\mathbf{G}_{a}$
	(\cite[IV, \S 2, 3.9]{DG70}).
	Therefore the connected case is further reduced to the case of $\mathbf{G}_{a}$
	by a similar argument.
	
	Assume that $\mathcal{G} = \mathbf{G}_{a}$.
	Then $(\F_{q,\infty})^{\vee}$ is a rank one free $\F_{q}[[\Gamma]]$-module.
	In particular, it is a finitely generated torsion $\Lambda$-module
	with characteristic ideal $(q)$.
	Assume that $\mathcal{G}$ is \'etale.
	Then $\mathcal{G}(\F_{q,\infty})^{\vee}$ is finitely generated over $\Z_{p}$.
	Hence it is a finitely generated torsion $\Lambda$-module
	with characteristic ideal generated by
	the characteristic polynomial of the Frobenius action
	on the finite-dimensional $\Q_{p}$-vector space
	$\mathcal{G}(\F_{q,\infty})^{\vee} \otimes_{\Z_{p}} \Q_{p}$.
\end{proof}

\begin{theorem}\label{muissA} The $\mu$-invariant of $(\mathcal{G}^{(1)}(\F_{q,\infty})[p^\infty])^{\vee}$ is $s_{A / K}$.
Therefore,
$$\mu_{A/K}=s_{A / K}.$$
In particular, $\mu_{A / K}$ (and hence $\dim \Sha(A)$ when Tate-Shafarevich groups are finite) is an integer.
\end{theorem}

\begin{proof}
	This follows from \eqref{e:gsh}, Proposition \ref{p:ts}, and Proposition \ref{p:Gchar}.
\end{proof}

In the proof of Propositions \ref{p:limint} and \ref{p:liminf},
we have used the control theorem for Selmer groups \cite[Theorem 4]{Tan10}.
The group schemes $\mathcal{G}^{(i)}$, their derived descent property and arithmetic duality give
a control theorem for Tate-Shafarevich groups,
which we are now going to see below (Proposition \ref{p:ShaGone}).
For $n\ge 0$, consider the natural homomorphisms:
\[
		\Sha(A / K_{n}^{(p)})
	\overset{i_{n}}{\longrightarrow}
		\coh^{1}(\mathcal{C}_{\F_{q, n}}, \mathcal{A})
	\overset{j_{n}}{\longrightarrow}
		\mathcal{G}^{(1)}(\F_{q,n}).
\]
Note that the $p$-primary part of the last term is isomorphic to
$\Sha_{p^\infty}(A / K_{\infty}^{(p)})^{\Gamma^{p^{n}}}$
by \eqref{e:Gsha}.

\begin{proposition} \label{p:ShaGone}
For all $n\ge 0$, the map $i_n$ is injective with finite cokernel of bounded order in $n$,
and the map $j_n$ has finite kernel and cokernel of bounded order in $n$.
\end{proposition}

\begin{proof}
	The first statement follows from the exact sequence
		\[
				0
			\to
				\Sha(A / K_{n}^{(p)})
			\to
				\coh^{1}(\mathcal{C}_{\F_{q, n}}, \mathcal{A})
			\to
				\bigoplus_{v \in \mathcal{C}_{\F_{q, n}}}
					\coh^{1} \bigl(
						\F_{q, n}(v), \pi_{0}(\mathcal{A}_{v})
					\bigr)
		\]
	written in \cite[III, Proposition 9.2]{Mil06}.
	For the second statement, we have a canonical spectral sequence
		\[
				E_{2}^{i j}
			=
				\coh^{i} \bigl(
					\F_{q,n},
					\mathcal{G}^{(j)}
				\bigr)
			\Longrightarrow
				\coh^{i + j}(\mathcal{C}_{\F_{q,n}}, \mathcal{A})
		\]
	by \cite[Proposition 2.7.8]{Suz19}.
	Hence we have an exact sequence
		\begin{align*}
					0
			&	\to
					\coh^{1}(\F_{q,n}, \mathcal{G}^{(0)})
				\to
					\coh^{1}(\mathcal{C}_{\F_{q,n}}, \mathcal{A})
				\to
					\mathcal{G}^{(1)}(\F_{q,n})
			\\
			&	\to
					\coh^{2}(\F_{q, n}, \mathcal{G}^{(0)})
				\to
					\coh^{2}(\mathcal{C}_{\F_{q, n}}, \mathcal{A}).
		\end{align*}
	The natural morphism from $\coh^{m}(\F_{q,n}, \mathcal{G}^{(0)})$
	to $\coh^{m}(\F_{q,n}, \pi_{0}(\mathcal{G}^{(0)}))$ is an isomorphism
	for any $m \ge 1$ since $\coh^{m}(\F_{q,n}, (\mathcal{G}^{(0)})^{0}) = 0$ by Lang's theorem.
	The group scheme $\pi_{0}(\mathcal{G}^{(0)})$ over $\F_{q}$ is \'etale
	with group of geometric points finitely generated
	by \cite[Theorem 3.4.1 (2)]{Suz19}.
	Hence $\coh^{1}(\F_{q,n}, \pi_{0}(\mathcal{G}^{(0)}))$ is finite of order bounded in $n$.
	Therefore the kernel of $\coh^{1}(\mathcal{C}_{\F_{q, n}}, \mathcal{A}) \to \mathcal{G}^{(1)}(\F_{q, n})$
	is finite of order bounded in $n$.
	
	To show that the cokernel of this map is also finite of order bounded in $n$,
	let $B / K$ be the abelian variety dual to $A$ with N\'eron model $\mathcal{B}$.
	Let $\mathcal{B}^{0}$ be the part of $\mathcal{B}$ with connected fibers.
	Let
		\[
				\mathcal{F}^{(i)}
			=
				\mathbf{H}^{i}(\mathcal{C}, \mathcal{B}^{0})
		\]
	be the group scheme defined in \cite{Suz19}.
	The group $\coh^{2}(\F_{q,n}, \mathcal{G}^{(0)})$ is isomorphic
	to $\coh^{1}(\F_{q,n}, \pi_{0}(\mathcal{G}^{(0)}) \otimes \Q / \Z)$
	since Galois cohomology is torsion in positive degrees.
	The latter group $\coh^{1}(\F_{q,n}, \pi_{0}(\mathcal{G}^{(0)}) \otimes \Q / \Z)$
	is Pontryagin dual to the profinite completion of
	$\Hom(\pi_{0}(\mathcal{G}^{(0)}) / \mathrm{tor}, \Z)(\F_{q, n})$
	($=$ the group of $\F_{q, n}$-valued points of the dual $\Z$-lattice of $\pi_{0}(\mathcal{G}^{(0)}) / \mathrm{tor}$).
	The height pairing gives a non-degenerate pairing
		\[
				\pi_{0}(\mathcal{G}^{(0)}) / \mathrm{tor}
			\times
				\pi_{0}(\mathcal{F}^{(0)}) / \mathrm{tor}
			\to
				\Z
		\]
	of $\Z$-lattices over $\F_{q}$ (\cite[Theorem 3.4.1 (6e), (7)]{Suz19}).
	Hence we have injections
		\[
				\pi_{0}(\mathcal{F}^{(0)})(\F_{q, n}) / \mathrm{tor}
			\hookrightarrow
				\bigl(
					\pi_{0}(\mathcal{F}^{(0)}) / \mathrm{tor}
				\bigr)(\F_{q, n})
			\hookrightarrow
				\Hom(\pi_{0}(\mathcal{G}^{(0)}) / \mathrm{tor}, \Z)(\F_{q, n})
		\]
	with finite cokernels of order bounded in $n$.
	Therefore we have a natural map
		\[
				\pi_{0}(\mathcal{F}^{(0)})(\F_{q, n})^{\wedge}
			\to
				\coh^{2}(\F_{q,n}, \mathcal{G}^{(0)})^{\vee}
		\]
	with finite cokernel of order bounded in $n$,
	where $\wedge$ denotes the profinite completion.
	
	On the other hand, we have a canonical exact sequence
		\[
				0
			\to
				\mathcal{B}^{0}(\mathcal{C}_{\F_{q,n}})^{\wedge}
			\to
				\coh^{2}(\mathcal{C}_{\F_{q,n}}, \mathcal{A})^{\vee}
			\to
				\mathrm T \coh^{1}(\mathcal{C}_{\F_{q,n}}, \mathcal{B}^{0})
			\to
				0
		\]
	by the duality result \cite[Proposition 4.2.3]{Suz19},
	where $\mathrm T$ denotes the profinite Tate module.
	By the connected-\'etale sequence
		\[
				0
			\to
				(\mathcal{F}^{(0)})^{0}
			\to
				\mathcal{F}^{(0)}
			\to
				\pi_{0}(\mathcal{F}^{(0)})
			\to
				0
		\]
	and Lang's theorem (namely, $\coh^{1}(\F_{q, n}, (\mathcal{F}^{(0)})^{0}) = 0$), the map
		\[
				\mathcal{F}^{(0)}(\F_{q,n})
			=
				\mathcal{B}^{0}(\mathcal{C}_{\F_{q,n}})
			\to
				\pi_{0}(\mathcal{F}^{(0)})(\F_{q,n})
		\]
	is surjective.
	
	Therefore the cokernel of the map from $\coh^{2}(\mathcal{C}_{\F_{q,n}}, \mathcal{A})^{\vee}$
	to $\coh^{2}(\F_{q,n}, \mathcal{G}^{(0)})^{\vee}$ is finite of order bounded in $n$.
	Thus the kernel of the map from $\coh^{2}(\F_{q,n}, \mathcal{G}^{(0)})$
	to $\coh^{2}(\mathcal{C}_{\F_{q,n}}, \mathcal{A})$ is finite of order bounded in $n$.
	This implies that the cokernel of
	$\coh^{1}(\mathcal{C}_{\F_{q, n}}, \mathcal{A}) \to \mathcal{G}^{(1)}(\F_{q, n})$
	is finite of order bounded in $n$.
\end{proof}

The above proof actually gives an explicit bound on the orders of (the $p$-primary part of) the kernel and cokernel of
$i_{n}$ and $j_{n}$ for large $n$ in terms of the finite \'etale groups $\pi_{0}(\mathcal{A}_{v})$,
$\pi_{0}(\mathcal{G}^{(0)})_{\mathrm{tor}}$ and $\pi_{0}(\mathcal{F}^{(0)})_{\mathrm{tor}}$
and the discriminant of the height pairing
on $\pi_{0}(\mathcal{G}^{(0)}) / \mathrm{tor} \times \pi_{0}(\mathcal{F}^{(0)}) / \mathrm{tor}$.

Using this, we can give another proof of
Propositions \ref{p:limint} and \ref{p:liminf}:

\begin{proof}[Another proof of Propositions \ref{p:limint} and \ref{p:liminf}]
	The $\Gamma^{p^{n}}$-invariant part of
	$\mathcal{G}^{(1)}(\F_{q, \infty})$ is $\mathcal{G}^{(1)}(\F_{q,n}$).
	Therefore
		\[
				\lim_{n \to \infty}
					\frac{
						\log |\mathcal{G}^{(1)}(\F_{q,n})[p^m]|
					}{
						p^n\log q
					}
			=
				s_{A / K}
		\]
	for all large $m$ by the same argument as the first part of the proof of Proposition \ref{p:limint},
	and if $\mathcal{G}^{(1)}(\F_{q, n})[p^{\infty}]$ is finite for all $n$, then
		\[
				\lim_{n \to \infty}
					\frac{
						\log |\mathcal{G}^{(1)}(\F_{q,n})[p^{\infty}]|
					}{
						p^n\log q
					}
			=
				s_{A / K}
		\]
	by Iwasawa's formula \cite[Theorem 8]{Ser95}.
	By Proposition \ref{p:ShaGone}, we can compare the asymptotic behavior of
	$|\mathcal{G}^{(1)}(\F_{q,n})[p^m]|$ and $|\Sha_{p^m}(A/K_n^{(p)})|$.
	This, with Theorem \ref{muissA}, gives
		\[
				\lim_{n\rightarrow \infty}\frac{\log |\Sha_{p^m}(A/K_n^{(p)})|}{{p^n}\log q}
			=
				\lim_{n \to \infty}
					\frac{
						\log |\mathcal{G}^{(1)}(\F_{q,n})[p^m]|
					}{
						p^n\log q
					}
			=
				s_{A / K}
			=
				\mu_{A / K}.
		\]
	for all large $m$, and if $\Sha_{p^{\infty}}(A / K_{n}^{(p)})$ is finite for all $n$, then
		\begin{equation}\label{e:Gone}
				\lim_{n\rightarrow \infty}\frac{\log |\Sha_{p^{\infty}}(A/K_n^{(p)})|}{{p^n}\log q}
			=
				\lim_{n \to \infty}
					\frac{
						\log |\mathcal{G}^{(1)}(\F_{q,n})[p^{\infty}]|
					}{
						p^n\log q
					}
			=
				s_{A / K}
			=
				\mu_{A / K}.
		\end{equation}
\end{proof}

\begin{remark} \mbox{}
	\begin{enumerate}
	\item
		Propositions \ref{p:Gchar} and \ref{p:ShaGone},
		together with the properties of $\mathcal{G}^{(1)}$ cited after \eqref{e:gsh},
		give another proof that $X_{A/K}$ is finitely generated torsion over $\Lambda$,
		i.e.\ \cite[Theorem 1.7]{OT09}.
	\item
		To seek for a number field analogue of the Tate-Shafarevich scheme,
		let $E / \Q$ be an elliptic curve and $p$ a prime number.
		Assume that $\Sel_{p^{\infty}}(E / \Q^{cyc})^{\vee}$ is torsion over
		$\Z_{p}[[\Gal(\Q^{cyc} / \Q)]]$.
		Does there exist the perfection of some smooth group scheme over $\F_{p}$
		whose group of $\F_{p, \infty}$-valued points is isomorphic to
		$\Sha_{p^{\infty}}(E / \Q^{cyc})$ as
		$\Z_{p}[[\Gal(\Q^{cyc} / \Q)]] \cong \Z_{p}[[\Gal(\F_{p, \infty} / \F_{p})]]$-modules?
	\end{enumerate}
\end{remark}

\subsection{The $\mu$-invariant formula}

With the formula \eqref{e:shafni} or \eqref{e:Gone}
and Proposition \ref{p:slope},
we can now reinterpret Proposition 4.2 of \cite{Ulm19} as follows:

\begin{corollary}\label{muA}
	Assume that $\Sha(A / K_{n}^{(p)})$ is finite for all $n$.
	We have
		\begin{equation} \label{e:muform}
				\mu_{A / K}
			=
					\deg(\mathcal{L})
				+
					\dim(A)(g_{\mathcal{C}} - 1)
				+
					\dim(\Tr_{K / \F_{q}}(A))
				-
					\theta,
		\end{equation}
	where $\mathcal{L}$ is the invertible sheaf on $\mathcal{C}$ defined by
		\[
				\mathcal{L}
			:=
				o^{*} \Omega_{\mathcal{A} / \mathcal{C}}^{\dim(A)}
		\]
	(here $o: \mathcal{C} \to \mathcal{A}$ denotes the zero section of $\mathcal{A} \to \mathcal{C}$).
\end{corollary}

We will prove below that the formula \eqref{e:muform} also holds
for the following two cases without any assumption on $\Sha$:
$A$ is the Jacobian of a projective smooth curve over $K$;
and $A$ is a semistable abelian variety over $K$.

\begin{remark}
	As in Proposition \ref{p:liminf},
	Ulmer \cite[4.1]{Ulm19} called the right-hand side of \eqref{e:shafni}
	the \emph{dimension of $\Sha(A)$}.
	He justified this terminology in \cite[4.3 (2)]{Ulm19}
	in the special case that $A$ is a Jacobian.
	For a general abelian variety $A$,
	our definition $s_{A / K} = \dim \mathcal{G}_{A}^{(1)}$
	and the formula \eqref{e:Gone} together justify the terminology.
\end{remark}

The term $\deg(\mathcal{L})$ is non-negative (\cite[Theorem 2.6]{Yua18}). Also, recall that the term $\theta$ in the formula of the above corollary is non-negative so that we get
the following upper bound on the $\mu$-invariant (see \cite[Proposition 4.4]{Ulm19} for another upper bound for $\mu$):

\begin{equation} \label{e:upperbound}
				\mu_{A / K}
			\le
					\deg(\mathcal{L})
				+
					\dim(A)(g_{\mathcal{C}} - 1)
				+
					\dim(\Tr_{K / \F_{q}}(A)).
		\end{equation}\

About the following definition, see also Theorem \ref{c:szpiro} below.

\begin{definition} \label{d:Sz}
	Define the Szpiro difference as
		\[
				d
			=
				d_{A}
			:=
					\frac{\deg(N)}{2}
				+
					\dim(A) (g_{\mathcal{C}} - 1)
				+
					\dim \mathrm{Tr}_{K / \F_{q}}(A)
				-
					\deg(\mathcal{L}).
		\]
	Also define $b = b_{A} := a / 2 - \theta$.
\end{definition}

Note that $b$ is non-negative by Proposition \ref{p:slope}
and defined purely from the $L$-function of $A$.
By \eqref{e:OS}, we can rewrite the formula \eqref{e:muform} in terms of $b$ and $d$:

\begin{corollary} \label{c:bminusd}
	The formula \eqref{e:muform} is equivalent to the formula
	$\mu_{A / K} = b - d$.
\end{corollary}

\begin{remark}
	The formula \eqref{e:muform} can be understood as a certain kind of Euler characteristic formulas
	for $L_{A}(s)$ at $s = 1$.
	To see this, let $\mathrm{Lie}(\mathcal{A})$ be the vector bundle on $\mathcal{C}$
	dual to $o^{*} \Omega_{\mathcal{A} / \mathcal{C}}$.
	Then by the Riemann-Roch theorem, we have
		\[
				\chi(\mathcal{C}, \mathrm{Lie}(\mathcal{A}))
			:=
				\sum_{n}
					(-1)^{n} \dim_{\F_{q}} \coh^{n}(\mathcal{C}, \mathrm{Lie}(\mathcal{A}))
			=
				\dim(A) (1 - g_{\mathcal{C}}) - \deg(\mathcal{L}).
		\]
	Also let
		\[
				\chi^{0}(\mathcal{C}, \mathcal{A})
			:=
				\sum_{n}
					(-1)^{n} \dim \mathcal{G}_{A}^{(n)}.
		\]
	By \cite[Propositions 3.2.2, 3.2.3, Theorem 3.4.1 (2)]{Suz19},
	we have $\dim \mathcal{G}_{A}^{(0)} = \dim \mathrm{Tr}_{K / \F_{q}}(A)$
	and $\dim \mathcal{G}_{A}^{(n)} = 0$ for $n \ne 0, 1$.
	Therefore the formula \eqref{e:muform} can be written as
		\[
				\theta
			=
					\chi^{0}(\mathcal{C}, \mathcal{A})
				-
					\chi(\mathcal{C}, \mathrm{Lie}(\mathcal{A})).
		\]
	Compare this with the Weil-\'etale BSD formula \cite[Proposition 8.4]{GS20}.
	This presentation might be useful when one wants to generalize \eqref{e:muform},
	for example to motives over $K$ other than abelian varieties.
\end{remark}

\section{The $\mu$-invariant for Jacobians} \label{s:JacSur}
In this section, we will show that the formula (\ref{e:muform}) holds for Jacobians
without any hypothesis (Corollary \ref{c:MilUlm})
and will deduce some necessary and sufficient conditions
about the (non-)vanishing of the $\mu$-invariant
(Propositions \ref{p:trivmu} and \ref{p:aborK3}).
A good reference on fibered surfaces is \cite{Liu02}.
Let $\mathcal{S}$ be a projective smooth surface over $\F_{q}$
and $\pi \colon \mathcal{S} \to \mathcal{C}$ a flat morphism over $\F_{q}$.
Assume that $\pi_{\ast} \mathcal{O}_{\mathcal{S}} = \mathcal{O}_{\mathcal{C}}$ and that
the generic fiber $\pi_{K} \colon \mathcal{S}_{K} \to \spec K$ of $\pi$ is smooth.
Let $A = \Pic_{\mathcal{S}_{K} / K}^{0}$ be the Jacobian variety of $\mathcal{S}_{K}$ over $K$.
Any elliptic curve is an example of such $A$
by the theory of regular models.
We do not assume that $\pi$ admits a section.

\subsection{The Brauer group}

Recall from \cite[\S 3]{Art74b}, \cite[\S 5]{Mil76} that there is a perfect group scheme
$\underline{\coh}^{i}(p^{\infty})$ over $\F_{q}$ for each integer $i$
such that for any perfect field extension $k / \F_{q}$, we have
	\[
			\underline{\coh}^{i}(p^{\infty})(k)
		=
			\coh_{fl}^{i}(\mathcal{S}_{\bar{k}}, \mu_{p^{\infty}})^{\Gal(\bar{k} / k)}
	\]
as a functor in the variable $k$,
where $\mathcal{S}_{\bar{k}} = \mathcal{S} \times_{\F_{q}} \bar{k}$.
Its identity component and component group are denoted by
$\underline{\mathrm{U}}^{i}(p^{\infty})$ and $\underline{\mathrm{D}}^{i}(p^{\infty})$, respectively.
The group $\underline{\mathrm{U}}^{i}(p^{\infty})$ is
the perfection of a commutative unipotent algebraic group over $\F_{q}$.
The group $\underline{\mathrm{D}}^{i}(p^{\infty})$ is
a $p$-primary torsion cofinite \'etale group scheme over $\F_{q}$.

Let $\underline{\mathrm{NS}} = \pi_{0}(\Pic_{\mathcal{S} / \F_{q}})$.
It is an \'etale group scheme over $\F_{q}$
such that its group of $\bar{\F}_{q}$-points is
the N\'eron-Severi group $\mathrm{NS}(\mathcal{S}_{\bar{\F}_{q}})$ of $\mathcal{S}_{\bar{\F}_{q}}$,
which is a finitely generated abelian group.
The Kummer sequence defines a canonical injection
$\underline{\mathrm{NS}} \otimes \Q_{p} / \Z_{p} \hookrightarrow \underline{\coh}^{2}(p^{\infty})$
of perfect group schemes.
Define $\underline{\mathrm{Br}}_{p^{\infty}}$ to be its cokernel
(a priori as a sheaf on the fppf site or the big \'etale site of $\F_{q}$).

\begin{proposition}
	The sheaf $\underline{\mathrm{Br}}_{p^{\infty}}$ is represented by
	a perfect group scheme over $\F_{q}$.
	Its identity component is a quotient of $\underline{\mathrm{U}}^{2}(p^{\infty})$
	by a finite \'etale group.
\end{proposition}

\begin{proof}
	We have an exact sequence
		\[
				0
			\to
				\frac{
					\underline{\mathrm{U}}^{2}(p^{\infty})
				}{
						(\underline{\mathrm{NS}} \otimes \Q_{p} / \Z_{p})
					\cap
						\underline{\mathrm{U}}^{2}(p^{\infty})
				}
			\to
				\underline{\mathrm{Br}}_{p^{\infty}}
			\to
				\frac{
					\underline{\mathrm{D}}^{2}(p^{\infty})
				}{
					\underline{\mathrm{NS}} \otimes \Q_{p} / \Z_{p}
				}
			\to
				0.
		\]
	The intersection
	$(\underline{\mathrm{NS}} \otimes \Q_{p} / \Z_{p})
	\cap \underline{\mathrm{U}}^{2}(p^{\infty})$
	is a finite \'etale group
	since the unipotent group $\underline{\mathrm{U}}^{2}(p^{\infty})$ over $\F_{q}$ is
	killed by a power of $p$.
	Hence the first term in the above exact sequence is
	the perfection of a connected unipotent algebraic group.
	The third term is an \'etale group scheme.
	Therefore $\underline{\mathrm{Br}}_{p^{\infty}}$ is represented by
	a perfect group scheme,
	and the above exact sequence is its connected-\'etale sequence.
\end{proof}

For a prime $l \ne p$, define $\underline{\Br}_{l^{\infty}}$ to be the \'etale group scheme over $\F_{q}$
whose group of $\bar{\F}_{q}$-points is $\Br(\mathcal{S}_{\bar{\F}_{q}})[l^{\infty}]$,
where $\Br$ denotes the usual Brauer group for schemes.
Define $\underline{\Br}$ to be the direct sum of
$\underline{\Br}_{l^{\infty}}$ over all primes $l$,
which is again a perfect group scheme.
For any perfect field extension $k / \F_{q}$, we have a canonical isomorphism
	\[
			\underline{\mathrm{Br}}(k)
		\cong
			\mathrm{Br}(\mathcal{S}_{\bar{k}})^{\Gal(\bar{k} / k)}
	\]
functorial in $k$.

\begin{remark}
	There is a natural map $\mathrm{Br}(\mathcal{S}) \to \underline{\mathrm{Br}}(\F_{q})$.
	It can be shown that the kernel and cokernel of this map are finite,
	and the kernel and cokernel of the map
	$\mathrm{Br}(\mathcal{S}_{\F_{q, n}}) \to \underline{\mathrm{Br}}(\F_{q, n})$
	are finite of order bounded in $n$,
	by an argument similar to the proof of Proposition \ref{p:ShaGone},
	using \cite{Mil76} instead of \cite{Suz19}.
	An explicit bound for large $n$ involves the order of the torsion part of $\underline{\mathrm{NS}}$
	and the discriminant of the intersection pairing on $\underline{\mathrm{NS}} / \mathrm{tor}$.
	We do not use this fact, so we do not prove it.
	See \cite[Section 6]{Art74b} for a discussion about this map
	in the special case of an elliptic supersingular K3 surface over a large enough finite field.%
	\footnote{Note however that \cite[(6.1)]{Art74b} claims that
	the map $\mathrm{Br}(\mathcal{S}) \to \underline{\mathrm{Br}}(\F_{q})$ is bijective for such $\mathcal{S}$.
	This is not correct: the map is injective, but
	the cokernel is isomorphic to the discriminant group
	$\mathrm{NS}(\mathcal{S})^{\ast} / \mathrm{NS}(\mathcal{S}) \ne 0$
	of the intersection pairing on $\mathcal{S}$,
	as everywhere else in \cite[Section 6]{Art74b} correctly suggests.}
\end{remark}

\begin{definition}
	Define $s_{\mathcal{S}}$ to be the common integer
		\[
				\dim \underline{\coh}^{i}(p^{\infty})
			=
				\dim \underline{\mathrm{U}}^{i}(p^{\infty})
			=
				\dim \underline{\mathrm{Br}}_{p^{\infty}}
			=
				\dim \underline{\mathrm{Br}}.
		\]
\end{definition}

\begin{remark}\label{sdef}
Note that the integer $s_{\mathcal{S}}$  is also equal
to the length of
	\[
		\coh^2(\mathcal{S}, W \mathcal{O}_\mathcal{S})[p^{\infty}] \otimes_{W[[V]]} W((V))
	\]
as a $W((V))$-module,
where $W = W(\F_{q})$ and $V$ denotes the Verschiebung;
see \cite[(1.3)]{Mil75} (see also \cite[Lemma 2.1, Proposition 4.4]{MR15}).
If the formal Brauer group of $\mathcal{S}$ is
pro-representable by a formal Lie group
(such as when $\Pic_{\mathcal{S} / \F_{q}}$ is smooth \cite[(4.1)]{AM77}),
then its Dieudonn\'e module is
given by $\coh^2(\mathcal{S}, W \mathcal{O}_\mathcal{S})$ (\cite[II, (4.3)]{AM77}).
Therefore in this case, $s_{\mathcal{S}}$ is also equal to the
dimension of the unipotent part of the formal Brauer group.
\end{remark}

\begin{proposition} \label{p:ShaBr}
	There exists a canonical morphism
	$\mathcal{G}_{A}^{(1)} \to \underline{\mathrm{Br}}$
	of perfect group schemes over $\F_{q}$
	with finite \'etale kernel and cokernel.
	In particular,
		\[
				s_{A / K} = s_{\mathcal{S}}
		\]
	(which is equal to $\mu_{A / K}$ by Theorem \ref{muissA}).
\end{proposition}

\begin{proof}
	Let $j \colon \spec K \hookrightarrow \mathcal{C}$ be the inclusion.
	Let $\mathrm{Pic}_{\mathcal{S} / \mathcal{C}}$ and $\mathrm{Pic}_{\mathcal{S}_{K} / K}$ be the Picard functors sheafified in the \'etale topology
	(\cite[Definition 2.2]{Kle05}, \cite[8.1]{BLR90}).
	For any perfect field extension $k / \F_{q}$, we have canonical homomorphisms
		\[
				\Br(\mathcal{S}_{k})
			\to
				\coh^{1}(\mathcal{C}_{k}, \Pic_{\mathcal{S} / \mathcal{C}})
			\to
				\coh^{1}(\mathcal{C}_{k}, j_{\ast} \Pic_{\mathcal{S}_{K} / K})
		\]
	functorial in $k$, which are isomorphisms if $k$ is algebraically closed,
	by \cite[Proposition (4.3), Equation (4.14 bis)]{Gro68}.
	
	On the other hand, consider the exact sequence
	$0 \to A \to \Pic_{\mathcal{S}_{K} / K} \to \Z \to 0$ over $K_{\mathrm{et}}$.
	Applying $j_{\ast}$, we have an exact sequence
	$0 \to \mathcal{A} \to j_{\ast} \Pic_{\mathcal{S}_{K} / K} \to \Z$ over $\mathcal{C}_{\mathrm{et}}$.
	Let $Q$ be the image of $j_{\ast} \Pic_{\mathcal{S}_{K} / K} \to \Z$,
	which is a $\Z$-constructible \'etale subsheaf of $\Z$ on $\mathcal{C}$.
	The quotient $\Z / Q$ is a skyscraper sheaf with finite stalks.
	Consider the exact sequence $0 \to \mathcal{A} \to j_{\ast} \Pic_{\mathcal{S}_{K} / K} \to Q \to 0$.
	For any perfect field extension $k / \F_{q}$, this induces an exact sequence
		\[
				Q(\mathcal{C}_{k})
			\to
				\coh^{1}(\mathcal{C}_{k}, \mathcal{A})
			\to
				\coh^{1}(\mathcal{C}_{k}, j_{\ast} \Pic_{\mathcal{S}_{K} / K})
			\to
				\coh^{1}(\mathcal{C}_{k}, Q)
		\]
	functorial in $k$.
	The image of $Q(\mathcal{C}_{k}) \to \coh^{1}(\mathcal{C}_{k}, \mathcal{A})$ is torsion
	since $\coh^{1}(\mathcal{C}_{k}, \mathcal{A})$ is torsion
	by the proof of \cite[III, Lemma 11.5]{Mil06}.
	If $k$ is algebraically closed, then $Q(\mathcal{C}_{k})$ is finitely generated
	and $\coh^{1}(\mathcal{C}_{k}, Q)$ is finite
	by the properties of $Q$ seen above.
	Therefore as functors in $k$, the kernel and cokernel of
		\[
				\coh^{1}(\mathcal{C}_{\bar{k}}, \mathcal{A})^{\Gal(\bar{k} / k)}
			\to
				\coh^{1}(\mathcal{C}_{\bar{k}}, j_{\ast} \Pic_{\mathcal{S}_{K} / K})^{\Gal(\bar{k} / k)}
		\]
	are represented by finite \'etale group schemes over $\F_{q}$.
	
	Now we have a homomorphism and an isomorphism
		\[
				\coh^{1}(\mathcal{C}_{\bar{k}}, \mathcal{A})^{\Gal(\bar{k} / k)}
			\to
				\coh^{1}(\mathcal{C}_{\bar{k}}, j_{\ast} \Pic_{\mathcal{S}_{K} / K})^{\Gal(\bar{k} / k)}
			\overset{\sim}{\gets}
				\Br(\mathcal{S}_{\bar{k}})^{\Gal(\bar{k} / k)}
		\]
	functorial in $k$.
	This gives the required morphism
	$\mathcal{G}_{A}^{(1)} \to \underline{\mathrm{Br}}$.
\end{proof}

\subsection{The trace and the Picard variety}

\begin{proposition} \label{p:TrPic}
	Let $\Pic_{\mathcal{S} / \F_{q}, \mathrm{red}}^{0}$ be the reduced part of $\Pic_{\mathcal{S} / \F_{q}}^{0}$.
	Then there exist canonical morphisms
	$\Pic_{\mathcal{C} / \F_{q}}^{0} \to \Pic_{\mathcal{S} / \F_{q}, \mathrm{red}}^{0} \to \mathrm{Tr}_{K / \F_{q}}(A)$
	over $\F_{q}$, which induce an exact sequence
		\[
				0
			\to
				\Pic_{\mathcal{C} / \F_{q}}^{0}
			\to
				\Pic_{\mathcal{S} / \F_{q}, \mathrm{red}}^{0}
			\to
				\mathrm{Tr}_{K / \F_{q}}(A)
			\to
				0
		\]
	in the category of abelian varieties up to isogeny over $\F_{q}$.
	In particular, we have
		\[
				\dim \Pic_{\mathcal{S} / \F_{q}, \mathrm{red}}^{0}
			=
				\dim \mathrm{Tr}_{K / \F_{q}}(A) + g_{\mathcal{C}}.
		\]
\end{proposition}

\begin{proof}
	The morphism
	$\Pic_{\mathcal{C} / \F_{q}}^{0} \to \Pic_{\mathcal{S} / \F_{q}, \mathrm{red}}^{0}$
	is the obvious one.
	To define the second morphism,
	consider the inclusion
		$
				\mathcal{S}_{K}
			=
				\mathcal{S} \times_{\mathcal{C}} K
			\hookrightarrow
				\mathcal{S} \times_{\F_{q}} K
		$
	over $K$.
	This induces a morphism
		\[
				(\Pic_{\mathcal{S} / \F_{q}}) \times_{\F_{q}} K
			=
				\Pic_{(\mathcal{S} \times_{\F_{q}} K) / K}
			\to
				\Pic_{\mathcal{S}_{K} / K}
		\]
	and hence a morphism
		\[
				(\Pic_{\mathcal{S} / \F_{q}, \mathrm{red}}^{0}) \times_{\F_{q}} K
			\to
				\Pic_{\mathcal{S}_{K} / K}^{0}
			=
				A
		\]
	over $K$.
	The universal property of the trace then induces a morphism
	$\Pic_{\mathcal{S} / \F_{q}, \mathrm{red}}^{0} \to \mathrm{Tr}_{K / \F_{q}}(A)$
	over $\F_{q}$.
	
	To prove the second claim, it is enough to show that
	these morphisms induce an exact sequence
		\[
				0
			\to
				\Pic_{\mathcal{C} / \F_{q}}^{0}(\bar{\F}_{q})
			\to
				\Pic_{\mathcal{S} / \F_{q}}^{0}(\bar{\F}_{q})
			\to
				\mathrm{Tr}_{K / \F_{q}}(A)(\bar{\F}_{q})
			\to
				0
		\]
	up to finite groups.
	By \cite[Equation (4.14)]{Gro68}, there exists a canonical exact sequence
		\[
				0
			\to
				E
			\to
				\Pic_{\mathcal{S} / \mathcal{C}}(\mathcal{C}_{\bar{\F}_{q}})
			\to
				\Pic_{\mathcal{S}_{K / K}}(K \bar{\F}_{q})
			\to
				0
		\]
	for some finitely generated group $E$.
	Since $\Br(\mathcal{C}_{\bar{\F}_{q}}) = 0$,
	we have a canonical exact sequence
		\[
				0
			\to
				\Pic(\mathcal{C}_{\bar{\F}_{q}})
			\to
				\Pic(\mathcal{S}_{\bar{\F}_{q}})
			\to
				\Pic_{\mathcal{S} / \mathcal{C}}(\mathcal{C}_{\bar{\F}_{q}})
			\to
				0
		\]
	by \cite[Equation (4.5)]{Gro68}.
	Hence we have an exact sequence
		\[
				0
			\to
				\Pic(\mathcal{C}_{\bar{\F}_{q}}) \oplus E
			\to
				\Pic(\mathcal{S}_{\bar{\F}_{q}})
			\to
				\Pic_{\mathcal{S}_{K / K}}(K \bar{\F}_{q})
			\to
				0.
		\]
	The group $\Pic_{\mathcal{S} / \F_{q}}^{0}(\bar{\F}_{q})$ is the divisible subgroup of
	$\Pic(\mathcal{S}_{\bar{\F}_{q}})$ and the quotient of $\Pic(\mathcal{S}_{\bar{\F}_{q}})$
	by $\Pic_{\mathcal{S} / \F_{q}}^{0}(\bar{\F}_{q})$ is $\mathrm{NS}(\mathcal{S}_{\bar{\F}_{q}})$,
	which is finitely generated.
	On the other hand,
	the group $\mathrm{Tr}_{K / \F_{q}}(A)(\bar{\F}_{q})$
	is the divisible subgroup of $A(K \bar{\F}_{q})$
	and the quotient of $A(K \bar{\F}_{q})$ by
	$\mathrm{Tr}_{K / \F_{q}}(A)(\bar{\F}_{q})$ is finitely generated
	by the Lang-N\'eron theorem (\cite[Theorem 7.1]{Con06}).
	Since $\Pic_{\mathcal{S}_{K} / K}(K \bar{\F}_{q}) / A(K \bar{\F}_{q})$ injects into $\Z$,
	this implies that the group $\mathrm{Tr}_{K / \F_{q}}(A)(\bar{\F}_{q})$
	is the divisible subgroup of $\Pic_{\mathcal{S}_{K} / K}(K \bar{\F}_{q})$
	and the quotient of $\Pic_{\mathcal{S}_{K} / K}(K \bar{\F}_{q})$ by
	$\mathrm{Tr}_{K / \F_{q}}(A)(\bar{\F}_{q})$ is finitely generated.
	Therefore the above sequence induces an exact sequence
		\[
				0
			\to
				\Pic_{\mathcal{C} / \F_{q}}^{0}(\bar{\F}_{q})
			\to
				\Ker \bigl(
						\Pic_{\mathcal{C} / \F_{q}}^{0}(\bar{\F}_{q})
					\twoheadrightarrow
						\mathrm{Tr}_{K / \F_{q}}(A)(\bar{\F}_{q})
				\bigr)
			\to
				\Z \oplus E,
		\]
	and the morphism from the middle kernel term to $\Z \oplus E$ has finite image.
	This proves the second claim.
\end{proof}

\subsection{The $L$-function and the zeta function}

Write the zeta function of $\mathcal{S}$ as
	\[
			Z(\mathcal{S}, s)
		=
			\frac{
				P_{\mathcal{S}, 1}(t) P_{\mathcal{S}, 3}(t)
			}{
				P_{\mathcal{S}, 0}(t) P_{\mathcal{S}, 2}(t) P_{\mathcal{S}, 4}(t)
			},
	\]
where $P_{\mathcal{S}, i}(t) \in \Z[t]$ is a polynomial in $t = q^{-s}$ with constant term $1$
whose reciprocal roots are Weil $q$-numbers $\{\alpha_{\mathcal{S}, ij}\}$ of weight $i$.
Define $\theta_{\mathcal{S}}$ to be the non-negative number such that
$q^{\theta_{\mathcal{S}}} P_{\mathcal{S}, 2}(t / q)$ is $p$-primitive.
Let $\lambda_{\mathcal{S}, j}$ be the $q$-valuation of $\alpha_{\mathcal{S}, 2 j}$.
Then again we have
	\[
			\theta_{\mathcal{S}}
		=
			\sum_{\lambda_{\mathcal{S}, j} < 1} (1 - \lambda_{\mathcal{S}, j}).
	\]

\begin{proposition} \label{p:Lcomp}
	The zeros and poles of the rational function
	$P_{2, \mathcal{S}}(t / q) / P_{1}(t / q)$
	are roots of unity.
	In particular, we have $\theta_{\mathcal{S}} = \theta_{A}$.
\end{proposition}

\begin{proof}
	Let $l \ne p$ be a prime number.
	We have an exact sequence
		\begin{equation} \label{e:NSBr}
				0
			\to
				\mathrm{NS}(\mathcal{S}_{\bar{\F}_{q}}) \otimes \Q_{l}
			\to
				\coh^{2}(\mathcal{S}_{\bar{\F}_{q}}, \Q_{l}(1))
			\to
				\mathrm{V}_{l}(\Br(\mathcal{S}_{\bar{\F}_{q}}))
			\to
				0
		\end{equation}
	of $l$-adic representations over $\F_{q}$,
	where $\mathrm{V}_{l}$ denotes the $l$-adic Tate module tensored with $\Q_{l}$
	and the cohomology in the middle term is the continuous cohomology.
	The characteristic polynomial of geometric Frobenius
	on the middle term $\coh^{2}(\mathcal{S}_{\bar{\F}_{q}}, \Q_{l}(1))$ is $P_{\mathcal{S}, 2}(t / q)$.
	Since the Galois action on the left term
	$\mathrm{NS}(\mathcal{S}_{\bar{\F}_{q}}) \otimes \Q_{l}$
	factors through a finite quotient,
	its Frobenius eigenvalues are roots of unity.
	
	On the other hand,
	let $\mathcal{A}^{0}$ be the part of $\mathcal{A}$ with connected fibers.
	For any $n \ge 0$, the sequence
	$0 \to \mathcal{A}^{0}[l^{n}] \to \mathcal{A}^{0} \overset{l^{n}}{\to} \mathcal{A}^{0} \to 0$
	is exact by \cite[7.3/1, 2]{BLR90}.
	Hence it induces an exact sequence
		\[
				0
			\to
				\mathcal{A}^{0}(\mathcal{C}_{\bar{\F}_{q}}) \otimes \Z / l^{n} \Z
			\to
				\coh^{1}(\mathcal{C}_{\bar{\F}_{q}}, \mathcal{A}^{0}[l^{n}])
			\to
				\coh^{1}(\mathcal{C}_{\bar{\F}_{q}}, \mathcal{A}^{0})[l^{n}]
			\to
				0
		\]
	of $\Z / l^{n} \Z$-representations over $\F_{q}$.
	Taking the inverse limit in $n$ and the tensor product with $\Q_{l}$,
	we have an exact sequence
		\begin{equation} \label{e:MWSha}
				0
			\to
				\frac{A(K \bar{\F}_{q})}{\mathrm{Tr}_{K / \F_{q}}(A)(\bar{\F}_{q})} \otimes \Q_{l}
			\to
				\coh^{1}(\mathcal{C}_{\bar{\F}_{q}}, \mathrm{V}_{l}(\mathcal{A}))
			\to
				\mathrm{V}_{l}(\coh^{1}(\mathcal{C}_{\bar{\F}_{q}}, \mathcal{A}))
			\to
				0
		\end{equation}
	of $l$-adic representations over $\F_{q}$.
	By \cite[Satz 1]{Sch82},
	the characteristic polynomial of geometric Frobenius
	on the middle term $\coh^{1}(\mathcal{C}_{\bar{\F}_{q}}, \mathrm{V}_{l}(\mathcal{A}))$ is $P_{1}(t / q)$.
	Since the Galois action on the left term
	$A(K \bar{\F}_{q}) / \mathrm{Tr}_{K / \F_{q}}(A)(\bar{\F}_{q}) \otimes \Q_{l}$
	factors through a finite quotient,
	its Frobenius eigenvalues are roots of unity.
	
	Now the morphism in Proposition \ref{p:ShaBr} induces an isomorphism
		\[
				\mathrm{V}_{l}(\coh^{1}(\mathcal{C}_{\bar{\F}_{q}}, \mathcal{A}))
			\overset{\sim}{\to}
				\mathrm{V}_{l}(\Br(\mathcal{S}_{\bar{\F}_{q}})).
		\]
	The result follows from this.
\end{proof}

\subsection{Comparison and criteria}

Define
	\[
			\chi(\mathcal{S}, \mathcal{O}_{\mathcal{S}})
		=
			\sum_{n} (-1)^{n} \dim_{\F_{q}} \coh^{n}(\mathcal{S}, \mathcal{O}_{\mathcal{S}}).
	\]
By \cite[Equation (14)]{LLR04}, we have
	\[
			\chi(\mathcal{S}, \mathcal{O}_{\mathcal{S}})
		=
			\deg(\mathcal{L}) + (1 - \dim(A)) (1 - g_{\mathcal{C}}).
	\]
By Milne's formula \cite[the last equation of Section 6]{Mil75} on Euler characteristics
or Crew's formula \cite[Proposition 6.3.9]{Ill83} on Hodge-Witt numbers, we have
	\begin{equation} \label{e:Milne}
			\sum_{\lambda_{\mathcal{S}, j} < 1} (1 - \lambda_{\mathcal{S}, j})
		=
				\chi(\mathcal{S}, \mathcal{O}_{\mathcal{S}})
			-
				(1 - \dim \Pic_{\mathcal{S} / \F_{q}, \mathrm{red}}^{0} + s_{\mathcal{S}}).
	\end{equation}
Combining these with Propositions \ref{p:ShaBr} and \ref{p:TrPic}, we have:

\begin{corollary} \label{c:MilUlm}
	The formula \eqref{e:Milne} is equivalent to the formula
	\eqref{e:muform} for the Jacobian $A$.
	In particular, the formula \eqref{e:muform} holds for $A$
	without any assumption on $\Sha$ (or the reduction type of $A$).
	In particular, it holds for any elliptic curve $A / K$.
\end{corollary}

The equivalence between Milne's formula for $\mathcal{S}$
and the formula \eqref{e:muform} for a Jacobian $A$ is
suggested by Ulmer \cite[Remark 4.3 (3), (4)]{Ulm19}.
The above corollary verifies his suggestion.

We will give some criteria and cases of $\mu = 0$ and $\mu > 0$.

\begin{proposition}\label{p:trivmu}
	The following are equivalent:
	\begin{itemize}
		\item[(i)] $\mu_{A / K} = 0$.
		\item[(ii)] $\coh_{fl}^{2}(\mathcal{S}_{\bar{\F}_{q}}, \mu_{p})$ is finite.
		\item[(iii)] $\Br(\mathcal{S}_{\bar{\F}_{q}})[p]$ is finite.
		\item[(iv)] $\mathcal{S}$ is of Hodge-Witt type, i.e.\
		$\coh^i(\mathcal{S},W \Omega^j_{\mathcal{S}})$ is a finitely generated $W(\F_q)$-module for any $i,j$.
	\end{itemize}
\end{proposition}

\begin{proof}
	Let $W = W(\F_{q})$.
	From the definition of $s_{\mathcal{S}}$,
	we know that $\mu_{A / K} = s_{\mathcal{S}} = 0$ if and only if $\underline{\coh}^{2}(\mu_{p^{\infty}})$ is \'etale.
	As $\underline{\mathrm{U}}^{2}(\mu_{p^{\infty}})$ is always unipotent,
	this is equivalent to $\underline{\coh}^{2}(\mu_{p^{\infty}})[p]$ being \'etale.
	The group scheme $\underline{\coh}^{1}(\mu_{p^{\infty}})$ is always \'etale
	(\cite[the paragraph after Theorem (3.1)]{Art74b}, \cite[Theorem 5.2]{Mil76}).
	Hence the exact sequence
		\[
				0
			\to
				\underline{\coh}^{1}(\mu_{p^{\infty}}) / p \underline{\coh}^{1}(\mu_{p^{\infty}})
			\to
				\underline{\coh}^{2}(\mu_{p})
			\to
				\underline{\coh}^{2}(\mu_{p^{\infty}})[p]
			\to
				0
		\]
	shows that $\underline{\coh}^{2}(\mu_{p^{\infty}})[p]$ is \'etale
	if and only if $\underline{\coh}^{2}(\mu_{p})$ is \'etale.
	Since $\underline{\coh}^{2}(\mu_{p})$ is the perfection of a smooth algebraic group,
	this is equivalent to its group of geometric points $\coh_{fl}^{2}(\mathcal{S}_{\bar{\F}_{q}}, \mu_{p})$ being finite.
	The group $\Pic(\mathcal{S}_{\bar{\F}_{q}}) / p \Pic(\mathcal{S}_{\bar{\F}_{q}})$ is always finite.
	Hence $\coh_{fl}^{2}(\mathcal{S}_{\bar{\F}_{q}}, \mu_{p})$ is finite if and only if
	$\Br(\mathcal{S}_{\bar{\F}_{q}})[p]$ is finite. This proves the first three equivalences.
	
	For the equivalence of (i) and (iv), first recall that
	$\coh^2(\mathcal{S}, W \mathcal{O}_{\mathcal{S}})$ is finitely generated over $W[[V]]$,
	its $p$-primary torsion submodule $\coh^2(\mathcal{S}, W \mathcal{O}_{\mathcal{S}})[p^{\infty}]$ is killed by a power of $p$,
	and the quotient of $\coh^2(\mathcal{S}, W \mathcal{O}_{\mathcal{S}})$
	by $\coh^2(\mathcal{S}, W \mathcal{O}_{\mathcal{S}})[p^{\infty}]$ is finite free over $W$
	by \cite[II, Corollaire 2.11 and Th\'eor\`eme 2.13]{Ill79}.
	In particular, $\mu_{A / K} = s_{\mathcal{S}} = 0$ if and only if
	$\coh^2(\mathcal{S}, W \mathcal{O}_{\mathcal{S}})[p^{\infty}] \otimes_{W[[V]]} W((V))$ has length $0$ as $W((V))$-module,
	i.e.\ is trivial.
	
	To prove the equivalence, first assume that $\coh^2(\mathcal{S}, W \mathcal{O}_{\mathcal{S}})[p^{\infty}] \otimes_{W[[V]]}
	W((V))$ is zero. This implies that $\coh^2(\mathcal{S}, W \mathcal{O}_{\mathcal{S}})[p] \otimes_{W[[V]]} W((V))$,
	or $\coh^2(\mathcal{S}, W \mathcal{O}_{\mathcal{S}})[p] \otimes_{k[[V]]} k((V))$ is zero.
	Since $\coh^2(\mathcal{S}, W \mathcal{O}_{\mathcal{S}})[p]$ is finitely generated over $k
	[[V]]$, this implies that $\coh^2(\mathcal{S}, W \mathcal{O}_{\mathcal{S}})[p]$ is finite-dimensional over $k$.
	Then $\coh^2(\mathcal{S}, W \mathcal{O}_{\mathcal{S}})[p^{\infty}]$ has finite length over $W$.
	Hence $\coh^2(\mathcal{S}, W \mathcal{O}_{\mathcal{S}})$ is
	finitely generated over $W$. Since $\mathcal{S}$ is a surface,
	by Nygaard's theorem (\cite[II, Corollaire 3.14]{Ill79}),
	we know that the slope spectral sequence for $\mathcal{S}$ degenerates at $E_1$. Since
	crystalline cohomology is finitely generated over $W$, this implies that
	all the $E_1$-terms are finitely generated over $W$.
	Hence $\mathcal{S}$ is of Hodge-Witt type.
	
	Conversely, assume that $\mathcal{S}$ is of Hodge-Witt type.
	Then $\coh^2(\mathcal{S}, W \mathcal{O}_{\mathcal{S}})[p^{\infty}]$ has
	finite length over $W$. Since the $V$-action is topologically nilpotent,
	this implies that $\coh^2(\mathcal{S}, W \mathcal{O}_{\mathcal{S}})[p^{\infty}]$ is killed by a power of $V$. This proves the claim.
\end{proof}

\begin{remark}
It is a non-trivial result \cite[III, Th\'eor\`eme (4.13)]{IR83} that
if $\mathcal{S}$ is ordinary in the sense of Illusie-Raynaud \cite[IV, D\'efinition (4.12)]{IR83},
then it is of Hodge-Witt type.
The reciprocal statement does not hold however.
\end{remark}

\begin{proposition}\label{p:aborK3}
	Assume that $\mathcal{S}$ is relatively minimal over $\mathcal{C}$
	and has Kodaira dimension $\le 0$.
	Then $\mu_{A / K} \ne 0$ if and only if $\mu_{A / K} = 1$
	if and only if $\mathcal{S}$ is a supersingular abelian surface
	or a supersingular K3 surface.
\end{proposition}

\begin{proof}
	We will use the Kodaira-Enriques classification (\cite[\S 6, \S 7]{Lie13}).
	If the Kodaira dimension of $\mathcal{S}$ is $- \infty$,
	then $\mathcal{S}$ is of Hodge-Witt type by
	\cite[\S 5.8, Proposition 5.8.1]{Jos14},
	so $\mu_{A / K} = 0$.
	If $\mathcal{S}$ is an abelian surface or a K3 surface,
	then it is not of Hodge-Witt type if and only if it is supersingular
	by \cite[Proposition 5.9.1]{Jos14},
	in which case the formal Brauer group of $\mathcal{S}$ is the formal completion of the additive group
	by \cite[II, \S 7.1, 7.2]{Ill79}
	and hence $\mu_{A / K} = 1$
	by Remark \ref{sdef}.
	If $\mathcal{S}$ is (quasi-)hyperelliptic, then
	$\dim \Pic_{\mathcal{S} / \F_{q}, \mathrm{red}}^{0} = 1$
	and $\chi(\mathcal{S}, \mathcal{O}_{\mathcal{S}}) = 0$
	by the table in Introduction of \cite{BM77}
	(with $\Pic_{\mathcal{S} / \F_{q}}^{0}$ reduced or not).
	Hence by \cite[Corollary 5]{Suw83}, we know that $\mathcal{S}$ is of Hodge-Witt type.
	If $\mathcal{S}$ is Enriques, then it is of Hodge-Witt type
	by \cite[II, Corollaire 7.3.3 (a)]{Ill79}.
\end{proof}

An explicit example is given as follows.

\begin{proposition} \label{p:Shio}
	Suppose $p \equiv 3 \mod 4$.
	Let $\mathcal{S} \to \mathcal{C} = \mathbb{P}_{\F_{q}}^{1}$ be Shioda's modular elliptic surface of level $4$
	(\cite[Appendices A-C]{Shi72}, \cite[\S 3]{Shi75}; see also \cite{Shi20}),
	whose generic fiber $\mathcal{S}_{K} = A$ is the elliptic curve defined by the equation
		\[
				y^{2}
			=
				x (x - 1)
				\left(
					x -
					\frac{1}{4} \left(
						t + \frac{1}{t}
					\right)^{2}
				\right),
		\]
	where $t$ is the coordinate of $\mathbb{P}_{\F_{q}}^{1}$.
	Then $\mu_{A / K} = 1$, and $A$ has semistable reduction everywhere.
\end{proposition}

\begin{proof}
	By \cite[Appendix A, viii)]{Shi72}, \cite[\S 3, Theorem 1, Corollary 1]{Shi75}
	(see also \cite[Theorem 1.1 (1), (2)]{Shi20}),
	we know that $\mathcal{S}$ is a supersingular K3 surface
	semistable everywhere over $\mathcal{C}$.
	Hence this proposition is a special case of the previous proposition.
\end{proof}

Not much can be said systematically about the case of positive Kodaira dimension,
except Proposition \ref{p:KodOne}, Section \ref{su:generic} and
an example in Section \ref{su:high} below.

\begin{remark} \mbox{}
	\begin{enumerate}
	\item
		In the proof of Corollary \ref{c:MilUlm}, we saw that
		Crew's formula on Hodge-Witt numbers of $\mathcal{S}$ (with $i = 0$)
		is equivalent to the $\mu$-invariant formula \eqref{e:muform} for the Jacobian $A / K$.
		By \cite[(6.3.5-8)]{Ill83},
		Crew's formula is applicable to any object of $D^{b}_{c}(R)$,
		where $R$ is the Raynaud ring.
		We used this formula for the de Rham-Witt cohomology
		$R \Gamma(\mathcal{S}, W \Omega_{\mathcal{S}}) \in D^{b}_{c}(R)$.
		On the other hand, Illusie, at the end of \cite{Ill83},
		expressed his hope to define a theory of de Rham-Witt complexes
		with coefficients in $F$-crystals.
		
		With these observations in mind, one may speculate that
		there should be a ``de Rham-Witt complex with coefficients in the N\'eron model $\mathcal{A}$'',
		``$W \Omega_{\mathcal{C}} \otimes_{W \mathcal{O}_{\mathcal{C}}} D(\mathcal{A})$'',
		corresponding to an arbitrary abelian variety $A / K$,
		so that Crew's formula for
			$
					R \Gamma \bigl(
						\mathcal{C},
						W \Omega_{\mathcal{C}} \otimes_{W \mathcal{O}_{\mathcal{C}}} D(\mathcal{A})
					\bigr)
				\in
					D^{b}_{c}(R)
			$
		contains the $\mu$-invariant formula \eqref{e:muform} for $A / K$ as a special case.
	\item
		Let $\mathcal{S}$ and $A = \Pic^{0}_{\mathcal{S}_{K} / K}$ as above.
		Consider the canonical domino associated with the graded $R$-module
			$
					\coh^{2}(\mathcal{S}, W \mathcal{O}_{\mathcal{S}})
				\to
					\coh^{2}(\mathcal{S}, W \Omega_{\mathcal{S}}^{1})
			$
		as in \cite[Proposition 2.5.2 (ii)]{Ill83}.
		It admits a finite filtration whose successive subquotients are isomorphic to
		elementary dominoes
			$
				\underline{\underline{\mathrm{U}}}_{i_{1}},
				\dots,
				\underline{\underline{\mathrm{U}}}_{i_{m}}
			$
		\cite[\S 2.2.2]{Ill83}.
		The number $m$ of the elementary dominoes is
		the dimension of this domino in the sense of \cite[paragraph after Proposition 2.5.2]{Ill83},
		which is equal to $\mu_{A / K}$ by Remark \ref{sdef}, Proposition \ref{p:ShaBr}
		and \cite[\S 2.5, b2)]{Ill83}.
		
		As a more precise relation,
		can we read off the set of integers $i_{1}, \dots, i_{m}$
		from the Iwasawa theory of $A / K$?
		When $\mathcal{S}$ is a supersingular K3 surface (so $m = 1$),
		the number $i = i_{1} = i_{m}$ is called the Artin invariant and usually denoted by $\sigma_{0}$
		(\cite[\S 2.2.2 b)]{Ill83}, \cite[Equation (4.6)]{Art74b}),
		which satisfies the property that the discriminant of the intersection pairing
		on the geometric N\'eron-Severi group of $\mathcal{S}$ is $- p^{2 \sigma_{0}}$ (where $p \ne 2$).
		For a more general $\mathcal{S}$,
		the integers $i_{1}, \dots, i_{m}$ might be related to
		the geometric Cassels-Tate discriminant group $\delta_{\mathrm{CT}}$
		\cite[Theorem 3.4.1 (6g)]{Suz19} and
		the analogue of Artin's period map
		\cite[paragraphs after Theorem 3.4.1]{Suz19}.
		
		It is an interesting problem to see
		if the numbers $i_{1}, \dots, i_{m}$ define a nice stratification
		on the moduli of elliptic curves with given $\deg(\Delta)$ in Section \ref{su:generic},
		generalizing Artin's filtration on the moduli of supersingular K3 surfaces \cite[\S 7]{Art74b}.
	\end{enumerate}
\end{remark}

\section{The $\mu$-invariant for semistable abelian varieties} \label{s:ssav}

The aim of this section is to prove that
the $\mu$-invariant formula \eqref{e:muform} holds for semistable abelian varieties $A$
without any assumption on $\Sha$.

\subsection{Statement and the $p$-adic $L$-function}

\begin{theorem} \label{t:muss}
	Let $A$ be a semistable abelian variety over $K$.
	Then the formula \eqref{e:muform} holds for $A$.
\end{theorem}

We will prove this below.
Let $Q(\Lambda)$ be the fraction field of $\Lambda$.
Let $\mathcal{Z}$ be a finite set of places of $K$
such that $A$ has good reduction over $\mathcal{C} \setminus \mathcal{Z}$.
View the Frobenius element $\mathrm{Frob}_{q} \in \Gal(\F_{q,\infty}/\F_q)$
sending $x$ to $x^q$ as an element in $\Gamma$.
For each place $v$, set $\mathrm{Frob}_{v} = \mathrm{Frob}_{q}^{\deg(v)}$
and $q_{v} = q^{\deg(v)}$.
Let $P_{v}(t)$ be the Euler factor at $v$ of the $L$-function of $A$,
which is a polynomial in $t = q^{-s}$ with $\Z$-coefficients and constant term $1$.
Let $L_{A, \mathcal{Z}}(s)$ be the $L$-function of $A$ without the Euler factors at points of $\mathcal{Z}$,
so that
	\[
			L_{A, \mathcal{Z}}(s)
		=
			L_{A}(s) \cdot
			\prod_{v \in \mathcal{Z}}
				P_{v}(q_{v}^{-s})
		=
			\frac{
				P_{1}(q^{-s})
			}{
				P_{0}(q^{-s})
				P_{2}(q^{-s})
			}
			\prod_{v \in \mathcal{Z}}
				P_{v}(q_{v}^{-s}).
	\]
Put
	\begin{equation} \label{e:padicL}
			\mathcal{L}_{A/K^{(p)}_\infty}
		:=
			\frac{
				P_{1}(q^{-1} \mathrm{Frob}_{q}^{-1})
			}{
				P_{0}(q^{-1} \mathrm{Frob}_{q}^{-1})
				P_{2}(q^{-1} \mathrm{Frob}_{q}^{-1})
			}
			\prod_{v \in \mathcal{Z}}
				P_{v}(q_{v}^{-1} \mathrm{Frob}_{v}^{-1})
	\end{equation}
and regard it as an element in $Q(\Lambda)$.
It is actually a rational function of the variable $\mathrm{Frob}_{q}$ with $\Q$-coefficients
(i.e.\ an element of the rational function field $\Q(\mathrm{Frob}_{q})$)
and the substitution $\mathrm{Frob}_{q}^{-1} \mapsto q^{-s}$ turns it into
$L_{A, \mathcal{Z}}(s + 1)$.

\begin{proposition}
	The element $\mathcal{L}_{A/K^{(p)}_\infty}$ coincides with
	the $p$-adic $L$-function of \cite{LLTT16}.
\end{proposition}

\begin{proof}
	By \cite[\S 3.1.2]{LLTT16},
	the $p$-adic $L$-function is defined as
		\begin{equation}
			\prod_{i = 0}^{2}
				\det\nolimits_{\Lambda[1/p]} \bigl(
					1 - \Phi^{i}_{0} \otimes [\mathrm{Frob}_q^{-1}],
					P^{i}_{0} \otimes_{\Z_p} \Lambda[1/p]
				\bigr)^{(-1)^{i + 1}},
		\end{equation}
	where $P^{i}_{0} := \coh^{i}_{\mathrm{log-crys},c}(\mathcal C^\sharp/\Z_p,D(A))$
	endowed with its Frobenius operator $\Phi^{i}_0$
	(Here $D(A)$ is the log Dieudonn\'e crystal
	associated with the semistable abelian variety $A$ following \cite{KT03}
	and $\mathcal{C}^\sharp$ is the log-scheme with underlying scheme $\mathcal{C}$
	and log-structure induced by $\mathcal{Z}$).
	On the other hand, by \cite[\S 3.2.1]{KT03}, the right-hand side of \eqref{e:padicL} is
		\[
			\prod_{i = 0}^{2}
				\det\nolimits_{\Q_p}(
					1-\mathrm{Frob}_q^{-1} \Phi^{i}_0,
					P^{i}_0 \otimes \Q_p
				)^{(-1)^{i + 1}}.
		\]
\end{proof}

Then by \cite[Theorem 1.1 (2)]{LLTT16},
for a characteristic element $c_{A/K^{(p)}_\infty}$ of $X_{A / K}$,
	\begin{equation} \label{e:IMC}
			\mathcal{L}_{A/K^{(p)}_\infty}
		=
			u \cdot \star \cdot c_{A/K^{(p)}_\infty},
	\end{equation}
where $u\in\Lambda^{\times}$, and $\star$ is given as follow.
For each $v \in \mathcal{Z}$, let $T_{v}$ be the maximal torus of $\mathcal{A}_{v}^{0}$
and set $B_{v} = \mathcal{A}_{v}^{0} / T_{v}$
(which is an abelian variety over $k(v)$ since $A$ is semistable).
Let $g_{v} = \dim(B_{v})$.
Let $\{\beta_{i}^{(v)}\}$ be the eigenvalues of $\mathrm{Frob}_{v}$
on the $l$-adic Tate module of $B_{v}$ (for any $l \ne p$)
viewed as elements of $\bar{\Q}$ and hence $\bar{\Q}_{p}$.
(Here, as usual, the Tate module $\mathrm{T}_{l}(B_{v})$ means
$\mathrm{T}_{l}(B_{v}(\bar{\F}_{q}))$
with the natural action of the $q_{v}$-th power Frobenius $\mathrm{Frob}_{v}$.)

Let $\zeta_{1}, \dots, \zeta_{m}$ be the eigenvalues of $\mathrm{Frob}_{q}$
on the $p$-adic Tate module of $A(K_{\infty}^{(p)})$.
Then
	\[
			\star
		=
			\frac{
				q^{- \dim(A)(\deg(\mathcal{Z}) + g_{\mathcal{C}} - 1) - \deg(\mathcal{L})} \cdot
				\prod_{v \in \mathcal{Z}} \prod_{i = 1}^{2 g(v)}
					(\beta_{i}^{(v)} - \mathrm{Frob}_{v}^{-1})
			}{
				\prod_{i = 1}^{m}
					(1 - \zeta_{i}^{-1} \mathrm{Frob}_{q})
					(1 - \zeta_{i}^{-1} \mathrm{Frob}_{q}^{-1})
			}.
	\]
(The term $\delta$ in the original formula is equal to $\deg(\mathcal{L})$ here
by \cite[Proposition 3.3]{GS20} for example.)

\subsection{Local factors and the trace}

To prove the theorem, we will simplify both sides of \eqref{e:IMC}.
For each $v \in \mathcal{Z}$, let $t(v) = \dim(T_{v})$, so that $t(v) + g(v) = \dim(A)$.
Let $P_{v, \mathrm{tor}}(t)$ (resp.\ $P_{v, \mathrm{ab}}(t)$) be
the characteristic polynomial of $q_{v} \mathrm{Frob}_{v}^{-1}$
on $\mathrm{V}_{l}(T_{v})$ (resp.\ $\mathrm{V}_{l}(B_{v})$),
so that $P_{v}(t) = P_{v, \mathrm{tor}}(t) P_{v, \mathrm{ab}}(t)$.
We denote by $\equiv$ the equality in
	$
			Q(\Lambda \otimes_{\Z_{p}} \mathcal{O})^{\times}
		/
			(\Lambda \otimes_{\Z_{p}} \mathcal{O})^{\times}
	$,
where $\mathcal{O}$ is the ring of integers of some large enough finite extension of $\Q_{p}$.
Note that $Q(\Lambda)^{\times} / \Lambda^{\times}$ injects into
	$
			Q(\Lambda \otimes_{\Z_{p}} \mathcal{O})^{\times}
		/
			(\Lambda \otimes_{\Z_{p}} \mathcal{O})^{\times}
	$.

\begin{lemma}
	For any $v \in \mathcal{Z}$, we have
		\begin{gather*}
					P_{v, \mathrm{tor}}(q_{v}^{-1} \mathrm{Frob}_{v}^{-1})
				\equiv
					q_{v}^{- t(v)},
				\quad
					P_{v, \mathrm{ab}}(q_{v}^{-1} \mathrm{Frob}_{v}^{-1})
				\equiv
					q_{v}^{- g(v)}
					\prod_{i = 1}^{2 g(v)}
						(\beta_{i}^{(v)} - \mathrm{Frob}_{v}^{-1}),
			\\
					P_{0}(q^{-1} \mathrm{Frob}_{q}^{-1})
				\equiv
						q^{- \dim \Tr_{K / \F_{q}}(A)}
					\cdot
						\prod_{i = 1}^{m}
							(1 - \zeta_{i}^{-1} \mathrm{Frob}_{q}^{-1}),
			\\
					P_{2}(q^{-1} \mathrm{Frob}_{q}^{-1})
				\equiv
					\prod_{i = 1}^{m}
						(1 - \zeta_{i}^{-1} \mathrm{Frob}_{q}),
		\end{gather*}
\end{lemma}

\begin{proof}
	For $P_{v, \mathrm{tor}}$,
	let $\{\gamma_{j}^{(v)}\}$ be the eigenvalues of $\mathrm{Frob}_{v}$ on $\mathrm{V}_{l}(T_{v})$.
	Then $q_{v}^{-1} \gamma_{j}^{(v)}$ are roots of unity.
	We have
		\begin{align*}
					P_{v, \mathrm{tor}}(q_{v}^{-1} \mathrm{Frob}_{v}^{-1})
			&	=
					\prod_{j} (1 - (\gamma_{j}^{(v)})^{-1} \mathrm{Frob}_{v}^{-1})
			\\
			&	=
					q_{v}^{- t(v)}
					\prod_{j} (q_{v} -  q_{v} (\gamma_{j}^{(v)})^{-1} \mathrm{Frob}_{v}^{-1}).
		\end{align*}
	Each factor $q_{v} -  q_{v} (\gamma_{j}^{(v)})^{-1} \mathrm{Frob}_{v}^{-1}$ is a unit
	in $\Lambda \otimes_{\Z_{q}} \mathcal{O}$ for some $\mathcal{O}$.
	
	For $P_{v, \mathrm{ab}}$, we have
		\[
				P_{v, \mathrm{ab}}(q_{v}^{-1} \mathrm{Frob}_{v}^{-1})
			=
				\prod_{i} (1 - (\beta_{i}^{(v)})^{-1} \mathrm{Frob}_{v}^{-1})
			=
				\prod_{i} (\beta_{i}^{(v)})^{-1}
				\prod_{i} (\beta_{i}^{(v)} - \mathrm{Frob}_{v}^{-1}).
		\]
	The product $\prod_{i} (\beta_{i}^{(v)})$ is a Weil $q_{v}$-number of weight $2 g(v)$
	that is in $\Z$.
	Hence it is $\pm q_{v}^{g(v)}$.
	
	For $P_{0}$, first notice that by the Lang-N\'eron theorem (\cite[Theorem 7.1]{Con06}),
	the group $A(K_{\infty}^{(p)}) / \mathrm{Tr}_{K / \F_{q}}(A)(\F_{q, \infty})$
	is a finitely generated abelian group.
	Therefore the $p$-adic Tate module of $\mathrm{Tr}_{K / \F_{q}}(A)(\F_{q, \infty})$
	injects into that of $A(K_{\infty}^{(p)})$ with finite cokernel.
	Hence $\zeta_{i}$ are the eigenvalues of $\mathrm{Frob}_{q}$
	on the $p$-adic Tate module of $\mathrm{Tr}_{K / \F_{q}}(A)(\F_{q, \infty})$.
	In particular, they are $p$-adic units.
	Let $g_{0} = \dim \mathrm{Tr}_{K / \F_{q}}(A)$.
	
	By \cite[Proposition 2.9]{Tan14} and the proof of \cite[Lemma 2.2.1]{LLTT16},
	we know that $\zeta_{1}, \dots, \zeta_{m}$ are part of the eigenvalues
	$\zeta_{1}, \dots, \zeta_{m}, \dots, \zeta_{2 g_{0}}$
	of $\mathrm{Frob}_{q}$ on the $l$-adic Tate module of $\mathrm{Tr}_{K / \F_{q}}(A)$ (for any $l \ne p$)
	and
		\begin{equation} \label{e:pandl}
				\prod_{i = 1}^{m}(\zeta_{i} - \mathrm{Frob}_{q}^{-1})
			\equiv
				\prod_{i = 1}^{2 g_{0}}(\zeta_{i} - \mathrm{Frob}_{q}^{-1}).
		\end{equation}
	Together with the relation between $P_{0}$ and $\mathrm{Tr}_{K / \F_{q}}(A)$
	we saw in Section \ref{s:Lfun}, we have
		\begin{align*}
			&		P_{0}(q^{-1} \mathrm{Frob}_{q}^{-1})
				=
					\prod_{i = 1}^{2 g_{0}}
						(1 - \zeta_{i}^{-1} \mathrm{Frob}_{q}^{-1})
				\equiv
					q^{- g_{0}}
					\prod_{i = 1}^{2 g_{0}}
						(\zeta_{i} - \mathrm{Frob}_{q}^{-1})
			\\
			&	\equiv
					q^{- g_{0}}
					\prod_{i = 1}^{m}
						(\zeta_{i} - \mathrm{Frob}_{q}^{-1})
				\equiv
					q^{- g_{0}}
					\prod_{i = 1}^{m}
						(1 - \zeta_{i}^{-1} \mathrm{Frob}_{q}^{-1}).
		\end{align*}
	
	For $P_{2}$, first notice that \eqref{e:pandl} also holds
	with $\mathrm{Frob}_{q}$ replaced by $\mathrm{Frob}_{q}^{-1}$
	since $\mathrm{Frob}_{q} \leftrightarrow \mathrm{Frob}_{q}^{-1}$ is
	an automorphism of $\Lambda$.
	Also, $\{\zeta_{i}\}_{i = 1}^{2 g_{0}} = \{q / \zeta_{i}\}_{i = 1}^{2 g_{0}}$
	(counting the multiplicities)
	by the functional equation for the zeta function of $\mathrm{Tr}_{K / \F_{q}}(A)$.
	Using these, we have
		\begin{align*}
			&		P_{2}(q^{-1} \mathrm{Frob}_{q}^{-1})
				=
					\prod_{i = 1}^{2 g_{0}}
						(1 - q \zeta_{i}^{-1} \mathrm{Frob}_{q}^{-1})
				=
					\prod_{i = 1}^{2 g_{0}}
						(1 - \zeta_{i} \mathrm{Frob}_{q}^{-1})
			\\
			&	\equiv
					\prod_{i = 1}^{2 g_{0}}
						(\zeta_{i} -  \mathrm{Frob}_{q})
				\equiv
					\prod_{i = 1}^{m}
						(\zeta_{i} -  \mathrm{Frob}_{q})
				\equiv
					\prod_{i = 1}^{m}
						(1 -  \zeta_{i}^{-1} \mathrm{Frob}_{q}).
		\end{align*}
\end{proof}

Using this lemma, we see that \eqref{e:IMC} is equivalent to
	\[
			c_{A / K_{\infty}^{(p)}}
		\equiv
			q^{\deg(\mathcal{L}) + \dim(A) (g_{\mathcal{C}} - 1) + \dim \Tr_{K / \F_{q}}(A)}
			\cdot P_{1}(q^{-1} \mathrm{Frob}_{q}^{-1}).
	\]
The exact $p$-power factor of $c_{A / K_{\infty}^{(p)}}$ is $q^{\mu_{A / K}}$ by definition.
The exact $p$-power factor of $P_{1}(q^{-1} \mathrm{Frob}_{q}^{-1})$ is $q^{- \theta_{A}}$ by definition.
Comparing the exact $p$-power factors of both sides, we obtain Theorem \ref{t:muss}.

\section{The $\mu$-invariant for elliptic curves}\label{s:muelliptic}
For the rest of the paper, we deal with the case where $\dim A=1$.
Let $\Delta:=\sum_{v\in \mathcal{C}} \delta_v\cdot v$ the global discriminant of $A/K$, where $\delta_v$ is the valuation of the minimal discriminant of $A$ at $v$ (\cite{Sil86}, VII.1).

When it becomes necessary to emphasize the roles played by $A$ and $K$,
we shall use $\Delta_{A/K}$ to denote it.
By \cite[Example 2.2]{Con06},
we know $\Tr_{K / \F_{q}}(A) = 0$ unless $A$ is constant,
in which case $\Tr_{K / \F_{q}}(A)=A$ (\cite[Theorem 6.4 (1)]{Con06}).
In the case $A$ is non-constant, we have $P_{0}(t) = P_{2}(t) = 1$,
so we simply write $P(t) = P_{1}(t)$.

\subsection{The $\mu$-invariant formula simplified}
We have $\deg(\mathcal{L}) = \deg(\Delta) / 12$
by \cite[\S 2.2.1, Footnote 1]{LLTT16}.
Therefore the formula \eqref{e:muform}
(true for $A$ by Corollary \ref{c:MilUlm}) and the inequality \eqref{e:upperbound}
become the following statement.

\begin{theorem} \label{t:tan}
		\[
				\mu_{A / K}
			=
					\frac{\deg(\Delta)}{12} + g_{\mathcal{C}} - 1 - \theta
				\le
					\frac{\deg(\Delta)}{12} + g_{\mathcal{C}} - 1
		\]
	if $A$ is non-constant and
		\[
				\mu_{A / K}
			=
					g_{\mathcal{C}}  - \theta
				\le
					g_{\mathcal{C}}
		\]
	if $A$ is constant.
\end{theorem}

Let $\pi \colon \mathcal{S} \to \mathcal{C}$ be the minimal elliptic surface
associated with $A / K$.
By \cite[Chapter 7, Lemma 14]{Fri98} (adapted to positive characteristic),
	\[
				\dim H^{1}(\mathcal{S}, \mathcal{O}_{\mathcal{S}})
			=
				g_{\mathcal{C}},
		\quad
				\dim H^{2}(\mathcal{S}, \mathcal{O}_{\mathcal{S}})
			=
				\deg(\mathcal{L}) + g_{\mathcal{C}} - 1
	\]
if $\mathcal{L}$ is non-trivial, and
	\[
				\dim H^{1}(\mathcal{S}, \mathcal{O}_{\mathcal{S}})
			=
				g_{\mathcal{C}} + 1,
		\quad
				\dim H^{2}(\mathcal{S}, \mathcal{O}_{\mathcal{S}})
			=
				g_{\mathcal{C}}
	\]
if $\mathcal{L}$ is trivial.
The canonical bundle formula
$\Omega_{\mathcal{S}}^{2} \cong \pi^{\ast}(\Omega_{\mathcal{C}}^{1} \otimes \mathcal{L})$
implies that the Kodaira dimension of $\mathcal{S}$ is $- \infty$, $0$, $1$
if $\deg(\mathcal{L}) + 2 g_{\mathcal{C}} - 2 < 0$,
$\deg(\mathcal{L}) + 2 g_{\mathcal{C}} - 2 = 0$,
$\deg(\mathcal{L}) + 2 g_{\mathcal{C}} - 2 > 0$, respectively
(see for instance \cite[Lecture 3, Proposition 4.3]{Ulm11} and the paragraphs after%
\footnote{We take the opportunity to mention a mistake in \cite[Lecture 3 (6.3)]{Ulm11}.
The correct formula can be found in
\cite[Equation (2.9) and Lemma 6]{Shi92}.}).

Let $\bar{r}$ be the geometric Mordell-Weil rank of $A$,
i.e.\ the rank of
$A(K \bar{\F}_{q}) / \Tr_{K / \F_{q}}(A)(\bar{\F}_{q})$.
Recall from Definition \ref{d:theta} that $a = a_{A}$ denotes the degree of the polynomial $P_{1}(t)$.
Then \eqref{e:MWSha} shows that $\bar{r} \le a$.

\begin{proposition} \label{p:KodOne}
	Assume that the minimal elliptic surface $\mathcal{S} \to \mathcal{C}$ over $\F_{q}$ associated with $E / K$
	has Kodaira dimension $1$
	and that $\bar{r} = a$.
	Then $\mu_{A / K} > 0$.
\end{proposition}

\begin{proof}
	By the proof of Proposition \ref{p:Lcomp}
	and by \eqref{e:NSBr}, \eqref{e:MWSha},
	we know that $a - \bar{r} = \deg P_{\mathcal{S}, 2} - \rho_{\mathcal{S}_{\bar{\F}_{q}}}$,
	where $\rho_{\mathcal{S}_{\bar{\F}_{q}}}$ is the rank of $\mathrm{NS}(\mathcal{S}_{\bar{\F}_{q}})$.
	Hence the assumption implies that $\rho_{\mathcal{S}_{\bar{\F}_{q}}} = \deg P_{\mathcal{S}, 2}$,
	i.e.\ $\mathcal{S}$ is Shioda supersingular.
	In particular (cf.\ the paragraph after \cite[Definition 9.11]{Lie13}),
	it is Artin supersingular, i.e.\ $\lambda_{\mathcal{S}, i} = 1$ for all $i$,
	or $\theta_{\mathcal{S}} = 0$.
	By Proposition \ref{p:Lcomp},
	we know that $\theta_{A} = 0$.
	We have $\deg(\mathcal{L}) + 2 g_{\mathcal{C}} - 2 > 0$ as above.
	Hence if $g_{\mathcal{C}} = 0$, then $\deg(\mathcal{L}) \ge 3$ and $A$ is non-constant;
	if $g_{\mathcal{C}} = 1$, then $\deg(\mathcal{L}) \ge 1$ and $A$ is non-constant;
	or else $g_{\mathcal{C}} \ge 2$.
	Applying Theorem \ref{t:tan}, we know $\mu_{A / K} > 0$.
\end{proof}

\subsection{The $j$-invariant}

Let $j_A$ denote the $j$-invariant of $A$.
Recall the Szpiro difference $d = d_{A}$
from Definition \ref{d:Sz}.

\begin{theorem}[Pesenti-Szpiro {\cite[Th\'eor\`eme 0.1]{PS00}}] \label{c:szpiro}
	If $j_A$ is not a $p$th power or $A$ is isotrivial, then $d\geq 0$.
\end{theorem}

Since $\mu_{A / K} = b - d$ by Corollary \ref{c:bminusd},
this theorem allows us to bound $\mu_{A / K}$ purely in terms of the $L$-function:

\begin{corollary}
	If $j_A$ is not a $p$th power or $A$ is isotrivial, then $\mu_{A / K} \le b$.
\end{corollary}

For the rest of this section, we further assume that
the elliptic curve $A$ is \emph{non-isotrivial}.
We will study the behavior of the $j$-invariant and the $\mu$-invariant under isogeny.

\begin{lemma}\label{l:j} The following conditions are equivalent
\begin{enumerate}
\item[(a)] The group scheme $A_p$ has a non-trivial $\acute{\text{e}}$tale subgroup.
\item[(b)] The $j$-invariant $j_A$ is a $p$th power of an element in $K$.
\end{enumerate}
\end{lemma}
\begin{proof}
If $E$ is a non-trivial $\acute{\text{e}}$tale subgroup scheme of $A_p$, then it is of order $p$ and $B:=A/E$
is isogenous to $A$ by the natural homomorphism $\varphi:A\longrightarrow B$.
Because the group scheme $C= \varphi(A_p)$ is connected and hence $B/C=A$ is isomorphic to $B^{(p)}$,
$\varphi$ is isomorphic to the Verschiebung. This shows
$j_A=j_{B^{(p)}}$ is a $p$th power.
Conversely, if $j_A=j_0^p$ and let $B/K$ be an elliptic curve with $j_B=j_0$,
then $j_A=j_{B^{(p)}}$ and hence $A$ and $B^{(p)}$ are isomorphic over the separable closure of $K$ \cite[III.1.4(b),A.1.2(b)]{Sil86}.
Since $B^{(p)}_p$ contains a non-trivial $\acute{\text{e}}$tale subgroup scheme, so does $A$.
\end{proof}

\begin{lemma}\label{l:isogeny}
Every non-isotrivial elliptic curve is isogenous to an elliptic curve with $j$-invariant not a $p$-th power.
\end{lemma}
\begin{proof} Let $\mathcal E$ be the maximal  $\acute{\text{e}}$tale subgroup of $A_{p^\infty}$. Then $\mathcal E$ is finite and
$(A/\mathcal E)_p$ has no non-trivial $\acute{\text{e}}$tale subgroup.
\end{proof}

\subsection{The isogeny class}\label{su:isogen}

Consider the connected-\'etale sequence
$$0\longrightarrow A_{p^n}^0\longrightarrow A_{p^n}\longrightarrow \pi_0(A_{p^n})\longrightarrow 0,$$
where $A_{p^n}^0$ and $\pi_0(A_{p^n})$ are Cartier dual to each other. Since $\pi_0(A_{p^n})$ is of $p$-rank $1$, there exists
a unique filtration
$$0\subset G_1\subset \cdots \subset G_{n-1}\subset G_{n}=\pi_0(A_{p^n})$$
with $G_i$ of order $p^i$, and by duality
the filtration
$$0\subset H_1\subset \cdots \subset H_{n-1}\subset H_{n}=A_{p^n}^0$$
with $H_i$ of order $p^i$. If $\varphi:A\longrightarrow B$
is an isogeny of degree $p^n$, then $E:=\ker \varphi\subset A_{p^n}$ with $E^0\subset A_{p^n}^0$ and $\pi_0(E)\subset \pi_0(A_{p^n})$.
By the above filtrations, $\varphi$ is decomposed into a sequence
$$
\xymatrix{A \ar[r]^-{\phi_1} & A_1 \ar[r]^-{\phi_2} &\cdots \ar[r]^-{\phi_c} & A_c \ar[r]^-{\psi_1} & B_1\ar[r]^-{\psi_2} & \cdots
\ar[r]^{\psi_e} & B_e \ar[r]^-{\psi_{e+1}}
 & B}
$$
where each $\phi_i$ is a Frobenius, $\psi_j$ a Verschiebung, all of degree $p$. Suppose both $c$ and $e$ are positive.
Then up to isomorphisms, the segment $\xymatrix{A_{c-1}\ar[r]^-{\phi_c} & A_c \ar[r]^-{\psi_1} & B_1}$ equals to the multiplication by $p$.
Hence $\varphi$ can be written as $\varphi'\circ [p]$ with $\varphi'$ of less degree than $\varphi$. By repeating this procedure, we prove the
following lemma.

\begin{lemma}\label{l:piso}
If $\varphi:A\longrightarrow B$ is an isogeny of degree $p^n$ for some $n$, then either $B=A^{(p^m)}$ or $A=B^{(p^m)}$
for some $m$.
\end{lemma}

Suppose $A/K$ has semi-stable reduction everywhere. Then
\begin{equation}\label{e:delta}
\deg \Delta_{A^{(p^m)}/K}=p^m \deg\Delta_{A/K},
\end{equation}
while $N_{A^{(p^m)}/K}=N_{A/K}$ and the Hasse-Weil $L$-function for both curves are the same. Hence Theorem \ref{t:tan} says
\begin{equation}\label{e:apm}
\mu_{A^{(p^m)}/K}=\mu_{A/K}+(p^m-1)\cdot \frac{\deg \Delta_{A/K}}{12}.
\end{equation}

\begin{lemma}\label{l:l}
If $\varphi:B\longrightarrow A$ is an isogeny of degree $l$ with $(l,p)=1$, then
$$\deg \Delta_{B/K}=\deg \Delta_{A/K}.$$
Hence $\mu_{B/K}=\mu_{A/K}$.
\end{lemma}
\begin{proof}
Choose an invariant differential $\omega'$ of $A/K$ and denote $\omega:=\varphi^*\omega'$.
For each place $v$, let $\omega_v$ and $\omega_v'$ respectively be local N$\acute{\text{e}}$ron differentials
of $B$ and $A$.  Then we have (see \cite[\S 1.2]{Tan95})
$$\deg \Delta_{B/K}=12\sum_v \ord_v{\frac{\omega}{\omega_v}}\cdot \deg v,$$
and
$$\deg \Delta_{A/K}=12\sum_v \ord_v{\frac{\omega'}{\omega'_v}}\cdot \deg v.$$
Let $\psi:A\longrightarrow B$ denote the dual isogeny. Write $\varphi^*\omega'_v=\alpha_v\cdot \omega_v$, $\psi^*\omega_v=\beta_v\cdot\omega'_v$, $\alpha_v,\beta_v\in\O_v$. But, since $\psi\circ\varphi=\left[l\right]$, we have $\alpha_v\cdot\beta_v=l$. Hence $\alpha_v\in\O_v^*$.
%$\ord_v\frac{\varphi^*\omega'_v}{\omega_v}=0$.
Consequently,
$$\ord_v\frac{\omega'}{\omega'_v}=\ord_v\frac{\varphi^*\omega'}{\varphi^*\omega'_v}=\ord_v\frac{\omega}{\alpha_v\cdot\omega_v}=\ord_v\frac{\omega}{\omega_v},$$
so $\deg \Delta_{B/K}=\deg \Delta_{A/K}$.
\end{proof}
We summarize the above discussion below.
\begin{proposition}\label{p:isoclass}
Among an isogeny class of elliptic curves with semi-stable reduction everywhere, those curves with $j$-invariant not a $p$-th power have minimal $\mu$-invariant.
\end{proposition}

The elliptic curve $A$ in Proposition \ref{p:Shio} has $j$-invariant not a $p$-th power.
Hence by Proposition \ref{p:isoclass},
we cannot replace $A$ by an isogenous elliptic curve
for which the $\mu$-invariant is zero.
This should be compared with Greenberg's conjecture \cite[Conjecture 1.11]{Gre99}
on $\mu$-invariants of elliptic curves over $\Q$.
To recall this conjecture, let $E$ be an elliptic curve over $\Q$.
Assume that $\Sel_{p^{\infty}}(E / \Q^{cyc})^{\vee}$ is torsion over $\Z_{p}[[\Gal(\Q^{cyc} / \Q)]]$.
Then Greenberg's conjecture claims the existence of a $\Q$-isogenous elliptic curve $E'$
such that $\Sel_{p^{\infty}}(E' / \Q^{cyc})^{\vee}$ has $\mu$-invariant zero.

\subsection{Elliptic curves with trivial $\mu$-invariant}\label{su:trivmu}

In this subsection, we produce infinitely many elliptic curves with trivial $\mu$-invariant.

Assume that $\mathrm{char}(K) \ne 2$ and $A/K$ is defined by the Legendre form
(see \cite[\S III.1.7]{Sil86})
\begin{equation}\label{e:legend}
Y^2=X(X-1)(X-\lambda),
\end{equation}
with $\lambda\in K$, so that $F_0:= \F_q(\lambda)\subset K$. The equation \eqref{e:legend} defines an elliptic curve $E_0/F_0$ such that $A/K$ is just its base change to $K$. Put $t^2=\lambda$, $F=\F_q(t)$ and let $E/F$ be the elliptic curve defined by
\begin{equation}\label{e:legsemi}
 Y^2=X(X-1)(X-t^2)
\end{equation}
Then $E/F$ is the base change of $E_0$ to $F$. Let $K'=KF$, then $A/K'$ is the base change of $E$ to $K'$.

\begin{definition}\label{d:Legendreptype}
An elliptic curve is of Legendre $p$-type, if
\begin{enumerate}
\item $char(K)\neq 2$ and $A$ has the Legendre invariant $\lambda\in K$.
\item $K/F_0$, where $F_0=\F_q(\lambda)$, is a Galois $p$-extension unramified at supersingular places of $E_{0} / F_{0}$.
\end{enumerate}
\end{definition}

\begin{theorem}\label{t:Legendretype} If $A/K$ is of Legendre $p$-type, then $\mu_{A/K}=0$.
\end{theorem}

\begin{proof} By Lemma \ref{l:kl}, it is enough to show $\mu_{A/K'}=0$.
Here $K'=KF$ and $F=\F_q(t)$ with $t^2=\lambda$.
Let $E_{0} / F_{0}$ and $E / F$ be the elliptic curves
defined by \eqref{e:legend} and \eqref{e:legsemi}, respectively.
By the condition (2), if a place of $F$ ramifies in $K'$,
then the corresponding place of $F_{0}$ is
either an ordinary or additive place for $E_{0} / F_{0}$.
The only place where $E_{0} / F_{0}$ has additive reduction is where $\lambda = \infty$
(see the proof of Lemma \ref{l:ss} below).
The corresponding place $t = \infty$ of $F$ is an ordinary place for $E/F$ (see Lemma \ref{l:ss}).
Therefore $K' / F$ is a Galois $p$-extension unramified outside ordinary places of $E / F$.
Now Theorem \ref{t:ocht} says we only need to prove $\mu_{E/F}=0$.
Lemma \ref{l:ss} below says $E/F$ has semi-stable reduction everywhere, and by simple computation (see \S\ref{sub:example}),
we find $\deg(\Delta_{E / F})=12$, and hence $\mu_{E / F} = 0$
by Theorem \ref{t:tan}.
Thus $\mu_{A/K} = 0$.
\end{proof}

\begin{remark}
	Let $A / K$ be an elliptic curve satisfying the condition (1)
	such that $\mu_{A / K} = 0$.
	The previous theorem is a type of results
	saying $\mu_{A / L} = 0$ for
	some extension $L / K$ unramified at supersingular places of $A / K$.
	But $L / K$ being ramified at a supersingular place does not always imply that $\mu_{A / L} > 0$.
	For example, let $K=\F_3(t)$ and $A$ defined by $Y^2=X(X-1)(X-t^2)$.
	There is one supersingular place which is the zero of $(t^2+1)$.
	The extension $L=K(s)$ of $K$ with $s^2=1+t^2$ is a quadratic extension ramified at $(t^2+1)$.
	However, Magma gives a trivial $\mu$-invariant for $A/L$ (see \S\ref{su:high} below).
\end{remark}

\section{$\mu = 0$ for generic elliptic curves} \label{su:generic}
In this section, we show that ``generic'' elliptic curves over $K$ with $p > 3$
in a certain precise sense have trivial $\mu$-invariant
(Theorem \ref{t:GenVan}).
This partially extends the result \cite[Lemma (5.10)]{Art74b}
that the generic fibers of general Weierstrass fibrations which are K3 have $\mu = 0$.
We do not assume that the elliptic curves are non-isotrivial or semistable.

\subsection{An elliptic curve having $\mu = 0$ with given $\deg(\Delta)$}

We first produce some (isotrivial) elliptic curves having $\mu = 0$ with given $\deg(\Delta)$.

\begin{proposition} \label{p:NonEmp}
	Assume $p > 2$.
	Let $n$ be a non-negative integer.
	Then there exist a finite constant extension $K'$ of $K$
	and an elliptic curve $E / K'$ such that $\deg(\Delta_{E / K'}) = 12 n$ and $\mu_{E / K'} = 0$.
\end{proposition}

\begin{proof}
	Let $K'$ be a finite constant extension of $K$
	having $2 n$ places $v_{1}, \dots, v_{2 n}$ of degree $1$.
	We first show that by replacing $K'$
	by a larger finite constant extension of $K$,
	we can choose an element $f \in K'$
	such that $v_{i}(f)$ is odd for any $i$
	and $v(f)$ is even for any other place $v$ of $K'$.
	Consider the exact sequence
		\[
				 (K \bar{\F}_{q})^{\times}
			\to
				\bigoplus_{v} \Z
			\to
				\Pic(\mathcal{C}_{\bar{\F}_{q}})
			\to
				 0,
		\]
	where the middle direct sum is over all places of $K \bar{\F}_{q}$,
	the first map is the sum of the valuations and
	the second map is $1 \mapsto v$ for each summand corresponding to a place $v$.
	Since $\Pic^{0}(\mathcal{C}_{\bar{\F}_{q}})$ is
	the group of geometric points of the Jacobian of $\mathcal{C}$ and hence divisible,
	this gives an exact sequence
		\[
			 (K \bar{\F}_{q})^{\times} \to \bigoplus_{v} \Z / 2 \Z \to \Z / 2 \Z \to 0,
		\]
	where the second map is the summation map.
	Consider the element of $\bigoplus_{v} \Z / 2 \Z$
	such that the component corresponding to each $v_{i}$ is $1$
	and the other components are zero.
	This element belongs to the kernel of the summation map.
	Hence there exists an element $f \in (K \bar{\F}_{q})^{\times}$
	such that $v_{i}(f) \equiv 1 \mod 2$ for any $i$ and $v(f) \equiv 0 \mod 2$ for any other $v$.
	Replace $K'$ by a larger finite constant extension of $K$ so that $f$ is in $K'$.
	This choice of $K'$ and $f$ does the job.

	Let $E_{0} \colon y^{2} = x^{3} + a x^{2} + b x + c$ be
	any constant ordinary elliptic curve over $K'$.
	Let $E \colon f y^{2} = x^{3} + a x^{2} + b x + c$ be
	its quadratic twist by $f$.
	This $E$ has $\mu = 0$ by \cite[Theorem 1.8 (1)]{OT09}.
	We show $\deg(\Delta_{E / K'}) = 12 n$.
	Let $v$ be any place of $K'$.
	If $v(f)$ is even and hence $f = u \pi^{m}$ for some unit $u$ at $v$,
	uniformizer $\pi$ at $v$ and integer $m$,
	then $E$ is isomorphic to $u y^{2} = x^{3} + a x^{2} + b x + c$ over $K'_{v}$,
	which has good reduction at $v$.
	If $v(f)$ is odd, then similarly, $E$ is isomorphic to
	$y^{2} = x^{3} + \pi a x^{2} + \pi^{2} b x + \pi^{3} c$ over $K_{v}$
	for some uniformizer $\pi$ at $v$.
	Its discriminant is the discriminant of $E_{0}$ times $\pi^{6}$.
	Hence this equation is minimal at $v$ and $v(\Delta_{E / K'}) = 6$.
	Therefore, by the choice of $f$,
	we know that $\deg(\Delta_{E / K'}) = 6 \cdot 2 n = 12 n$.
\end{proof}

\subsection{Moduli of elliptic curves with given $\deg(\Delta)$}

To precisely define ``generic'' elliptic curves,
we need a parameter space for elliptic curves over $K$ with given $\deg(\Delta)$.
We follow \cite{Kas77}.
See also \cite{Mir81} and \cite{Sei87},
which work over positive characteristic as well.

For a perfect field extension $k / \F_{q}$ and
an abelian variety $A$ over the fraction field of $K \otimes_{\F_{q}} k$
with N\'eron model $\mathcal{A}$ over $\mathcal{C}_{k}$,
we define $\mu_{A}$ to be the dimension of the group scheme
$\mathbf{H}^{1}(\mathcal{C}_{k}, \mathcal{A})$ over $k$.
If $A$ is an elliptic curve and
the associated minimal elliptic surface $\mathcal{S}$ has smooth Picard scheme,
then this $\mu_{A}$ is equal to the dimension of the unipotent part
of the formal Brauer group of $\mathcal{S}$
by Remark \ref{sdef} and Proposition \ref{p:ShaBr}
(which are still true with $\F_{q}$ replaced by $k$ by the same proof).
When $k = \bar{\F}_{q}$, we denote $\bar{\mathcal{C}} = \mathcal{C}_{\bar{\F}_{q}}$.
In what follows, any fiber product is taken over $\bar{\F}_{q}$ when ``$\times$'' has no specified base.
For any integer $n$,
let $\Pic_{\bar{\mathcal{C}} / \bar{\F}_{q}}^{n}$ be the degree $n$ component of
the Picard scheme $\Pic_{\bar{\mathcal{C}} / \bar{\F}_{q}}$,
which parametrizes the line bundles of degree $n$.
Let $\mathcal{P}$ be a Poincar\'e sheaf on
$\bar{\mathcal{C}} \times \Pic_{\bar{\mathcal{C}} / \bar{\F}_{q}}$
(\cite[Exercise 4.3]{Kle05}, \cite[8.2/4]{BLR90})
(rigidified at some closed point of $\bar{\mathcal{C}}$).
Its restriction to $\bar{\mathcal{C}} \times \Pic_{\bar{\mathcal{C}} / \bar{\F}_{q}}^{n}$
for any $n$ is denoted by $\mathcal{P}_{n}$.
Denote the second projection
	$
			\bar{\mathcal{C}} \times \Pic_{\bar{\mathcal{C}} / \bar{\F}_{q}}^{n}
		\twoheadrightarrow
			\Pic_{\bar{\mathcal{C}} / \bar{\F}_{q}}^{n}
	$
by $\mathrm{pr}_{2}$.

If $n > 2 g_{\mathcal{C}} - 2$,
then for any line bundle $\mathscr{L}$ of degree $n$ on $\bar{\mathcal{C}}$
(that is, any element of $\Pic_{\bar{\mathcal{C}} / \bar{\F}_{q}}^{n}(\bar{\F}_{q})$),
the dimension of $\coh^{0}(\bar{\mathcal{C}}, \mathscr{L})$ is $n + 1 - g_{\mathcal{C}}$
by the Riemann-Roch theorem,
which does not depend on $\mathscr{L}$.
Hence the direct image sheaf $\mathrm{pr}_{2 \ast} \mathcal{P}_{n}$ on
$\Pic_{\bar{\mathcal{C}} / \bar{\F}_{q}}^{n}$
has a constant rank over all points of $\Pic_{\bar{\mathcal{C}} / \bar{\F}_{q}}^{n}$.
Therefore Grauert's theorem (\cite[III, Corollary 12.9]{Har77}) implies that
$\mathrm{pr}_{2 \ast} \mathcal{P}_{n}$ is a locally free sheaf
whose fiber at any point $\mathscr{L}$ of $\Pic_{\bar{\mathcal{C}} / \bar{\F}_{q}}^{n}$ is given by
$\coh^{0}(\bar{\mathcal{C}}, \mathscr{L})$.
Let
	$
			\mathbf{V}((\mathrm{pr}_{2 \ast} \mathcal{P}_{n})^{\vee})
		\twoheadrightarrow
			\Pic_{\bar{\mathcal{C}} / \bar{\F}_{q}}^{n}
	$
be the total space of $\mathrm{pr}_{2 \ast} \mathcal{P}_{n}$
(where $\vee$ denotes the dual bundle),
which is the relative $\spec$ of the symmetric algebra of
$(\mathrm{pr}_{2 \ast} \mathcal{P}_{n})^{\vee}$.
The closed points of $\mathbf{V}((\mathrm{pr}_{2 \ast} \mathcal{P}_{n})^{\vee})$ are
given by pairs $(\mathscr{L}, s)$,
where $\mathscr{L}$ is a line bundle of degree $n$ on $\bar{\mathcal{C}}$
and $s \in \coh^{0}(\bar{\mathcal{C}}, \mathscr{L})$.

Assume $p > 3$.
Let $n > (g_{\mathcal{C}} - 1) / 2$ be an integer.
Let $\mathcal{P}_{n}^{\otimes 4}$ and $\mathcal{P}_{n}^{\otimes 6}$
be the tensor powers of $\mathcal{P}_{n}$,
which are line bundles on $\Pic_{\bar{\mathcal{C}} / \bar{\F}_{q}}^{n}$.
Since $6 n \ge 4 n > 2 g_{\mathcal{C}} - 2$,
we have a vector bundle
	$
		\mathbf{V}((
			\mathrm{pr}_{2 \ast}(
				\mathcal{P}_{n}^{\otimes 4} \oplus \mathcal{P}_{n}^{\otimes 6}
			)
		)^{\vee})
	$
over $\Pic_{\bar{\mathcal{C}} / \bar{\F}_{q}}^{n}$ as above.
The closed points of
	$
		\mathbf{V}((
			\mathrm{pr}_{2 \ast}(
				\mathcal{P}_{n}^{\otimes 4} \oplus \mathcal{P}_{n}^{\otimes 6}
			)
		)^{\vee})
	$
are given by triples $(\mathscr{L}, g_{2}, g_{3})$,
where $\mathscr{L}$ is a line bundle of degree $n$ on $\bar{\mathcal{C}}$
and $g_{2} \in \coh^{0}(\bar{\mathcal{C}}, \mathscr{L}^{\otimes 4})$,
$g_{3} \in \coh^{0}(\bar{\mathcal{C}}, \mathscr{L}^{\otimes 6})$.
Following the paragraphs after \cite[Theorem 1]{Kas77},
we make the following definition:

\begin{definition}
	Define $Y(n, \bar{\mathcal{C}})$ to be the set of triples $(\mathscr{L}, g_{2}, g_{3})$ as above
	satisfying the following two conditions:
	\begin{itemize}
		\item
			$4 g_{2}^{3} - 27 g_{3}^{2} \ne 0$ as an element of
			$\coh^{0}(\bar{\mathcal{C}}, \mathscr{L}^{\otimes 12})$.
		\item
			For every closed point $v$ of $\bar{\mathcal{C}}$,
			$\min(3 \ord_{v}(g_{2}), 2 \ord_{v}(g_{3})) < 12$.
	\end{itemize}
\end{definition}

Then $Y(n, \bar{\mathcal{C}})$ is an open subset of the set of closed points of
	$
		\mathbf{V}((
			\mathrm{pr}_{2 \ast}(
				\mathcal{P}_{n}^{\otimes 4} \oplus \mathcal{P}_{n}^{\otimes 6}
			)
		)^{\vee})
	$.
Adding generic points, we obtain an open subscheme of
	$
		\mathbf{V}((
			\mathrm{pr}_{2 \ast}(
				\mathcal{P}_{n}^{\otimes 4} \oplus \mathcal{P}_{n}^{\otimes 6}
			)
		)^{\vee})
	$,
which we denote by the same symbol $Y(n, \bar{\mathcal{C}})$ by abuse of notation.
Below we only consider $Y(n, \bar{\mathcal{C}})$ as a scheme.

Consider the action of $\mathbb{G}_{m, \Pic_{\bar{\mathcal{C}} / \bar{\F}_{q}}^{n}}$
on
	$
		\mathbf{V}((
			\mathrm{pr}_{2 \ast}(
				\mathcal{P}_{n}^{\otimes 4} \oplus \mathcal{P}_{n}^{\otimes 6}
			)
		)^{\vee})
	$
minus the zero section given by
$\lambda \cdot (\mathscr{L}, g_{2}, g_{3}) = (\mathscr{L}, \lambda^{4} g_{2}, \lambda^{6} g_{3})$.
The quotient of
	$
		\mathbf{V}((
			\mathrm{pr}_{2 \ast}(
				\mathcal{P}_{n}^{\otimes 4} \oplus \mathcal{P}_{n}^{\otimes 6}
			)
		)^{\vee})
	$
minus the zero section by this action of $\mathbb{G}_{m, \Pic_{\bar{\mathcal{C}} / \bar{\F}_{q}}^{n}}$
is the (total space of) weighted projective bundle associated with the vector bundle
$\mathrm{pr}_{2 \ast}(\mathcal{P}_{n}^{\otimes 4} \oplus \mathcal{P}_{n}^{\otimes 6})$
over $\Pic_{\bar{\mathcal{C}} / \bar{\F}_{q}}^{n}$
with weight $(4, 6)$ (\cite[\S 1]{Dol82}),
which is a universal geometric quotient (\cite[\S 1.2.1]{Dol82}).
The subscheme $Y(n, \bar{\mathcal{C}})$ is stable under the action of
$\mathbb{G}_{m, \Pic_{\bar{\mathcal{C}} / \bar{\F}_{q}}^{n}}$.

\begin{definition}
	Define $X(n, \bar{\mathcal{C}})$ to be the quotient of $Y(n, \bar{\mathcal{C}})$
	by this action of $\mathbb{G}_{m, \Pic_{\bar{\mathcal{C}} / \bar{\F}_{q}}^{n}}$.
\end{definition}

Let $k$ be an algebraically closed extension field of $\bar{\F}_{q}$.
By \cite[Theorem 2]{Kas77}
(which also works in any characteristic $p > 3$, since
minimal Weierstrass equations still have only rational double points by
\cite[Chapter 9, Theorem 4.35 (a)]{Liu02} and duality%
\footnote{The exact statement can also be found in Corollary 8.4 of Conrad's notes
``Minimal models for elliptic curves'' available at
\texttt{http://math.stanford.edu/\~{}conrad/papers/minimalmodel.pdf}}),
the set of $k$-points of $X(n, \bar{\mathcal{C}})$
corresponds bijectively to the isomorphism classes of
minimal elliptic surfaces $\mathcal{S} \to \mathcal{C}_{k}$
with section and $\chi(\mathcal{S}, \mathcal{O}_{\mathcal{S}}) = n$.
It also corresponds bijectively to the isomorphism classes of
elliptic curves over the fraction field of $K \otimes_{\F_{q}} k$ with $\deg(\Delta) = 12 n$
by the paragraph after Theorem \ref{t:tan}.
For a (not necessarily closed) point $a \in X(n, \bar{\mathcal{C}})$,
we define its $\mu$-invariant $\mu_{a}$ to be
the $\mu$-invariant of the elliptic curve
corresponding to a geometric point (the spectrum of an algebraically closed field) lying over $a$.
This defines a function on the underlying set of the scheme $X(n, \bar{\mathcal{C}})$
valued in non-negative integers.

\subsection{Statement and proof}

Now the following theorem states that general points of $X(n, \bar{\mathcal{C}})$
correspond to elliptic curves with trivial $\mu$-invariant.

\begin{theorem} \label{t:GenVan}
	Let $p > 3$ and $n > (g_{\mathcal{C}} - 1) / 2$ be as above.
	Then the $\mu$-invariant function $a \mapsto \mu_{a}$ on $X(n, \bar{\mathcal{C}})$ as defined above
	is upper semicontinuous.
	The (open) locus where $\mu_{a} = 0$ is dense.
\end{theorem}

Note that if $g_{\mathcal{C}} = 0$, then there is no restriction on $n$.
We will prove this theorem below.
We need some preparations.
We first define a universal family of Weierstrass models over $Y(n, \bar{\mathcal{C}})$.

Let $h \colon Y(n, \bar{\mathcal{C}}) \to \Pic_{\bar{\mathcal{C}} / \bar{\F}_{q}}^{n}$
be the composite of the inclusion
	$
			Y
		\hookrightarrow
			\mathbf{V}((
				\mathrm{pr}_{2 \ast}(
					\mathcal{P}_{n}^{\otimes 4} \oplus \mathcal{P}_{n}^{\otimes 6}
				)
			)^{\vee})
	$
and the natural projection
	$
			\mathbf{V}((
				\mathrm{pr}_{2 \ast}(
					\mathcal{P}_{n}^{\otimes 4} \oplus \mathcal{P}_{n}^{\otimes 6}
				)
			)^{\vee})
		\twoheadrightarrow
			\Pic_{\bar{\mathcal{C}} / \bar{\F}_{q}}^{n}
	$,
which is smooth.
Let
	\[
			h_{\bar{\mathcal{C}}}
		\colon
			\bar{\mathcal{C}} \times Y(n, \bar{\mathcal{C}})
		\to
			\bar{\mathcal{C}} \times \Pic_{\bar{\mathcal{C}} / \bar{\F}_{q}}^{n}
	\]
be the product of the identity on $\bar{\mathcal{C}}$ with $h$.
Let
	$
		\mathbb{P} \bigl(
			h_{\bar{\mathcal{C}}}^{\ast}(
				\mathcal{P}_{n}^{\otimes 2} \oplus \mathcal{P}_{n}^{\otimes 3} \oplus \mathcal{O}
			)^{\vee}
		\bigr)
	$
be the (total space of) projective bundle associated with the vector bundle
	$
		h_{\bar{\mathcal{C}}}^{\ast}(
			\mathcal{P}_{n}^{\otimes 2} \oplus \mathcal{P}_{n}^{\otimes 3} \oplus \mathcal{O}
		)^{\vee}
	$
on $\bar{\mathcal{C}} \times Y(n, \bar{\mathcal{C}})$,
where $\mathcal{O}$ is the structure sheaf of
$\bar{\mathcal{C}} \times \Pic_{\bar{\mathcal{C}} / \bar{\F}_{q}}^{n}$.

\begin{definition}
	Define $\mathcal{S}^{\ast}$ to be the closed subscheme of
		$
			\mathbb{P} \bigl(
				h_{\bar{\mathcal{C}}}^{\ast}(
					\mathcal{P}_{n}^{\otimes 2} \oplus \mathcal{P}_{n}^{\otimes 3} \oplus \mathcal{O}
				)^{\vee}
			\bigr)
		$
	defined by the homogeneous Weierstrass equation
	$y^{2} z = x^{3} - g_{2} x z^{2} - g_{3} z^{3}$,
	where $(\mathscr{L}, g_{2}, g_{3})$ is any point of $Y(n, \bar{\mathcal{C}})$
	(so $g_{2}$ is a global section of $\mathscr{L}^{\otimes 4}$
	and $g_{3}$ is a global section of $\mathscr{L}^{\otimes 6}$)
	and $x, y, z$ are coordinates of
	$\mathcal{P}_{n}^{\otimes 2} \oplus \mathcal{P}_{n}^{\otimes 3} \oplus \mathcal{O}$
	(or $\mathscr{L}^{\otimes 2} \oplus \mathscr{L}^{\otimes 3} \oplus \mathcal{O}$).
\end{definition}

The equations for the relative affine patches
(affine over $\bar{\mathcal{C}} \times Y(n, \bar{\mathcal{C}})$) are given by
$y^{2} = x^{3} - g_{2} x - g_{3}$ in
	$
		\mathbb{V} \bigl(
			h_{\bar{\mathcal{C}}}^{\ast}(
				\mathcal{P}_{n}^{\otimes 2} \oplus \mathcal{P}_{n}^{\otimes 3}
			)^{\vee}
		\bigr)
	$
(for the locus $z \ne 0$) and
$z = x^{3} - g_{2} x z^{2} - g_{3} z^{3}$ in
	$
		\mathbb{V} \bigl(
			h_{\bar{\mathcal{C}}}^{\ast}(
				\mathcal{P}_{n}^{\otimes -1} \oplus \mathcal{P}_{n}^{\otimes -3}
			)^{\vee}
		\bigr)
	$
(for the locus $y \ne 0$).
The composite of the inclusion
	$
			\mathcal{S}^{\ast}
		\hookrightarrow
			\mathbb{P} \bigl(
				h_{\bar{\mathcal{C}}}^{\ast}(
					\mathcal{P}_{n}^{\otimes 2} \oplus \mathcal{P}_{n}^{\otimes 3} \oplus \mathcal{O}
				)^{\vee}
			\bigr)
	$
and the natural projection
	$
			\mathbb{P} \bigl(
				h_{\bar{\mathcal{C}}}^{\ast}(
					\mathcal{P}_{n}^{\otimes 2} \oplus \mathcal{P}_{n}^{\otimes 3} \oplus \mathcal{O}
				)^{\vee}
			\bigr)
		\twoheadrightarrow
			\bar{\mathcal{C}} \times Y(n, \bar{\mathcal{C}})
	$
defines a morphism
$\mathcal{S}^{\ast} \to \bar{\mathcal{C}} \times Y(n, \bar{\mathcal{C}})$.

\begin{lemma} \label{l:UniWei} \mbox{}
	\begin{itemize}
		\item
			The morphism $\mathcal{S}^{\ast} \to \bar{\mathcal{C}} \times Y(n, \bar{\mathcal{C}})$ is
			projective and flat.
		\item
			The fibers of the morphism
			$\mathcal{S}^{\ast} \to Y(n, \bar{\mathcal{C}})$
			are normal projective surfaces with only rational double points.
		\item
			For any geometric point $b$ of $Y(n, \bar{\mathcal{C}})$,
			let $\mathcal{S}_{b}^{\ast}$ be the fiber of
			$\mathcal{S}^{\ast} \to Y(n, \bar{\mathcal{C}})$ over $b$
			and $\mathcal{S}_{b}$ the minimal resolution of singularities of $\mathcal{S}_{b}^{\ast}$.
			Then for any algebraically closed field $k$ over $\F_{q}$,
			the map $b \in Y(n, \bar{\mathcal{C}})(k) \mapsto [\mathcal{S}_{b} \to \mathcal{C}_{k}]$
			gives a bijection between $X(n, \bar{\mathcal{C}})(k) = Y(n, \bar{\mathcal{C}})(k) / k^{\times}$
			and the set of isomorphism classes of minimal elliptic surfaces over $k$
			fibered over $\mathcal{C}_{k}$ with section and
			$\chi(\mathcal{S}_{b}, \mathcal{O}_{\mathcal{S}_{b}}) = n$.
	\end{itemize}
\end{lemma}

\begin{proof}
	The flatness follows from the above description of the relative affine patches.
	The rest is \cite[Theorems 1 and 2]{Kas77}
	(valid in any characteristic $p > 3$ as noted before).
\end{proof}

Note that a blowup of $\mathcal{S}^{\ast}$ does not give
a family over $Y(n, \bar{\mathcal{C}})$ with fibers $\mathcal{S}_{b}$
due to exceptional divisors.
A simultaneous resolution of singularities for a family of surfaces
does not always exist as a scheme (\cite{Art74a}).
That is why we work with the non-smooth surfaces $\mathcal{S}_{b}^{\ast}$.

\begin{lemma} \label{l:PicFib}
	For any geometric point $b$ of $Y(n, \bar{\mathcal{C}})$,
	the object $\Pic_{\mathcal{S}_{b}^{\ast} / b}^{0}$ is
	an abelian variety of dimension independent of $b$.
\end{lemma}

\begin{proof}
	If $n = 0$, then $\deg(\Delta) = 0$ and $g_{\mathcal{C}} = 0$.
	Hence $\mathcal{S}_{b} = \mathcal{S}_{b}^{\ast} \to \mathcal{C}_{b}$ is
	a smooth family of elliptic curves over a projective line,
	which is necessarily a constant family.
	Therefore $\Pic_{\mathcal{S}_{b}^{\ast} / b}^{0}$ is an elliptic curve.
	Let $n > 0$.
	The morphism $\mathcal{S}_{b}^{\ast} \to \mathcal{C}_{b}$ has a section
	corresponding to $(x : y : z) = (0 : 1 : 0)$.
	Hence the morphism
	$\Pic_{\mathcal{C}_{b} / b} \to \Pic_{\mathcal{S}_{b}^{\ast} / b}$
	is injective.
	The group scheme $\Pic_{\mathcal{S}_{b}^{\ast} / b}^{0}$ is projective
	by the (geometric) normality of $\mathcal{S}_{b}^{\ast}$ and \cite[Theorem 5.4]{Kle05}.
	Its tangent space at zero is $\coh^{1}(\mathcal{S}_{b}^{\ast}, \mathcal{O}_{\mathcal{S}_{b}^{\ast}})$,
	which is isomorphic to $\coh^{1}(\mathcal{S}_{b}, \mathcal{O}_{\mathcal{S}_{b}})$
	since $\mathcal{S}_{b}^{\ast}$ has only rational singularities
	by Lemma \ref{l:UniWei}.
	The dimension of this space is $g_{\mathcal{C}}$
	by $\deg(\mathcal{L}) = n > 0$ and the paragraph after Theorem \ref{t:tan}.
	Therefore $\Pic_{\mathcal{C}_{b} / b}^{0}$ and $\Pic_{\mathcal{S}_{b}^{\ast} / b}^{0}$
	have the same tangent space.
	Hence $\Pic_{\mathcal{C}_{b} / b}^{0} \overset{\sim}{\to} \Pic_{\mathcal{S}_{b}^{\ast} / b}^{0}$,
	which is an abelian variety of dimension $g_{\mathcal{C}}$.
\end{proof}

\begin{lemma} \label{l:BrRes}
	For any geometric point $b$ of $Y(n, \bar{\mathcal{C}})$, we have
	$\hat{\Br}_{\mathcal{S}_{b}^{\ast} / b} \overset{\sim}{\to} \hat{\Br}_{\mathcal{S}_{b} / b}$,
	which is a formal Lie group.
\end{lemma}

\begin{proof}
	The same argument as the previous lemma shows that
	$\Pic_{\mathcal{S}_{b} / b}^{0}$
	is an abelian variety.
	Hence the formal Brauer group of $\mathcal{S}_{b}$ is a formal Lie group
	by \cite[II, Corollary (4.1)]{AM77} and the paragraphs after.
	By the same argument as \cite[(2.1)]{Art74b}, the morphism
	$\hat{\Br}_{\mathcal{S}_{b}^{\ast} / b} \to \hat{\Br}_{\mathcal{S}_{b} / b}$
	is an isomorphism.
\end{proof}

\begin{lemma} \label{l:RelPic}
	$\Pic_{\mathcal{S}^{\ast} / Y(n, \bar{\mathcal{C}})}^{0}$ is represented
	by a smooth scheme over $Y(n, \bar{\mathcal{C}})$.
	In particular, the formal completion of $\Pic_{\mathcal{S}^{\ast} / Y(n, \bar{\mathcal{C}})}$
	along the zero section is formally smooth.
\end{lemma}

\begin{proof}
	The morphism $\mathcal{S}^{\ast} \to Y(n, \bar{\mathcal{C}})$ is projective and flat
	and its fibers are normal surfaces
	by Lemma \ref{l:UniWei}.
	Hence the sheaf $\Pic_{\mathcal{S}^{\ast} / Y(n, \bar{\mathcal{C}})}$ is represented by
	a separated scheme locally of finite type by \cite[Theorem 4.8]{Kle05}.
	The fibers of $\Pic_{\mathcal{S}^{\ast} / Y(n, \bar{\mathcal{C}})}^{0}$
	are abelian varieties of the same dimension by Lemma \ref{l:PicFib}.
	Therefore \cite[Proposition 5.20]{Kle05} implies that
	$\Pic_{\mathcal{S}^{\ast} / Y(n, \bar{\mathcal{C}})}^{0}$ is smooth over $Y(n, \bar{\mathcal{C}})$.
\end{proof}

\begin{lemma} \label{l:RelBr}
	The formal Brauer group $\hat{\Br}_{\mathcal{S}^{\ast} / Y(n, \bar{\mathcal{C}})}$
	of $\mathcal{S}^{\ast} \to Y(n, \bar{\mathcal{C}})$ exists
	as a formally smooth formal group scheme over $Y(n, \bar{\mathcal{C}})$.
	Any point of $Y(n, \mathcal{C})$ has an affine open neighborhood $\mathcal{U} = \spec R$
	where $\hat{\Br}_{\mathcal{S}^{\ast} / Y(n, \bar{\mathcal{C}})}$
	is given by a formal group law over $R$.
\end{lemma}

\begin{proof}
	The morphism $\mathcal{S}^{\ast} \to Y(n, \bar{\mathcal{C}})$ is
	cohomologically flat in dimension zero by \cite[(7.8.6)]{Gro63}
	since it is flat with geometrically reduced fibers
	by Lemma \ref{l:UniWei}.
	The same lemma says that these fibers are normal projective surfaces.
	Therefore, with Lemma \ref{l:RelPic},
	we may apply \cite[Corollary (4.1)]{AM77} and the paragraphs after
	to see that $\hat{\Br}_{\mathcal{S}^{\ast} / Y(n, \bar{\mathcal{C}})}$
	exists as a formally smooth formal group scheme over $Y(n, \bar{\mathcal{C}})$.
	By \cite[II, (2.4) and (3.2)]{AM77},
	the Lie algebra of $\hat{\Br}_{\mathcal{S}^{\ast} / Y(n, \bar{\mathcal{C}})}$
	is the second direct image of the structure sheaf of $\mathcal{S}^{\ast}$
	by the morphism $\mathcal{S}^{\ast} \to Y(n, \bar{\mathcal{C}})$.
	For any point $b \in Y(n, \bar{\mathcal{C}})$,
	the dimension of
		$
				\coh^{2}(\mathcal{S}_{b}^{\ast}, \mathcal{O}_{\mathcal{S}_{b}^{\ast}})
			\cong
				\coh^{2}(\mathcal{S}_{b}, \mathcal{O}_{\mathcal{S}_{b}})
		$
	depends only on $n$ and $g_{\mathcal{C}}$ but not on $b$
	by the paragraph after Theorem \ref{t:tan}.
	Therefore by Grauert's theorem (\cite[III, Corollary 12.9]{Har77}),
	$\mathrm{Lie}(\hat{\Br}_{\mathcal{S}^{\ast} / Y(n, \bar{\mathcal{C}})})$
	is a vector bundle on $Y(n, \mathcal{C})$.
	On each open affine $\mathcal{U} = \spec R$ of $Y(n, \mathcal{C})$
	where $\mathrm{Lie}(\hat{\Br}_{\mathcal{S}^{\ast} / Y(n, \bar{\mathcal{C}})})$ is free,
	the formal group $\hat{\Br}_{\mathcal{S}^{\ast} / Y(n, \bar{\mathcal{C}})}$
	is given by a formal group law over $R$
	by \cite[II, Corollary 2.32]{Zin84}.
\end{proof}

By pullback, the function $a \mapsto \mu_{a}$ on $X(n, \bar{\mathcal{C}})$ defines
a function $b \mapsto \mu_{b}$ on $Y(n, \bar{\mathcal{C}})$.
The value $\mu_{b}$ is equal to the dimension of the unipotent part of
the geometric fiber of $\hat{\Br}_{\mathcal{S}^{\ast} / Y(n, \bar{\mathcal{C}})}$ at $b$
by Remark \ref{sdef}, Proposition \ref{p:ShaBr} and Lemma \ref{l:BrRes}.

\begin{lemma} \label{l:usc}
	The function $b \mapsto \mu_{b}$ on $Y(n, \bar{\mathcal{C}})$ is upper semicontinuous.
\end{lemma}

\begin{proof}
	Set $\mathcal{Y} = Y(n, \bar{\mathcal{C}})$,
	$\mathcal{G} = \hat{\Br}_{\mathcal{S}^{\ast} / Y(n, \bar{\mathcal{C}})}$
	and $d = \dim(\mathcal{G})$.
	For any non-negative integer $m$,
	let $\mathcal{G}[F^{m}]$ be the kernel of the $p^{m}$-th power relative Frobenius morphism
	$\mathcal{G} \to \mathcal{G}^{(p^{m})}$ over $\mathcal{Y}$.
	Then $\mathcal{G}[F^{m}]$ is finite flat over $\mathcal{Y}$ by Lemma \ref{l:RelBr}.
	Hence the intersection $\mathcal{G}[F^{m}] \cap \mathcal{G}[p^{d}]$ of
	$\mathcal{G}[F^{m}]$ and the $p^{d}$-torsion part $\mathcal{G}[p^{d}]$
	is finite over $\mathcal{Y}$.
	
	For any $b \in \mathcal{Y}$,
	let $\mathcal{G}_{b, \mathrm{uni}}$ be the unipotent part of the fiber $\mathcal{G}_{b}$
	and $\mathcal{G}_{b, \mathrm{div}}$ the $p$-divisible quotient of $\mathcal{G}_{b}$
	(\cite[V, \S 8, Theorem 5.36]{Zin84}).
	Let $h_{b}$ be the height of $\mathcal{G}_{b, \mathrm{div}}$.
	Let $h_{b, m}$ be the log base $p$ of the degree of the finite group scheme
	$\mathcal{G}_{b}[F^{m}] \cap \mathcal{G}_{b}[p^{d}]$ over $b$.
	We first show that
		\[
				m \mu_{b}
			\le
				h_{b, m}
			\le
				m \mu_{b} + d h_{b}.
		\]
	The exact sequence
		$
				0
			\to
				\mathcal{G}_{b, \mathrm{uni}}
			\to
				\mathcal{G}_{b}
			\to
				\mathcal{G}_{b, \mathrm{div}}
			\to
				0
		$
	induces an exact sequence
		\[
				0
			\to
				\mathcal{G}_{b, \mathrm{uni}}[F^{m}] \cap \mathcal{G}_{b, \mathrm{uni}}[p^{d}]
			\to
				\mathcal{G}_{b}[F^{m}] \cap \mathcal{G}_{b}[p^{d}]
			\to
				\mathcal{G}_{b, \mathrm{div}}[F^{m}] \cap \mathcal{G}_{b, \mathrm{div}}[p^{d}].
		\]
	Since $\mathcal{G}_{b, \mathrm{uni}}$ is killed by $p^{d}$,
	the first term
	$\mathcal{G}_{b, \mathrm{uni}}[F^{m}] \cap \mathcal{G}_{b, \mathrm{uni}}[p^{d}]$
	is equal to $\mathcal{G}_{b, \mathrm{uni}}[F^{m}]$,
	whose log base $p$ of the degree is $m \mu_{b}$.
	The third term
	$\mathcal{G}_{b, \mathrm{div}}[F^{m}] \cap \mathcal{G}_{b, \mathrm{div}}[p^{d}]$
	is contained in $\mathcal{G}_{b, \mathrm{div}}[p^{d}]$,
	whose log base $p$ of the degree is $d h_{b}$.
	These show the desired inequalities.
	
	Set $m = m_{b} = d h_{b} + 1$.
	Let $\mathcal{U}_{b}$ be the set of points $c \in \mathcal{Y}$
	such that $h_{c, m} \le h_{b, m}$.
	By the finiteness of $\mathcal{G}[F^{m}] \cap \mathcal{G}[p^{d}]$ over $\mathcal{Y}$
	and Nakayama's lemma,
	we know that $\mathcal{U}_{b}$ is an open subset of $\mathcal{Y}$
	and hence an open neighborhood of $b$.
	For any $c \in \mathcal{U}_{b}$,
	we have
		\[
				m \mu_{c}
			\le
				h_{c, m}
			\le
				h_{b, m}
			\le
				m \mu_{b} + d h_{b}
			=
				m (\mu_{b} + 1) - 1,
		\]
	so $\mu_{c} \le \mu_{b}$.
	This shows that $b \mapsto \mu_{b}$ is upper semicontinuous.
\end{proof}

Now we prove Theorem \ref{t:GenVan}.

\begin{proof}[Proof of Theorem \ref{t:GenVan}]
	By Lemma \ref{l:usc},
	the locus of $Y(n, \bar{\mathcal{C}})$
	where $\mu_{b} = 0$ (or equivalently, $\mu_{b} < 1$) is open.
	We show that this locus is dense.
	The scheme $\Pic_{\bar{\mathcal{C}} / \bar{\F}_{q}}^{n}$ is irreducible
	since for any choice of a closed point $P$ of $\bar{\mathcal{C}}$,
	the map $D \mapsto D + P$ of divisors defines an isomorphism
		$
				\Pic_{\bar{\mathcal{C}} / \bar{\F}_{q}}^{0}
			\overset{\sim}{\to}
				\Pic_{\bar{\mathcal{C}} / \bar{\F}_{q}}^{n}
		$
	from the Jacobian variety $\Pic_{\bar{\mathcal{C}} / \bar{\F}_{q}}^{0}$ of $\bar{\mathcal{C}}$.
	Hence the vector bundle
		$
			\mathbf{V}((
				\mathrm{pr}_{2 \ast}(
					\mathcal{P}_{n}^{\otimes 4} \oplus \mathcal{P}_{n}^{\otimes 6}
				)
			)^{\vee})
		$
	over it is also irreducible.
	Hence so is its non-empty open subscheme $Y(n, \bar{\mathcal{C}})$.
	The locus where $\mu = 0$ on $X(n, \bar{\mathcal{C}})$ or $Y(n, \bar{\mathcal{C}})$ is non-empty
	by Proposition \ref{p:NonEmp}.
	A non-empty open subset of an irreducible topological space is dense.
	Therefore the locus of $Y(n, \bar{\mathcal{C}})$ with $\mu_{b} = 0$ is dense.
	
	As we saw before, $X(n, \bar{\mathcal{C}})$ is
	a universal geometric quotient of $Y(n, \bar{\mathcal{C}})$.
	In particular, $X(n, \bar{\mathcal{C}})$ has a quotient topology of $Y(n, \bar{\mathcal{C}})$.
	Hence the function $a \mapsto \mu_{a}$ on $X(n, \bar{\mathcal{C}})$ is
	also upper semicontinuous,
	and the locus of $X(n, \bar{\mathcal{C}})$ with $\mu_{a} = 0$ is dense.
\end{proof}

If the condition $n > (g_{\mathcal{C}} - 1) / 2$ is not satisfied,
the scheme structure on the set $X(n, \bar{\mathcal{C}})$ is much more complicated;
see \cite{Sei87}.
In particular, it is not necessarily irreducible.
It is not clear if the theorem extends to this case.

\begin{remark}
	The variety $X(n, \bar{\mathcal{C}})$ over $\bar{\F}_{q}$ can be naturally defined
	over a finite extension of $\F_{q}$ where $\mathcal{C}$ has a rational point
	(so that a Poincar\'e sheaf is well-defined).
	Assume that $\mathcal{C}$ has a rational point
	and denote the resulting variety over $\F_{q}$ by $X(n, \mathcal{C})$.
	The residue fields at closed points of $X(n, \mathcal{C})$ are finite.
	Therefore Theorem \ref{t:GenVan} says that
	there are ``many'' elliptic curves over finite constant extensions of $K$
	with given $\deg(\Delta)$ and $\mu = 0$,
	where this $\mu$ means the ``usual'' $\mu$-invariant
	as defined in Definition \ref{d:mu}.
	More specifically, the $\mu=0$ locus
	$X(n, \mathcal{C})_{\mu = 0}$ is a dense open subvariety of $X(n, \mathcal{C})$.
	Hence the set of closed point of $X(n, \mathcal{C})_{\mu = 0}$
	with the induced Zariski topology is dense open in
	the set of closed points of $X(n, \mathcal{C})$.
\end{remark}

\section{Computations}\label{s:verify}

In this section, we give a certain number of examples where we are able to calculate the $\mu$-invariant.

\subsection{Legendre curves}
Assume that $\mathrm{char}(K)\not=2$ and $A$ is defined by the Legendre form,
\begin{equation}\label{e:legendre}
y^2=x(x-1)(x-f).
\end{equation}
For each non-zero $g\in K$, let $(g)_0$ and $(g)_\infty$ denote the divisors of zeros and poles of $g$.
Write
$$N=\sum_v n_v\cdot v,$$
and
$$\Delta=\sum_v\delta_v\cdot v.$$

For a place $v$, let $\F_v$ denote the residue field and choose a local parameter $\pi_v$.
Suppose $v$ does not divide $(f)_\infty$ and let $\bar f\in\F_v$ denote the residue class of $f$ at $v$.
Then the equation
\begin{equation}\label{e:bar}
y^2=x(x-1)(x-\bar f)
\end{equation}
defines a curve over $\F_v$ having at worst a node. Hence it actually defines $\bar A$, the reduction of $A$ at $v$,
so \eqref{e:legendre} is a minimal equation.
Furthermore, if $v$ does not divide $(f)_0+(f-1)_0$, then $\bar A$ is an elliptic curve, in this case, we have
\begin{equation}\label{e:nd0}
n_v=0,\;\;\text{and}\;\;\delta_v=0.
\end{equation}
In the case where $v$ divides $(f)_0$, the equation \eqref{e:bar} reads $y^2=x^2(x-1)$. Put $y=\xi\cdot x$, then
for $(x,y)\not=(0,0)$, the singularity of $\bar A$, we have $x=\xi^2+1$, $y=\xi(\xi^2+1)$. Thus, $A$ has split multiplicative reduction if
$-1$ is a square in $\F_v$; non-split multiplicative reduction, if $-1$ is not a square. Let $m_v$ be the coefficient of $v$ in $(f)_0$
so that $f=\pi^{m_v}u$, for some $u\in\O_v$. Then the discriminant of \eqref{e:legendre} equals
$$(\pi_v^{m_v}u(\pi_v^{m_v}u-1))^2=\pi^{2m_v}u',\;\;u'\in\O_v^*.$$
Thus, in this case, we have
\begin{equation}\label{e:nd1}
n_v=1,\;\;\text{and}\;\;\delta_v=2m_v.
\end{equation}

If $v$ divides $(f-1)_0$, the equation \eqref{e:bar} becomes $y^2=x(x-1)^2$. By putting $y=\xi\cdot (x-1)$
and by an argument similar to the above, we deduce that $A$ has split multiplicative reduction at $v$ and if $m_v$ is the coefficient of $v$ in
$(f-1)_0$, then also
\begin{equation}\label{e:nd2}
n_v=1,\;\;\text{and}\;\;\delta_v=2m_v.
\end{equation}

Call a divisor $C=\sum c_v\cdot v$ even, if all $c_v$ are even integers.

\begin{lemma}\label{l:ss}
The elliptic curve $A$ has semi-stable reduction at all places, if and only if $(f)_\infty$ is even.
\end{lemma}
\begin{proof}
Suppose $m_v>0$ is the coefficient of a place $v$ in $(f)_\infty$ so that $f=\pi^{-m_v}u$, for some $u\in\O_v^*$.
If $m_v=2e$ is even, then by replacing $(x,y)$ with $(x\pi_v^{-2e},y\pi_v^{-3e})$, the equation \eqref{e:legendre} becomes
$$y^2=x(x-\pi_v^{2e})(x-u).$$
This turns out to be a minimal equation with multiplicative reduction.
In this case,
\begin{equation}\label{e:nd3}
n_v=1,\;\;\text{and}\;\;\delta_v=2m_v.
\end{equation}
On the other hand, if $m_v=2e-1$ is odd, then by replacing $(x,y)$
with $(x\pi_v^{-2e},y\pi_v^{-3e})$, the equation \eqref{e:legendre} becomes
$$y^2=x(x-\pi_v^{2e})(x-u\pi_v),$$
which is also a minimal equation (see \cite[\S VI.1]{Sil86}) with additive reduction.
\end{proof}

If $(f)_\infty$ is even, then the equalities \eqref{e:nd0}, \eqref{e:nd1}, \eqref{e:nd2}, \eqref{e:nd3} together implies
\begin{equation}\label{e:ddelta}
\deg \Delta=6\deg f.
\end{equation}

\subsection{A simple example}\label{sub:example}
Let $K=\F_q(t)$. Here we demonstrate the case of lower degree $f$\footnote{Here the degree means that of $(f)_0$, or $(f)_\infty$.}.
Take $f(t)=\frac{g(t)}{t^2}$, $g(t)\in \F_q(t)$, $ t\nmid g(t)$, $\deg g(t)\leq 2$, so $\deg f=2$ and hence
$$\deg \Delta=12.$$
Therefore $\mu_{A / K} = 0$
by Theorem \ref{t:tan}.
The associated surface is rational
(hence has trivial Brauer group over $\bar{\F}_{q}$)
by the paragraph after Theorem \ref{t:tan}.

\subsection{Computation using Magma}\label{su:high}

In this subsection, we examine the $\mu$-invariant of an elliptic curve
over $K$ other than a rational function field using Magma \cite{BCP97}.
We limit ourselves to an elliptic curve over a quadratic extension of a rational
function field since Magma cannot treat curves over cubic extensions or
higher at the time we wrote the paper.

Let $\mathcal{C}/\F_7$ be the supersingular elliptic curve with Weierstrass equation $Y^2=X^3-5X$. We denote by $K$ its function field. It is a quadratic extension of the rational function field $\F_7(x)$. Let $E/\F_7(x)$ be the elliptic curve given by the Legendre form $Y^2=X(X-1)(X-1/x^2)$. Then, the $L$-function of $E/K$ can be computed as the product $L_{E}(s).L_{tw(E)}(s)$, where $tw(E)/\F_7(x)$ is the quadratic twist of $E$ by the element $x^3+5x$ of $\F_7(x)$. The integers $\deg(\Delta)$ and $\theta$ of $E/K$ are computed as follows: First to compute $\theta$ we just need to know $P(T)$ such that $P(7^{-s})=L_{E}(s)$. This information is obtained by the following Magma command:
\begin{enumerate}
\item[] \verb#F<x> := FunctionField(GF(7));#
\item[] \verb#f := (1/x^2);#
\item[] \verb#E := EllipticCurve([0,-1-f,0,f,0]);#
\item[] \verb#Etw := QuadraticTwist(E,x^3-5*x);#
\item[] \verb#LFunction(Etw);#
\item[] \verb#LFunction(E);#
\end{enumerate}
We obtain
\begin{enumerate}
\item[] \verb#-117649*T^6 + 1715*T^4 - 35*T^2 + 1#
\item[] \verb#1#
\end{enumerate}
By hand computation, we know that $7^{2} P(7^{-1} T)$ is $7$-primitive,
so $\theta = 2$.

Next, to implement the curve $E/K$, we use the Magma command:
\begin{enumerate}
\item[] \verb#S<z> := PolynomialRing(F);#
\item[] \verb#K<t> := ext<F |z^2-x^3-5*x>;#
\item[] \verb#EK:=BaseChange(E,K);#
\end{enumerate}
The integer $\deg(\Delta_{E/K})$ is obtained by the following command:
\begin{enumerate}
\item[] \verb#SK:=BadPlaces(EK);#
\item[] \verb#&+[LocalInformation(EK,v)[2]*Degree(v):v in SK];#
\end{enumerate}
We obtain $\deg(\Delta_{E/K})=24$ and conclude that the $\mu$-invariant of $E/K$ is zero
by Theorem \ref{t:tan}.
The associated surface has Kodaira dimension $1$
by the paragraph after Theorem \ref{t:tan}.

\subsection{An example of non-semistable reduction}\label{su:nonss}
We consider the curve $y^2=x(x-1)(x-f(t))$
with $f:= (1+t+t^2)/(1+t)\in\F_3[t]$ on Magma. The output is the following
\vskip5pt
\noindent
\begin{enumerate}
\item[(a)] The bad places are $( t + 1)_0$, $(1/t)_0$, those are additive places, and $(t + 2)_0$, $(t)_0$, multiplicative places.
\item[(b)] The $j$-invariant is
$$j=\frac{t^{12} + t^9 + 2t^6 + 2t^3 + 1}{t^{10} + t^9 + 2t^8 + t^7 + 2t^6 + t^5 + t^4}=\frac{t^{12} + t^9 + 2t^6 + 2t^3 + 1}{(t+1)^2(t-1)^4t^4}.$$
One can compute by
hand that
$$\Delta=8(t+1)_0+8(1/t)_0+4(t)_0+4(t+2)_0,$$
so $\deg \Delta=24$. The equation \eqref{e:ddelta} is not satisfied here. Note that $j=c_4^3/\Delta$, but in this case,
$\Delta$ is not the exact ``denominator''
of $j$.
\item[(c)] $P(T)=-9T^2+1$, so $\theta=0$.
\end{enumerate}

Therefore $\mu = 1$.
The associated surface is a K3 surface,
and the result $\mu = 1$ shows that it is supersingular.

\subsection{A supersingular isotrivial curve} \label{su:sic}

We give an example of an isotrivial non-constant elliptic curve
with everywhere good supersingular reduction and $\mu = 0$
over a function field with non-invertible Hasse-Witt matrix.
As any finite extension of such a function field has non-invertible Hasse-Witt matrix,
this gives a counterexample to the ``only if'' direction of \cite[Theorem 1.8 (2)]{OT09}.

Assume $p \ne 2$.
Let $\mathcal{C}' \to \mathcal{C}$ be a degree $2$ isogeny between supersingular elliptic curves over $\F_{q}$.
Let $K' / K$ be the corresponding extension of function fields.
Let $A_{0} / K$ be a constant supersingular elliptic curve.
Let $A / K$ be the quadratic twist of $A_{0}$ by $K' / K$,
which has good reduction everywhere since $K' / K$ is unramified everywhere.
The elliptic surfaces over $\F_{q}$ corresponding to $A_{0} / K$ and $A_{0} / K'$
are both supersingular abelian surfaces.
Therefore $\mu_{A_{0} / K} = \mu_{A_{0} / K'} = 1$
by Proposition \ref{p:aborK3}.
As $p \nmid [K' : K] = 2$,
the $\Lambda$-module $\Sel_{p^{\infty}}(A_{0} / K_{\infty}'^{(p)})$ is isomorphic to
$\Sel_{p^{\infty}}(A / K_{\infty}^{(p)}) \oplus \Sel_{p^{\infty}}(A_{0} / K_{\infty}^{(p)})$.
Therefore $\mu_{A / K} = \mu_{A_{0} / K'} - \mu_{A_{0} / K} = 0$.
This $A / K$ thus gives a desired example.

One checks that the elliptic surface corresponding to $A / K$ is a hyperelliptic surface.
As such, we have $\mu_{A / K} = 0$ also by Proposition \ref{p:aborK3}.
More generally, for any hyperelliptic surface $\mathcal{S}$
obtained as a quotient of the product of two supersingular elliptic curves (\cite[Theorem 4]{BM77}),
the generic fiber $A / K$ of its Albanese morphism $\mathcal{S} \to \mathcal{C}$ gives a desired example.
This includes the case of $p = 2$.

What is true about the ``only if'' direction of \cite[Theorem 1.8 (2)]{OT09}
and what the proof of this statement in \cite{OT09} actually says is that
for a \emph{constant} non-zero supersingular abelian variety $A / K$,
if $\mu_{A / K} = 0$, then the Hasse-Witt matrix of $K$ is invertible.

\subsection{Large rank curves of Ulmer and Shioda}
We compute the $\mu$-invariants of the elliptic curves with large ranks
considered by Ulmer \cite[Theorem 1.5]{Ulm02} and Shioda \cite[Remark 10]{Shi86}.
This is an application of the method of proof of Proposition \ref{p:KodOne}.

Ulmer's elliptic curve is defined by the equation
$y^{2} + x y = x^{3} - t^{d}$ over $K = \F_{p}(t)$,
where $p$ is an arbitrary prime, $d = p^{n} + 1$ and $n$ a positive integer.
As explained in \cite[\S 10.2]{Ulm02},
its geometric Mordell-Weil rank $\bar{r}$ attains the geometric bound,
i.e., $\bar{r} = a$ ($= \deg(P_{1}(t))$; see \eqref{e:OS}).
This implies $\theta = 0$ by the proof of Proposition \ref{p:KodOne}.
Therefore $\mu = \deg(\Delta) / 12 - 1$ by Theorem \ref{t:tan}.
The calculations in \cite[\S 2]{Ulm02} show that
$\deg(\Delta) / 12 = \lceil d / 6 \rceil$ (the least integer $\ge d / 6$).
Hence $\mu = \lceil d / 6 \rceil - 1$.

On the other hand, Shioda's elliptic curve is defined by the equation
$y^{2} = x^{3} + x + t^{d}$ over $K = \F_{p}(t)$,
where $p \equiv -1 \mod 4$, $d = (p^{\nu} + 1) / 2$
and $\nu$ is a positive odd integer.
Let $\mathcal{S} / \F_{p}$ be the elliptic surface corresponding to this curve.
Then its geometric Lefschetz number
	$
			\lambda(\mathcal{S}_{\bar{\F}_{p}})
		=
			\deg P_{\mathcal{S}, 2} - \rho_{\mathcal{S}_{\bar{\F}_{p}}}
	$
is zero as shown in \cite[Remark 10]{Shi86}.
Therefore $\theta = 0$ again by the proof of Proposition \ref{p:KodOne}.
Therefore $\mu = \deg(\Delta) / 12 - 1$ by Theorem \ref{t:tan}.
A calculation similar to \cite[\S 2]{Ulm02} shows that
$\deg(\Delta) / 12 = \lceil d / 6 \rceil$.
Hence $\mu = \lceil d / 6 \rceil - 1$.

\subsection{Other examples in the literature} Numerous other computations can be found in \cite{Ulm19} (for example \cite{Ulm19}, Remark 4.5, Theorem 5.1, Corollary 6.5 and its applications in Sections 10--12).

\appendix

\section{Integrality of sum of slopes less than one}

In this appendix, we show that the number $\theta$ defined in Definition \ref{d:theta}
is an integer.
We follow \cite{TY14}.

Let $A / K$ be an abelian variety.
Let $\mathcal{U}$ be an \emph{affine} open subscheme of $\mathcal{C}$
where $A$ has good reduction.
Recall from \cite[Equation (6)]{TY14} the following long exact sequence
	\[
			\cdots
		\to
			\coh^{i}_{\mathrm{rig}, c}(\mathcal{U}, D^{\dagger}(A))
		\to
			\coh^{i}_{\mathrm{rig}}(\mathcal{U}, D^{\dagger}(A))
		\to
			\coh^{i}(\mathrm{DR}_{A_{\mathcal{U}}^{\mathrm{loc}}}(D^{\dagger}(A)))
		\to
			\cdots
	\]
of $F$-isocrystals over $\F_{q}$,
where $D^{\dagger}(A)$ is the $F$-overconvergent isocrystal associated with $A$,
$\coh^{i}_{\mathrm{rig}}$ rigid cohomology and
$\coh^{i}_{\mathrm{rig}, c}$ its version with compact support
(see \cite[\S 4.2]{TY14} for the definition of the third term).
We write the ($p$-th power) Frobenius map for any of these $F$-isocrystals by $\varphi_{p}$
and set $\varphi_{q} = \varphi_{p}^{e}$.
As in the paragraph after \cite[Remark 4.2.1]{TY14},
let $\mathcal{H}^{1}_{\Q_{p}}$ be the image of the map
$\coh^{1}_{\mathrm{rig}, c}(\mathcal{U}, D^{\dagger}(A))
\to \coh^{1}_{\mathrm{rig}}(\mathcal{U}, D^{\dagger}(A))$.
All the determinants below are taken over $W(\F_{q})[1 / p]$.

\begin{Aproposition} \label{p:numpadic}
	$P_{1}(t) = \det(1 - \varphi_{q} t \,|\, \mathcal{H}^{1}_{\Q_{p}}(-1))$,
	where $(-1)$ denotes a Tate twist.
\end{Aproposition}

\begin{proof}
	We fix an embedding $\iota \colon W(\F_{q})[1 / p] \hookrightarrow \C$.
	Since $D^{\dagger}(A)$ is $\iota$-pure of weight $-1$ by \cite[Lemma 4.2.3]{TY14},
	we know that $\coh_{\mathrm{rig}}^{0}(\mathcal{U}, D^{\dagger}(A))$ is
	$\iota$-pure of weight $-1$
	by the third paragraph of \cite[\S 6.1]{Ked06}.
	Dualizing, $\coh_{\mathrm{rig}, c}^{2}(\mathcal{U}, D^{\dagger}(A))$ is
	$\iota$-pure of weight $1$.
	By the proof of \cite[Lemma 4.2.4]{TY14},
	$\coh^{0}(\mathrm{DR}_{A_{\mathcal{U}}^{\mathrm{loc}}}(D^{\dagger}(A)))$ is
	$\iota$-mixed of weights $\le -1$.
	Also, $\coh_{\mathrm{rig}, c}^{1}(\mathcal{U}, D^{\dagger}(A))$ is
	$\iota$-mixed of weights $\le 0$ by \cite[6.6.2 (a)]{Ked06}.
	Similarly, $\coh_{\mathrm{rig}}^{1}(\mathcal{U}, D^{\dagger}(A))$ is
	$\iota$-mixed of weights $\ge 0$.
	
	Hence in the exact sequence
		\[
			 	\coh^{0}(\mathrm{DR}_{A_{\mathcal{U}}^{\mathrm{loc}}}(D^{\dagger}(A)))
			 \to
			 	\coh_{\mathrm{rig}, c}^{1}(\mathcal{U}, D^{\dagger}(A))
			 \to
			 	\coh_{\mathrm{rig}}^{1}(\mathcal{U}, D^{\dagger}(A)),
		\]
	the terms are $\iota$-mixed of weights $\le -1$, $\le 0$ and $\ge 0$, respectively.
	Therefore the image of the second map, which is $\mathcal{H}_{\Q_{p}}^{1}$,
	is the maximal quotient of $\coh_{\mathrm{rig}, c}^{1}(\mathcal{U}, D^{\dagger}(A))$
	$\iota$-pure of weight $0$.
	
	Let $\mathcal{Z} = \mathcal{C} \setminus \mathcal{U}$.
	We have $\coh_{\mathrm{rig}, c}^{0}(\mathcal{U}, D^{\dagger}(A)) = 0$
	since $\coh_{\mathrm{rig}, c}^{i + 1}(\mathcal{U}, D^{\dagger}(A))$
	is the $i$-th cohomology of a complex concentrated in degrees $0$ and $1$ by \cite[\S 4.2]{TY14}.
	Therefore the $L$-function $L_{A, \mathcal{Z}}(s)$ can be written as
		\[
			\frac{
				\det \bigl(
					1 - \varphi_{q} t \,|\, \coh_{\mathrm{rig}, c}^{1}(\mathcal{U}, D^{\dagger}(A))(-1)
				\bigr)
			}{
				\det \bigl(
					1 - \varphi_{q} t \,|\, \coh_{\mathrm{rig}, c}^{2}(\mathcal{U}, D^{\dagger}(A))(-1)
				\bigr)
			}
		\]
	by \cite[4.3]{KT03},
	where $t = q^{-s}$.
	As seen above, the reciprocal roots of the numerator are
	Weil $q$-numbers of weights $\le 2$ (via the embedding $\iota$)
	and the reciprocal roots of the denominator are Weil $q$-numbers of weight $3$.
	The same function $L_{A, \mathcal{Z}}(s)$ can also be written as
		\[
				\frac{P_{1}(t)}{P_{0}(t) P_{2}(t)}
			\cdot
				\prod_{v \in \mathcal{Z}} P_{v}(t_{v}),
		\]
	where $t_{v} = t^{\deg(v)}$ and $P_{v}(t_{v})$ is the Euler factor at $v$.
	The reciprocal roots of $P_{i}(t)$ are Weil $q$-numbers of weight $i + 1$ and
	the reciprocal roots of $P_{v}(t_{v})$ are Weil $q_{v}$-numbers of weights $0$ and $1$.
	Comparing the weights of the zeros and poles of these two expressions of $L_{A, \mathcal{Z}}(s)$,
	we know that the maximal quotient of
		$
			\det \bigl(
				1 - \varphi_{q} t \,|\, \coh_{\mathrm{rig}, c}^{1}(\mathcal{U}, D^{\dagger}(A))(-1)
			\bigr)
		$
	with reciprocal roots of weight $2$ is $P_{1}(t)$.
	Hence $\det(1 - \varphi_{q} t \,|\, \mathcal{H}_{\Q_{p}}^{1}(-1))$ is $P_{1}(t)$.
\end{proof}

\begin{Aproposition} \label{p:isoc}
	Let $(M, \varphi_{p})$ be an $F$-isocrystal over $\F_{q}$.
	Then the vertices of the Newton polygon of $\det(1 - \varphi_{q} t \,|\, M)$
	with respect to the $q$-valuation have integer coordinates,
	where $\varphi_{q} = \varphi_{p}^{e}$.
\end{Aproposition}

\begin{proof}
	First assume that $M$ is isoclinic.
	Let $n$ be the dimension of $M$ over $W(\F_{q})[1 / p]$.
	Let $\lambda = s / r$ be the slope of $M$,
	where $s$ and $r$ are coprime integers with $r > 0$.
	Then $M$ admits a $W(\F_{q})$-lattice $M_{0}$ such that
	$\varphi_{p}^{r}(M_{0}) = p^{s} M_{0}$ by
	the Dieudonn\'e-Manin classification over $\bar{\F}_{q}$ and Galois descent.
	Hence $\varphi_{q}^{r}(M_{0}) = q^{s} M_{0}$.
	Therefore the $q$-valuations of the reciprocal roots of $\det(1 - \varphi_{q} t \,|\, M)$ are $\lambda$.
	Hence the vertices of the Newton polygon of $\det(1 - \varphi_{q} t \,|\, M)$
	with respect to the $q$-valuation are $(0, 0)$ and $(n, n \lambda)$.
	The number $n \lambda$ is an integer by the Dieudonn\'e-Manin classification over $\bar{\F}_{q}$.
	
	The general case follows from this calculation and the slope decomposition.
\end{proof}

\begin{Aproposition}
	$\theta \in \Z$.
\end{Aproposition}

\begin{proof}
	By Proposition \ref{p:numpadic},
	$P_{1}(t)$ is the characteristic polynomial of the $q$-th power Frobenius action
	on an $F$-isocrystal over $\F_{q}$.
	Hence the vertices of the Newton polygon of $P_{1}(t)$ with respect to the $q$-valuation
	have integer coordinates
	by Proposition \ref{p:isoc}.
	In particular, $\sum_{\lambda_{j} < 1} \lambda_{j} \in \Z$.
	This implies that $\theta \in \Z$
	by Proposition \ref{p:slope}.
\end{proof}

\begin{Aremark} \mbox{}
	\begin{enumerate}
	\item
		To see how non-trivial this statement $\theta \in \Z$ is,
		consider the case $q = 4$ and
		the polynomial $P'(t) = 1 - 2 t + 16 t^{2}$ instead of $P_{1}(t)$.
		The reciprocal roots of $P'(t)$ are Weil $q$-numbers of weight $2$, just as $P_{1}(t)$.
		The polynomial $P'(t)$ also satisfies the expected functional equation
		$P'(t) = 16 t^2 P'(1 / 16 t)$.
		The $q$-valuations $\{\lambda_{j}'\}$ of the reciprocal roots of $P'(t)$ are $1/2$ and $3/2$,
		which are numbers between $0$ and $2$
		invariant under $\lambda_{j}' \leftrightarrow 2 - \lambda_{j}'$.
		Therefore $P'(t)$ shares most of the properties of $P_{1}(t)$.
		Yet $\sum_{\lambda_{j}' < 1} \lambda_{j}'$ is $1/2$, which is not an integer.
		It is crucial that $P_{1}(t)$ comes from an $F$-isocrystal over $\F_{q}$.
	\item
		In the proof of \cite[4.1]{Ulm19},
		it is stated that since the break points of a Newton polygon have integer coordinates,
		$\sum_{\lambda_{i} < 1} (\lambda_{i} - 1)$ is an integer.
		A possible interpretation of this line is the combination of the above three propositions.
		This integrality is used to prove that
		$\dim \Sha(A)$ in the notation of \cite[4.1]{Ulm19} is an integer.
	\item
		In the case $A$ is a Jacobian,
		the integrality of $\theta$ ($= \theta_{A}$) also follows from Crew-Milne's formula
		\eqref{e:Milne} for surfaces (so $\theta_{\mathcal{S}} \in \Z$)
		and Proposition \ref{p:Lcomp}.
	\item
		In \cite[Theorem 1 b)]{Kah14}, it is stated that
		$P_{1}(t)$ is the reciprocal of the ``$Z$-function'' of a certain Chow motive over $\F_{q}$.
		If we used this statement, then $P_{1}(t)$ would be expressed as
		the characteristic polynomial of the $q$-th power Frobenius on
		the crystalline realization of this Chow motive
		by \cite[Proposition 3.4 and Lemma 5.1]{Kah14}.
		This could be used instead of Proposition \ref{p:numpadic}
		to give another proof of $\theta \in \Z$.
		Unfortunately, the proof of \cite[Theorem 1 b)]{Kah14}
		(especially \cite[Proposition 3.5]{Kah14}) is valid only for the case $A$ is a Jacobian,
		as pointed out in the paragraphs after \cite[Th\'eor\`eme 8.6]{Kah18b}.
	\end{enumerate}
\end{Aremark}

\end{document}